\documentclass[reqno,12pt,letterpaper]{amsart}
\usepackage[proof]{sdmacros}
\title[Fractal uncertainty principle]%
{An introduction to fractal uncertainty principle}
\author{Semyon Dyatlov}
\email{dyatlov@math.berkeley.edu}
\address{Department of Mathematics, University of California, Berkeley, CA 94720}
\address{Department of Mathematics, Massachusetts Institute of Technology,
77 Massachusetts Ave, Cambridge, MA 02139}
\begin{document}

\begin{abstract}
Fractal uncertainty principle states that no function can be localized
in both position and frequency near a fractal set.
This article provides a review of recent developments on the fractal uncertainty
principle and of their applications to quantum chaos,
including lower bounds on mass of eigenfunctions on negatively curved surfaces
and spectral gaps on convex co-compact hyperbolic surfaces.
\end{abstract}

\maketitle

%%%%%%%%%%%%%%%%%%%%%%%%%%%%%%%%%%%%%%%%%%%%%%%%%%%%%%%%%%%%%%%%%%%%%%%%%%%%%%%%
%                                 INTRODUCTION                                 %
%%%%%%%%%%%%%%%%%%%%%%%%%%%%%%%%%%%%%%%%%%%%%%%%%%%%%%%%%%%%%%%%%%%%%%%%%%%%%%%%
\addtocounter{section}{1}
\addcontentsline{toc}{section}{1. Introduction}

A \emph{fractal uncertainty principle} (FUP) is a statement in harmonic analysis
which can be vaguely formulated as follows (see Figure~\ref{f:fups-intro}):
\begin{center}
No function can be localized in both position and frequency\\
close to a fractal set.
\end{center}
%%%%%%%%%%%%%%%%%%%%%%%%%%%%%%%%%%%%%%%%%%%%%%%%%%%%%%%%%%%%%%%%%%%%%%%%%%%%%%%%
\begin{figure}
\includegraphics[width=7cm]{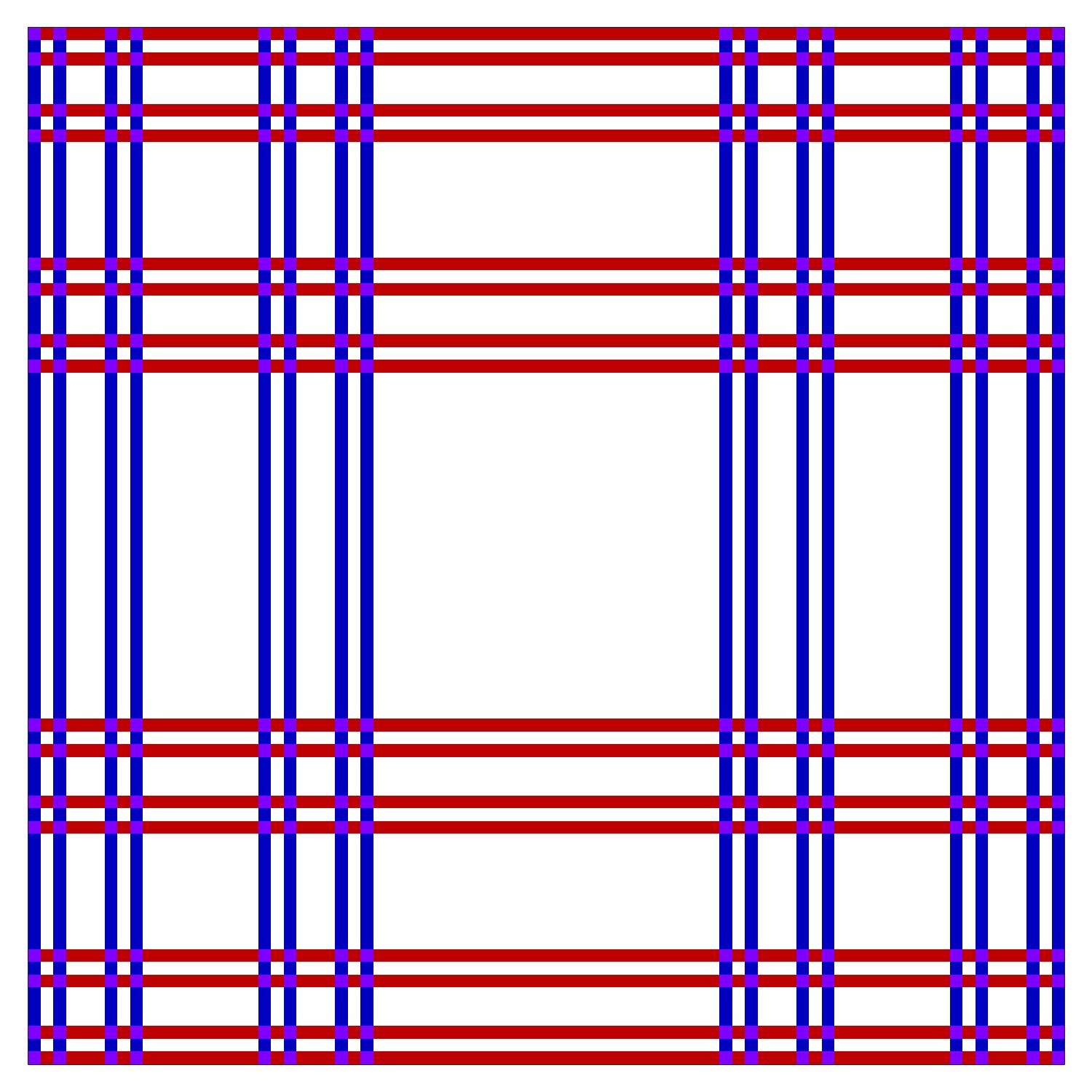}\qquad\quad
\includegraphics[width=7cm]{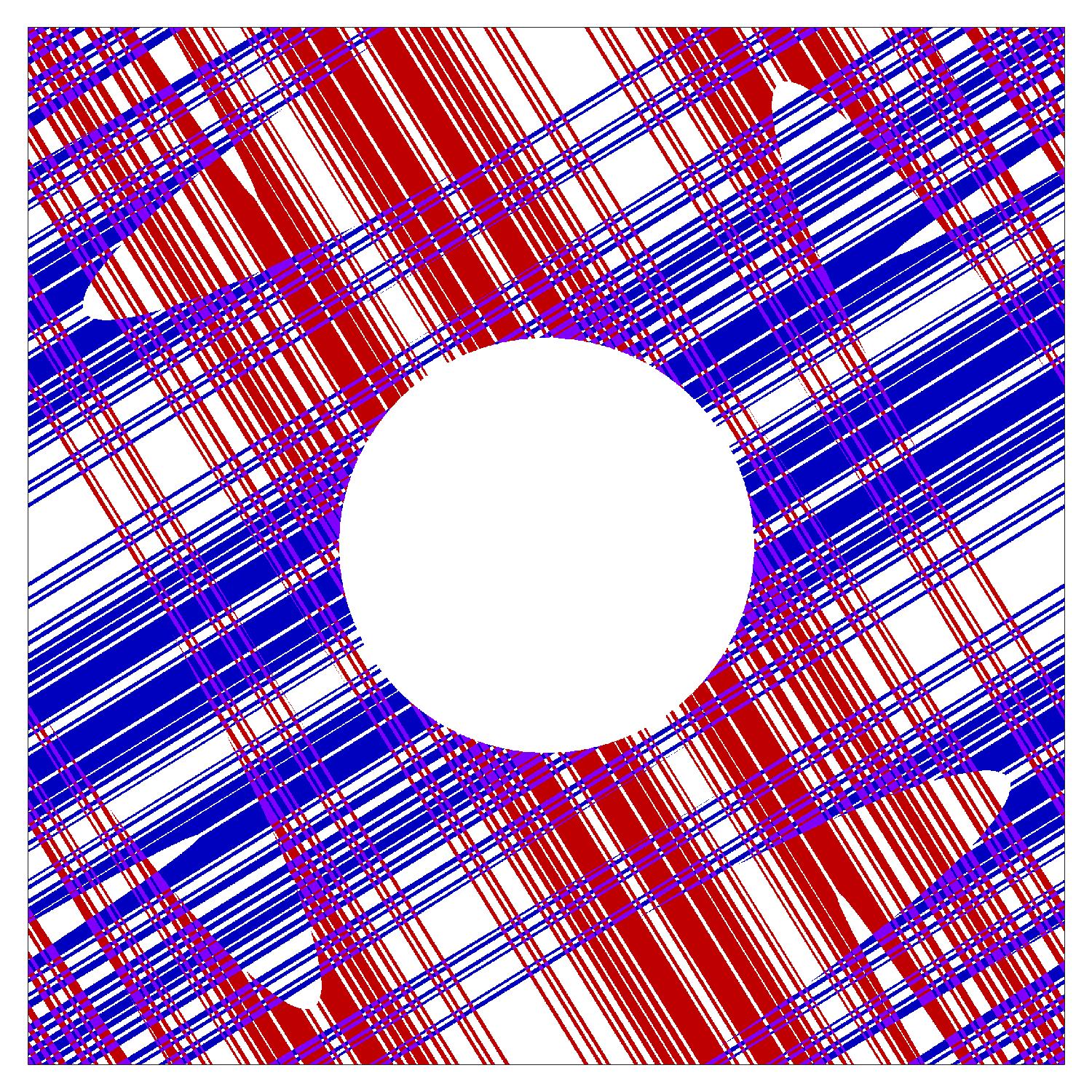}
\caption{Left: a model situation to which FUP applies, with the blue/red sets corresponding
to localization in position/frequency to a neighborhood of the middle third Cantor set.
FUP states that no quantum wavefunction can be localized on both the red and the blue set.
Right: a sample of two fractal sets used in applications of FUP to quantum chaos,
consisting of trajectories of a hyperbolic dynamical system which do not cross some open set
in forward/backward time direction~-- see Figure~\ref{f:porosity} below for details.}
\label{f:fups-intro}
\end{figure}
%%%%%%%%%%%%%%%%%%%%%%%%%%%%%%%%%%%%%%%%%%%%%%%%%%%%%%%%%%%%%%%%%%%%%%%%%%%%%%%%
FUP has been successfully applied to problems in \emph{quantum chaos},
which is the study of quantum systems in situations where the underlying
classical system has chaotic behavior. See the reviews of Marklof~\cite{MarklofReview},
Zelditch~\cite{ZelditchReview}, and
Sarnak~\cite{SarnakQUE} for an overview of this theory for compact systems,
and the reviews of Nonnenmacher~\cite{Nonnenmacher} and Zworski~\cite{ZworskiReview}
for the case of noncompact, or open, systems. Applications of FUP include:
\begin{itemize}
\item lower bounds on mass of eigenfunctions, control for the Schr\"odinger equation,
and exponential decay of damped waves on compact negatively curved surfaces
(see~\S\ref{s:appl-closed});
\item spectral gaps/exponential decay of waves on noncompact hyperbolic surfaces
(see~\S\ref{s:appl-open}).
\end{itemize}
The present article is a broad review of various FUP statements and their applications.
(A previous review article~\cite{fwlEDP} had a more detailed explanation of the proof
of two of the results here, Theorems~\ref{t:fup-porous} and~\ref{t:appl-eig}.)
It is structured as follows:
\begin{itemize}
\item \S\ref{s:fup} gives the main FUP statements (Theorems~\ref{t:fup-0}--\ref{t:hfup-porous}), and briefly discusses the proofs.
It also gives the definitions of fractal sets used throughout the article and
describes Schottky limit sets, which are an important example.
\S\S\ref{s:uncertainty}--\ref{s:fup-statement} are the core of the article,
the later parts of the article are often independent of each other;
\item \S\ref{s:appl} describes applications of FUP to negatively curved surfaces;
\item \S\ref{s:cantor} considers the special class of discrete Cantor sets,
giving a complete proof of FUP in this setting;
\item \S\ref{s:relator} studies the relation of FUP to Fourier decay and additive energy improvements
for fractal sets;
\item \S\ref{s:high-dim} discusses generalizations of FUP to higher dimensions,
which are largely not known at this point.
\end{itemize}
In addition to a review of known results, we state several open problems
(Conjectures~\ref{c:improve-0}, \ref{c:genericdil}, \ref{c:ae-schottky}, \ref{c:fup-curved}, and~\ref{c:hd-fup-disc}) and provide figures with numerical evidence for both the known results and the conjectures. We also provide a more detailed exposition of a few topics:
\begin{itemize}
\item the relation between regular and porous sets (Proposition~\ref{l:regular-porous});
\item reduction of FUP with a general phase to FUP for the Fourier transform
(\S\ref{s:hfup});
\item a proof of FUP for discrete Cantor sets (\S\ref{s:fup-cantor-proof});
\item a proof of a special case of the recent result of Han--Schlag~\cite{HanSchlag}
in the setting of two-dimensional discrete Cantor sets (Proposition~\ref{l:hds-2}).
\end{itemize}

\pagebreak

%%%%%%%%%%%%%%%%%%%%%%%%%%%%%%%%%%%%%%%%%%%%%%%%%%%%%%%%%%%%%%%%%%%%%%%%%%%%%%%%
%%%%%%%%%%%%%%%%%%%%%%%%%%%%%%%%%%%%%%%%%%%%%%%%%%%%%%%%%%%%%%%%%%%%%%%%%%%%%%%%
\section{General results on FUP}
\label{s:fup}

%%%%%%%%%%%%%%%%%%%%%%%%%%%%%%%%%%%%%%%%%%%%%%%%%%%%%%%%%%%%%%%%%%%%%%%%%%%%%%%%
\subsection{Uncertainty principle}
  \label{s:uncertainty}
  
Before going fractal, we briefly review the standard uncertainty principle.
Fix a small parameter $h>0$,
called the \emph{semiclassical parameter},
and consider the unitary semiclassical Fourier transform
\begin{equation}
  \label{e:F-h}
\mathcal F_h:L^2(\mathbb R)\to L^2(\mathbb R),\quad
\mathcal F_h f(\xi)=(2\pi h)^{-1/2}\int_{\mathbb R}e^{-ix\xi/h} f(x)\,dx.
\end{equation}
The version of the uncertainty principle we use is the following:
for any $f\in L^2(\mathbb R)$, either $f$ or its Fourier transform $\mathcal F_h f$
have little mass on the interval $[0,h]$.%
\footnote{This is consistent with the
uncertainty principle in quantum mechanics. Indeed, if both $f$ and $\mathcal F_h f$
are large on $[0,h]$ then we know the wave function $f$ is at position and
momentum 0 with precision~$h$, but $h\cdot h\ll h$.}
Specifically we have
\begin{equation}
  \label{e:basic-fup}
\|\indic_{[0,h]}\mathcal F_h\indic_{[0,h]}\|_{L^2(\mathbb R)\to L^2(\mathbb R)}\leq h^{1/2}.
\end{equation}
Here for $X\subset \mathbb R$, we denote by $\indic_X:L^2(\mathbb R)\to L^2(\mathbb R)$
the multiplication operator by the indicator function of $X$. One way to prove~\eqref{e:basic-fup}
is via H\"older's inequality:
\begin{equation}
  \label{e:basic-fup-proof}
\begin{aligned}
\|\indic_{[0,h]}\mathcal F_h\indic_{[0,h]}\|_{L^2\to L^2}
&\leq
\|\indic_{[0,h]}\|_{L^\infty\to L^2}\cdot \|\mathcal F_h\|_{L^1\to L^\infty}
\cdot \|\indic_{[0,h]}\|_{L^2\to L^1}
\\&= h^{1/2}\cdot (2\pi h)^{-1/2} \cdot h^{1/2}\leq h^{1/2}.
\end{aligned}
\end{equation}
A useful way to think about the norm bound~\eqref{e:basic-fup}
is as follows: if a function $f$ is supported in $[0,h]$,
then the interval $[0,h]$ contains at most $h$ of the $L^2$ mass of $\mathcal F_hf$
(here we use the convention that the mass is the square of the $L^2$ norm).

The fractal uncertainty principle studied below concerns localization
in position and frequency on more general sets:
%%%%%%%%%%%%%%%%%%%%%%%%%%%%%%%%%%%%%%%%%%%%%%%%%%%%%%%%%%%%%%%%%%%%%%%%%%%%%%%%
\begin{defi}
  \label{d:fup}
Let $X,Y\subset\mathbb R$ be $h$-dependent families of sets. We say
that $X,Y$ satisfy \textbf{uncertainty principle} with exponent $\beta\geq 0$,
if
\begin{equation}
  \label{e:fup}
\|\indic_X \mathcal F_h\indic_Y\|_{L^2(\mathbb R)\to L^2(\mathbb R)}=\mathcal O(h^\beta)\quad\text{as }
h\to 0.
\end{equation}
\end{defi}
%%%%%%%%%%%%%%%%%%%%%%%%%%%%%%%%%%%%%%%%%%%%%%%%%%%%%%%%%%%%%%%%%%%%%%%%%%%%%%%%

%%%%%%%%%%%%%%%%%%%%%%%%%%%%%%%%%%%%%%%%%%%%%%%%%%%%%%%%%%%%%%%%%%%%%%%%%%%%%%%%
\subsection{Fractal sets}
  \label{s:fractals}
  
We now give two definitions of a `fractal set' in $\mathbb R$. A more restrictive
definition would be to require self-similarity under a group of transformations
and this is true in some important examples (see~\S\ref{s:schottky} and~\S\ref{s:cantor}). However, here
we use a more general class of sets which have `fractal structure' at every point
and at a range of scales. To introduce those we use \emph{intervals},
which are sets of the form $I=[a,b]\subset\mathbb R$ where $a<b$.
The length of an interval is denoted by $|I|=b-a$.

The first definition we give is that of a regular set
of dimension $\delta$, or \emph{$\delta$-regular set}:
%%%%%%%%%%%%%%%%%%%%%%%%%%%%%%%%%%%%%%%%%%%%%%%%%%%%%%%%%%%%%%%%%%%%%%%%%%%%%%%%
\begin{defi}
  \label{d:delta-regular}
Assume that $X\subset\mathbb R$ is a nonempty closed set and $0\leq \delta\leq 1$,
$C_R\geq 1$, $0\leq\alpha_{\min}\leq\alpha_{\max}\leq \infty$. We say
that $X$ is \textbf{$\delta$-regular with constant $C_R$ on scales
$\alpha_{\min}$ to $\alpha_{\max}$} if
there exists a locally finite measure $\mu_X$ supported on $X$ such that
for every interval $I$ centered at a point in $X$
and such that $\alpha_{\min}\leq |I|\leq \alpha_{\max}$
we have
\begin{equation}
  \label{e:dr-def}
C_R^{-1}|I|^\delta\leq \mu_X(I)\leq C_R|I|^\delta.
\end{equation}
\end{defi}
%%%%%%%%%%%%%%%%%%%%%%%%%%%%%%%%%%%%%%%%%%%%%%%%%%%%%%%%%%%%%%%%%%%%%%%%%%%%%%%%
\begin{rema}
In applications the precise value of $C_R$ is typically irrelevant:
instead we consider a family of $\delta$-regular sets depending
on a parameter $h\to 0$ and it is important that $C_R$ is independent of~$h$.
\end{rema}
%%%%%%%%%%%%%%%%%%%%%%%%%%%%%%%%%%%%%%%%%%%%%%%%%%%%%%%%%%%%%%%%%%%%%%%%%%%%%%%%
\begin{rema}
From~\eqref{e:dr-def} we deduce that $\mu_X(I)\leq 2C_R|I|^\delta$
for any interval $I$ (not necessarily centered on~$X$)
with $\alpha_{\min}\leq |I|\leq \alpha_{\max}$. Indeed,
$I\cap X\subset I_1\cup I_2$ where $I_j$ is the interval of length $|I|$
centered at $x_j$ and $x_1=\min(I\cap X)$, $x_2=\max(I\cap X)$.
\end{rema}
%%%%%%%%%%%%%%%%%%%%%%%%%%%%%%%%%%%%%%%%%%%%%%%%%%%%%%%%%%%%%%%%%%%%%%%%%%%%%%%%
\begin{exam} Here are some basic examples of $\delta$-regular sets
(where $\alpha>0$):
\begin{enumerate}
\item the set $\{0\}$ is $0$-regular on scales $0$ to $\infty$
with constant~1;
\item the set $\mathbb R$ is $1$-regular on scales $0$ to $\infty$ with constant~1;
\item the set $[0,\alpha]$ is $0$-regular on scales $\alpha$ to $\infty$ with constant~2;
\item the set $[0,\alpha]$ is $1$-regular on scales $0$ to $\alpha$ with constant~2;
\item the set $\{0\}\cup [1,2]$ is \textbf{not} $\delta$-regular on scales~0 to~1
for any choice of $\delta,C_R$.
\end{enumerate}
\end{exam}
%%%%%%%%%%%%%%%%%%%%%%%%%%%%%%%%%%%%%%%%%%%%%%%%%%%%%%%%%%%%%%%%%%%%%%%%%%%%%%%%
Examples~(3) and~(4) above demonstrate that the effective dimension of a set
may depend on the scale: the interval $[0,\alpha]$ looks like a point on scales above $\alpha$
and like the entire real line on scales below $\alpha$. Example~(5) shows that not
every set has a dimension in the sense of Definition~\ref{d:delta-regular}.
On the other hand, one can show that no nonempty set can be $\delta$-regular
on scales $0$ to~1 with two different values of~$\delta$; that is, if a dimension in
the sense of Definition~\ref{d:delta-regular} exists, then it is unique.

A more interesting example is given by
%%%%%%%%%%%%%%%%%%%%%%%%%%%%%%%%%%%%%%%%%%%%%%%%%%%%%%%%%%%%%%%%%%%%%%%%%%%%%%%%
\begin{exam}
\label{x:cantor-3} The middle third Cantor set
$$
X:=\bigg\{\sum_{j=1}^\infty a_j 3^{-j}\,\,\bigg|\,\, a_1,a_2,\dots\in \{0,2\}\bigg\}\ \subset\ [0,1]
$$
is $\log_3 2$-regular on scales 0 to~1 with constant~2.
(To show this we can use that
$\mu(I)=2^{-j}$
for any interval $I$ of length $2\cdot 3^{-j}$, $j\in\mathbb N_0$,
centered at a point in~$X$,
where $\mu$ is the Cantor measure.)
\end{exam}
%%%%%%%%%%%%%%%%%%%%%%%%%%%%%%%%%%%%%%%%%%%%%%%%%%%%%%%%%%%%%%%%%%%%%%%%%%%%%%%%

Our second definition of a `fractal set' is more general.
Rather than requiring the same dimension at each point it
asks for the set to have gaps, or pores, and is a quantitative
version of being nowhere dense:
%%%%%%%%%%%%%%%%%%%%%%%%%%%%%%%%%%%%%%%%%%%%%%%%%%%%%%%%%%%%%%%%%%%%%%%%%%%%%%%%
\begin{defi}
  \label{d:porous}
Assume that $X\subset\mathbb R$ is a closed set and
$\nu>0$, $0\leq\alpha_{\min}\leq\alpha_{\max}\leq \infty$.
We say that $X$ is \textbf{$\nu$-porous on scales $\alpha_{\min}$
to $\alpha_{\max}$} if for each interval $I$
such that
$\alpha_{\min}\leq |I|\leq\alpha_{\max}$, there exists
an interval $J\subset I$ such that $|J|=\nu|I|$ and
$X\cap J=\emptyset$.
\end{defi}
%%%%%%%%%%%%%%%%%%%%%%%%%%%%%%%%%%%%%%%%%%%%%%%%%%%%%%%%%%%%%%%%%%%%%%%%%%%%%%%%
\begin{rema}
As with the regularity constant $C_R$,
the precise value of $\nu$ will typically not be of importance.
\end{rema}
%%%%%%%%%%%%%%%%%%%%%%%%%%%%%%%%%%%%%%%%%%%%%%%%%%%%%%%%%%%%%%%%%%%%%%%%%%%%%%%%
\begin{exam}
The middle third Cantor set is $\nu$-porous on scales 0 to~$\infty$
for any $\nu<{1\over 5}$.
\end{exam}
%%%%%%%%%%%%%%%%%%%%%%%%%%%%%%%%%%%%%%%%%%%%%%%%%%%%%%%%%%%%%%%%%%%%%%%%%%%%%%%%
The next proposition establishes a partial equivalence between
the notions of regularity and porosity by showing
that porous sets can be characterized as subsets of $\delta$-regular sets
with $\delta<1$:
%%%%%%%%%%%%%%%%%%%%%%%%%%%%%%%%%%%%%%%%%%%%%%%%%%%%%%%%%%%%%%%%%%%%%%%%%%%%%%%%
\begin{prop}
  \label{l:regular-porous}
Fix $0\leq\alpha_{\min}\leq\alpha_{\max}\leq\infty$.

1. Assume that $X$ is $\delta$-regular with constant
$C_R$ on scales $\alpha_{\min}$ to $\alpha_{\max}$, and $\delta<1$. Then
$X$ is $\nu$-porous on scales $C\alpha_{\min}$ to $\alpha_{\max}$
where $\nu>0$ and $C$ depend only on $\delta,C_R$.

2. Assume that $X$ is $\nu$-porous on scales
$\alpha_{\min}$ to $\alpha_{\max}$. Then $X$ is contained
in some set $Y\subset\mathbb R$ which is $\delta$-regular with constant
$C_R$ on scales $\alpha_{\min}$ to $\alpha_{\max}$ where $\delta<1$
and $C_R$ depend only on $\nu$.
\end{prop}
%%%%%%%%%%%%%%%%%%%%%%%%%%%%%%%%%%%%%%%%%%%%%%%%%%%%%%%%%%%%%%%%%%%%%%%%%%%%%%%%
\begin{proof}
1. Fix $T\in\mathbb N$ to be chosen later depending on $\delta,C_R$ and put
$\nu:=(3T)^{-1}$, $C:=3T$. Let $I$ be an interval
with $C\alpha_{\min}\leq |I|\leq \alpha_{\max}$. We partition $I$
into $T$ intervals $I_1,\dots,I_T$, each of length
$|I|/T$. We argue by contradiction, assuming that each
interval $J\subset I$ with $|J|=\nu|I|$ intersects~$X$.
Then this applies to the middle third of each of the intervals $I_r$,
implying that the interior of $I_r$ contains an interval $I'_r$ of length $|I|/C$
centered at a point in $X$. We now use that
$\bigsqcup_r I_r'\subset I$ and write using $\delta$-regularity of~$X$
$$
TC_R^{-1}\Big({|I|\over 3T}\Big)^\delta\leq \sum_{r=1}^T\mu(I'_r)\leq \mu(I)\leq 2C_R|I|^{\delta}.
$$
Since $\delta<1$ we may choose $T$ so that
$T^{1-\delta}>6C_R^2$, giving a contradiction.

2. We only provide a sketch, referring to~\cite[Lemma~5.4]{meassupp} for details.
Fix $L\in\mathbb N$ such that $L>2/\nu$. Assume for simplicity
that $\alpha_{\min}=0$,
$\alpha_{\max}=1$, and $X$ is contained in an interval $I$ with $|I|= 1$.
We partition $I$ into $L$ intervals $I_1,\dots,I_L$, each of length $|I|/L<\nu/2$.
By the porosity property there exists $\ell_0$ such that $I_{\ell_0}\cap X=\emptyset$.
Then $X$ is contained in the union $Y_1:=\bigcup_{\ell\neq \ell_0}I_\ell$.
We now partition each of the intervals $I_\ell$, $\ell\neq\ell_0$
into $L$ pieces, one of which will again not intersect $X$ by the porosity property
and will be removed. This gives a covering of $X$
by a union $Y_2$ of $(L-1)^2$ intervals, each of length $L^{-2}$.
Repeating the process we construct sets $Y_1\supset Y_2\supset\dots$
covering $X$ and the intersection $Y:=\bigcap_k Y_k$ is a `Cantor-like' set
which is $\delta$-regular with $\delta:=\log_L(L-1)<1$.
\end{proof}
%%%%%%%%%%%%%%%%%%%%%%%%%%%%%%%%%%%%%%%%%%%%%%%%%%%%%%%%%%%%%%%%%%%%%%%%%%%%%%%%
Many natural constructions give sets which are regular/porous on scales~0 to~1.
Neighborhoods of such sets of size $h\ll 1$ (to which FUP will be typically
applied) are then regular/porous
on scales~$Ch$ to~$1$:
%%%%%%%%%%%%%%%%%%%%%%%%%%%%%%%%%%%%%%%%%%%%%%%%%%%%%%%%%%%%%%%%%%%%%%%%%%%%%%%%
\begin{prop}
\label{l:regular-nbhd} Let $0<h<1$.

1. Assume that $X$ is $\delta$-regular on scales~0 to~1 with constant $C_R$.
Then the neighborhood
$$
X(h):=X+[-h,h]
$$
is $\delta$-regular on scales~$h$ to~$1$ with constant $C'_R$,
where $C'_R$ depends only on $C_R$.

2. Assume that $X$ is $\nu$-porous on scales 0 to~1.
Then $X(h)$ is $\nu\over 3$-porous on scales~${3\over\nu}h$ to~1.
\end{prop}
%%%%%%%%%%%%%%%%%%%%%%%%%%%%%%%%%%%%%%%%%%%%%%%%%%%%%%%%%%%%%%%%%%%%%%%%%%%%%%%%
\begin{proof}
1. See~\cite[Lemma~2.3]{fullgap}.

2. Take an interval $I$ such that ${3\over\nu}h\leq |I|\leq 1$.
By the porosity of~$X$ there exists an interval $J\subset I$
with $|J|=\nu|I|\geq 3h$ and $J\cap X=\emptyset$. Let $J'$
be the middle third of $J$, then $J'\subset I$,
$|J'|={\nu\over 3}|I|$, and $J'\cap X(h)=\emptyset$.
\end{proof}
%%%%%%%%%%%%%%%%%%%%%%%%%%%%%%%%%%%%%%%%%%%%%%%%%%%%%%%%%%%%%%%%%%%%%%%%%%%%%%%%

%%%%%%%%%%%%%%%%%%%%%%%%%%%%%%%%%%%%%%%%%%%%%%%%%%%%%%%%%%%%%%%%%%%%%%%%%%%%%%%%
\subsection{Statement of FUP}
  \label{s:fup-statement}

The fractal uncertainty principle gives a partial answer to the following question:
%%%%%%%%%%%%%%%%%%%%%%%%%%%%%%%%%%%%%%%%%%%%%%%%%%%%%%%%%%%%%%%%%%%%%%%%%%%%%%%%
\begin{center}
\smallskip
Fix $\delta\in [0,1]$ and $C_R\geq 1$. What is the largest value of $\beta$
such that~\eqref{e:fup} holds\\
for all $h$-dependent
families of sets $X,Y\subset [0,1]$\\
which are $\delta$-regular with constant $C_R$ on scales $h$ to 1?
\smallskip
\end{center}
%%%%%%%%%%%%%%%%%%%%%%%%%%%%%%%%%%%%%%%%%%%%%%%%%%%%%%%%%%%%%%%%%%%%%%%%%%%%%%%%
One way to establish~\eqref{e:fup} is to use the following volume (Lebesgue measure) bound:
if $X\subset [0,1]$ is $\delta$-regular on scales $h$ to 1 with constant $C_R$, then
for some $C=C(\delta,C_R)$
\begin{equation}
  \label{e:vol-x}
\vol(X)\leq Ch^{1-\delta}.
\end{equation}
See~\cite[Lemma~2.9]{fullgap} for the proof.

Using~\eqref{e:vol-x} and arguing as in~\eqref{e:basic-fup-proof},
we see that~\eqref{e:fup} holds for $\beta={1\over 2}-\delta$:
$$
\|\indic_X\mathcal F_h\indic_Y\|_{L^2\to L^2}
\leq h^{-1/2}\sqrt{\vol(X)\vol(Y)}\leq Ch^{1/2-\delta}.
$$
It also holds for $\beta=0$ since $\mathcal F_h$ is unitary. Therefore,
we get the \emph{basic FUP exponent}
\begin{equation}
  \label{e:easy-fup}
\beta_0=\max\Big(0,{1\over 2}-\delta\Big).
\end{equation}
Note that for $0\leq \delta\leq 1$, we have as $h\to 0$
$$
\|\indic_{[0,h^{1-\delta}]}\mathcal F_h\indic_{[0,h^{1-\delta}]}\|_{L^2\to L^2}
\sim h^{\max(0,1/2-\delta)}
$$
as can be seen by applying the operator on the left-hand side to
a function of the form $\chi(x/h^{1-\delta})$ where $\chi\in \CIc((0,1))$.
This shows that~\eqref{e:easy-fup} cannot be improved if we
only use the volumes of $X,Y$. Instead Theorems~\ref{t:fup-0}--\ref{t:fup-porous} below take advantage of the fractal structure
of $X$ and/or $Y$ at many different points and at different scales.
Also, by taking $X=Y=[0,1]$ or $X=Y=[0,h]$ we see that~\eqref{e:easy-fup}
is sharp when $\delta=0$ or $\delta=1$.

We now present the central result of this article which is a fractal uncertainty principle
improving over~\eqref{e:easy-fup} in the entire range $0<\delta<1$.
Improving over $\beta=0$ and over $\beta={1\over 2}-\delta$ is done using
different methods, so we split the result into two statements:
%%%%%%%%%%%%%%%%%%%%%%%%%%%%%%%%%%%%%%%%%%%%%%%%%%%%%%%%%%%%%%%%%%%%%%%%%%%%%%%%
\begin{theo}\cite[Theorem~4]{fullgap}
  \label{t:fup-0}
Fix $\delta<1$ and $C_R\geq 1$. Then there exists
\begin{equation}
  \label{e:fup-0-beta}
\beta=\beta(\delta,C_R)>0
\end{equation}
such that~\eqref{e:fup} holds for all $h$-dependent families
of sets $X,Y\subset [0,1]$ which are $\delta$-regular with constant
$C_R$ on scales $h$ to $1$.
\end{theo}
%%%%%%%%%%%%%%%%%%%%%%%%%%%%%%%%%%%%%%%%%%%%%%%%%%%%%%%%%%%%%%%%%%%%%%%%%%%%%%%%
%
\begin{theo}\cite[Theorem~1]{regfup}
  \label{t:fup-pres}
Fix $\delta>0$ and $C_R\geq 1$. Then there exists
\begin{equation}
  \label{e:fup-pres-beta}
\beta=\beta(\delta,C_R)>{1\over 2}-\delta
\end{equation}
such that~\eqref{e:fup} holds for all $h$-dependent families
of sets $X,Y\subset [0,1]$ which are $\delta$-regular with constant
$C_R$ on scales $h$ to $1$.
\end{theo}
%%%%%%%%%%%%%%%%%%%%%%%%%%%%%%%%%%%%%%%%%%%%%%%%%%%%%%%%%%%%%%%%%%%%%%%%%%%%%%%%
%
%%%%%%%%%%%%%%%%%%%%%%%%%%%%%%%%%%%%%%%%%%%%%%%%%%%%%%%%%%%%%%%%%%%%%%%%%%%%%%%%
\begin{rema}
There exist estimates on the size of the improvement
$\beta-\max(0,{1\over 2}-\delta)$ in terms of $\delta,C_R$,
see~\cite{regfup} for~\eqref{e:fup-pres-beta}
and Jin--Zhang~\cite{JinZhang} for~\eqref{e:fup-0-beta}.
\end{rema}
%%%%%%%%%%%%%%%%%%%%%%%%%%%%%%%%%%%%%%%%%%%%%%%%%%%%%%%%%%%%%%%%%%%%%%%%%%%%%%%%
%
%%%%%%%%%%%%%%%%%%%%%%%%%%%%%%%%%%%%%%%%%%%%%%%%%%%%%%%%%%%%%%%%%%%%%%%%%%%%%%%%
\begin{rema}
More results on FUP in special settings (Cantor sets, Schottky limit sets)
are described below in~\S\ref{s:cantor} and~\S\ref{s:fdec} and higher dimensions are discussed in~\S\ref{s:high-dim}.
\end{rema}
%%%%%%%%%%%%%%%%%%%%%%%%%%%%%%%%%%%%%%%%%%%%%%%%%%%%%%%%%%%%%%%%%%%%%%%%%%%%%%%%
An application of Theorem~\ref{t:fup-0} and Proposition~\ref{l:regular-porous}
is the following version of FUP for porous sets:
%%%%%%%%%%%%%%%%%%%%%%%%%%%%%%%%%%%%%%%%%%%%%%%%%%%%%%%%%%%%%%%%%%%%%%%%%%%%%%%%
\begin{theo}
  \label{t:fup-porous}
Fix $\nu>0$. Then there exists $\beta=\beta(\nu)>0$ such that~\eqref{e:fup}
holds for all $h$-dependent families of sets $X,Y\subset [0,1]$ which are
$\nu$-porous on scales~$h$ to~$1$.
\end{theo}
%%%%%%%%%%%%%%%%%%%%%%%%%%%%%%%%%%%%%%%%%%%%%%%%%%%%%%%%%%%%%%%%%%%%%%%%%%%%%%%%
We now briefly discuss the proofs of Theorem~\ref{t:fup-0}--\ref{t:fup-pres}.
See also~\cite[\S4]{fwlEDP} for a more detailed expository treatment of Theorem~\ref{t:fup-0}.

To prove Theorem~\ref{t:fup-0} we rewrite the estimate~\eqref{e:fup} as follows:
\begin{equation}
  \label{e:fup-rewritten}
f\in L^2(\mathbb R),\
\supp\hat f\subset h^{-1}\cdot Y
\quad\Longrightarrow\quad
\|\indic_{X}f\|_{L^2}\leq Ch^\beta \|f\|_{L^2}.
\end{equation}
The proof of~\eqref{e:fup-rewritten} proceeds by iteration on scales
$1,{1\over 2},\dots, 2^{-j},\dots, h$. At each scale~$\alpha$
we use the porosity of $X$ to find many intervals of size $\sim\alpha$
which do not intersect~$X$; denote their union by~$U_\alpha$.
The upper bound~\eqref{e:fup-rewritten}
then follows from a \emph{lower} bound on the mass of $f$ on $U_\alpha$. Such lower bounds
are known if $f$ belongs to a quasi-analytic class, i.e. if the Fourier transform $\hat f$
decays fast enough. (For instance if $\hat f$ decays exponentially fast then $f$ is real-analytic
and cannot identically vanish on any interval.)

To convert the Fourier support condition~\eqref{e:fup-rewritten}
to a Fourier decay statement we convolve $f$ with a function $\psi$ which is compactly
supported (so that we do not lose the ability to bound the norm of $f$ on an interval) and
has Fourier transform decaying almost exponentially fast on the set $h^{-1}\cdot Y$.
The function $\psi$ is constructed using the Beurling--Malliavin theorem~\cite{Beurling-Malliavin}
with a weight tailored to the set $Y$, whose existence uses the fact that $Y$ is $\delta$-regular with $\delta<1$.
(One does not actually need the full strength of the Beurling--Malliavin theorem as explained in~\cite{JinZhang}).
In particular, we use the fact that
$X,Y$ are `not too big' in two different ways: the porosity property
and a quantitative sparsity following from~\eqref{e:dr-def} when $\delta<1$.
In the much simpler setting of arithmetic Cantor sets
these two properties appear in the proof of Lemma~\ref{l:improve-0}.

The proof of Theorem~\ref{t:fup-pres} is inspired by the works of Dolgopyat~\cite{Dolgopyat} and Naud~\cite{NaudGap}. Note that if we replace $e^{-ix\xi/h}$ in the definition of the Fourier transform $\mathcal F_h$ by~1,
then the norm~\eqref{e:fup} is asymptotic to~$h^{1/2-\delta}$
(assuming $\vol(X)\sim\vol(Y)\sim h^{1-\delta}$).
Thus to get an improvement we need to use cancellations coming from the phase
in the Fourier transform. Using that $\delta>0$
(that is, $X,Y$ are `not too small') we can find many quadruples of points
$x_1,x_2\in X$, $y_1,y_2\in Y$ such that the phase factor
$e^{i(x_1-x_2)(y_1-y_2)/h}$ is far from~1. These quadruples cause cancellations
which lead to~\eqref{e:fup} with $\beta>{1\over 2}-\delta$.
In general one has to be careful at exploiting the cancellations to make
sure they compound on many different scales; the argument is again much simpler
in the setting of arithmetic Cantor sets, see Lemma~\ref{l:improve-1}.

%%%%%%%%%%%%%%%%%%%%%%%%%%%%%%%%%%%%%%%%%%%%%%%%%%%%%%%%%%%%%%%%%%%%%%%%%%%%%%%%
\subsection{Schottky limit sets}
  \label{s:schottky}

Many natural fractal sets are constructed using iterated function systems.
Here we briefly present a special class of these, \emph{Schottky limit sets},
which naturally arise in the spectral gap problem on hyperbolic surfaces
(see~\S\ref{s:appl-open} below). We refer to~\cite[\S2]{hyperfup} and~\cite[\S15.1]{BorthwickBook}
for more details.

Schottky limit sets are generated by fractional linear (M\"obius) transformations
\begin{equation}
  \label{e:mobius}
x\mapsto \gamma.x={ax+b\over cx+d},\quad
\gamma=\begin{pmatrix} a & b \\ c & d\end{pmatrix}\in\SL_2(\mathbb R),\quad
x\in\dot{\mathbb R}:=\mathbb R\cup\{\infty\}.
\end{equation}
More precisely, we:
%%%%%%%%%%%%%%%%%%%%%%%%%%%%%%%%%%%%%%%%%%%%%%%%%%%%%%%%%%%%%%%%%%%%%%%%%%%%%%%%
\newsavebox{\matone}
\newsavebox{\mattwo}
\savebox{\matone}{$\gamma_1=\left(\begin{smallmatrix}7 &-3\\ -2 &1\end{smallmatrix}\right)$}
\savebox{\mattwo}{$\gamma_2=\left(\begin{smallmatrix}3 &-7\\ -2 & 5\end{smallmatrix}\right)$}
%%%%%%%%%%%%%%%%%%%%%%%%%%%%%%%%%%%%%%%%%%%%%%%%%%%%%%%%%%%%%%%%%%%%%%%%%%%%%%%%
\begin{figure}
\includegraphics[width=14cm]{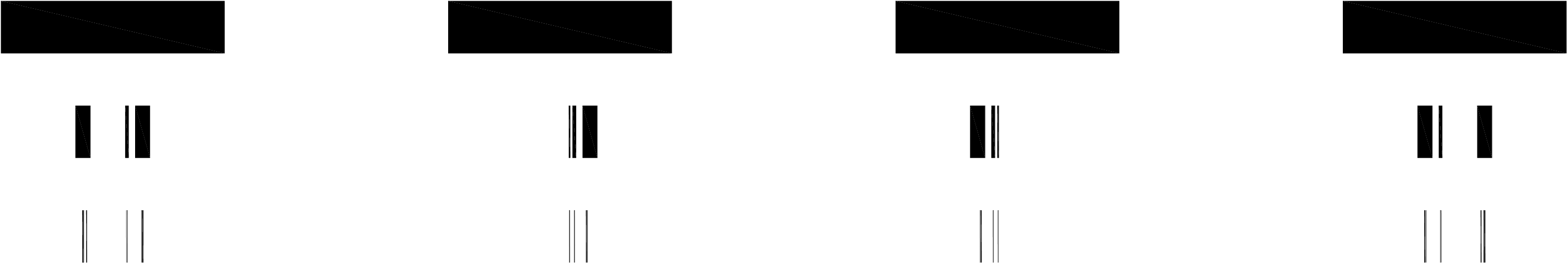}
\caption{An example of a Schottky limit set, picturing the sets in~\eqref{e:schottky-X}
for $n=1,2,3$. The Schottky data are as follows: $I_1=[-4,-3]$,
$I_2=[-2,-1]$, $I_3=[0,1]$, $I_4=[2,3]$,
\usebox{\matone}, \usebox{\mattwo}. The bottom row is a union of 36 intervals,
illustrating the difficulty of plotting a `nonlinear' fractal set.}
\label{f:sch1}
\end{figure}
%%%%%%%%%%%%%%%%%%%%%%%%%%%%%%%%%%%%%%%%%%%%%%%%%%%%%%%%%%%%%%%%%%%%%%%%%%%%%%%%
%
%%%%%%%%%%%%%%%%%%%%%%%%%%%%%%%%%%%%%%%%%%%%%%%%%%%%%%%%%%%%%%%%%%%%%%%%%%%%%%%%
\begin{itemize}
\item fix a collection of nonintersecting intervals $I_1,\dots,I_{2r}\subset\mathbb R$,
where $r\geq 1$;
\item denote $\mathcal A:=\{1,\dots,2r\}$ and for each
$w\in \mathcal A$, define $\overline w:=w+r$ if $w\leq r$ and $\overline w:=w-r$ otherwise;
\item fix transformations $\gamma_1,\dots,\gamma_{2r}\in\SL_2(\mathbb R)$
such that for all $w\in\mathcal A$
we have $\gamma_{\overline w}=\gamma_w^{-1}$ and
the image of $\dot{\mathbb R}\setminus I_{\overline w}$ under $\gamma_w$ is the interior
of $I_w$;
\item for $n\in\mathbb N$, define the set $\mathcal W_n$
consisting of words $\mathbf w=w_1\dots w_n$ such that
$w_{j+1}\neq \overline{w_j}$ for all $j=1,\dots,n-1$;
\item for each word $\mathbf w=w_1\dots w_n\in\mathcal W_n$ define the interval
$I_{\mathbf w}=\gamma_{w_1}\dots\gamma_{w_{n-1}}(I_{w_n})$.
Since $I_{w_n}\subset\dot{\mathbb R}\setminus I_{\overline{w_{n-1}}}$
we have $I_{\mathbf w}\subset I_{w_1\dots w_{n-1}}$,
so the collection of intervals $I_{\mathbf w}$ forms a tree
(see Figure~\ref{f:sch1});
\item the limit set $X$ is now defined as the intersection
of a decreasing family of sets
\begin{equation}
  \label{e:schottky-X}
X:=\bigcap_n\bigsqcup_{\mathbf w\in\mathcal W_n} I_{\mathbf w}.
\end{equation}
\end{itemize}
%%%%%%%%%%%%%%%%%%%%%%%%%%%%%%%%%%%%%%%%%%%%%%%%%%%%%%%%%%%%%%%%%%%%%%%%%%%%%%%%
The transformations $\gamma_1,\dots,\gamma_r$ generate a discrete free subgroup
$\Gamma\subset \SL_2(\mathbb R)$, called a \emph{Schottky group}, and
$\Gamma$ acts on the limit set $X$.
For $r=1$ the set $X$ consists of just two points, so we henceforth assume
$r\geq 2$. In this case $X$ is $\delta$-regular for some $\delta\in (0,1)$,
see~\cite[Lemma~2.12]{hyperfup}. The corresponding measure
in Definition~\ref{d:delta-regular} is the Patterson--Sullivan
measure on $X$, see~\cite[\S14.1]{BorthwickBook}.

Schottky limit sets give a fundamental example of `nonlinear' fractal sets
in the sense that the transformations generating them are nonlinear (as opposed
to linear Cantor sets such as those studied in~\S\ref{s:cantor} below).
This often complicates their analysis, however this nonlinearity is sometimes
also useful~-- in particular it implies Fourier decay for Schottky limit sets
while linear Cantor sets do not have this property, see Theorem~\ref{t:fdec}.

%%%%%%%%%%%%%%%%%%%%%%%%%%%%%%%%%%%%%%%%%%%%%%%%%%%%%%%%%%%%%%%%%%%%%%%%%%%%%%%%
\subsection{FUP with a general phase}
  \label{s:hfup}

In applications we often need a more general version of FUP, with the Fourier
transform~\eqref{e:F-h} replaced by an oscillatory integral operator
\begin{equation}
  \label{e:B-h}
\mathcal B_h f(x)=(2\pi h)^{-1/2}\int_{\mathbb R}e^{i\Phi(x,y)/h}
b(x,y)f(y)\,dy.
\end{equation}
Here the phase function $\Phi\in C^\infty(U;\mathbb R)$
satisfies the nondegeneracy condition
\begin{equation}
  \label{e:Phi-nondeg}
\partial^2_{xy}\Phi(x,y)\neq 0\quad\text{on}\quad U,
\end{equation}
$U\subset\mathbb R^2$ is an open set, and the amplitude $b$ lies in~$\CIc(U)$. The nondegeneracy
condition ensures that the norm $\|\mathcal B_h\|_{L^2\to L^2}$ is bounded uniformly as
$h\to 0$. The phase function used in applications
to hyperbolic surfaces in~\S\ref{s:appl} is
\begin{equation}
  \label{e:Phi-hyp}
\Phi(x,y)=\log|x-y|,\quad
U=\{(x,y)\in\mathbb R^2\mid x\neq y\},
\end{equation}
and FUP with this phase function is called the \emph{hyperbolic FUP}.

The following results generalize Theorems~\ref{t:fup-0}--\ref{t:fup-porous}.
In all of these we assume that $\Phi$ satisfies~\eqref{e:Phi-nondeg};
the constant $\beta$ depends only on $\delta,C_R$ (or $\nu$ in the case
of Theorem~\ref{t:hfup-porous}) and the constant $C$ additionally
depends on $\Phi,b$. However,
the values of $\beta$ obtained in Theorems~\ref{t:hfup-0},\ref{t:hfup-porous} below
are smaller than the ones in Theorems~\ref{t:fup-0},\ref{t:fup-porous}.
Since $b$ is compactly supported we may remove
the condition $X,Y\subset [0,1]$.
%%%%%%%%%%%%%%%%%%%%%%%%%%%%%%%%%%%%%%%%%%%%%%%%%%%%%%%%%%%%%%%%%%%%%%%%%%%%%%%%
\begin{theo}\cite[Proposition~4.3]{fullgap}
  \label{t:hfup-0}
Fix $\delta<1$ and $C_R\geq 1$. Then there exists
\begin{equation}
  \label{e:hfup-0-beta}
\beta=\beta(\delta,C_R)>0
\end{equation}
such that for all $X,Y\subset \mathbb R$ which are $\delta$-regular with constant
$C_R$ on scales $h$ to $1$ we have
\begin{equation}
  \label{e:hfup}
\|\indic_X \mathcal B_h\indic_Y\|_{L^2(\mathbb R)\to L^2(\mathbb R)}\leq Ch^\beta.
\end{equation}
\end{theo}
%%%%%%%%%%%%%%%%%%%%%%%%%%%%%%%%%%%%%%%%%%%%%%%%%%%%%%%%%%%%%%%%%%%%%%%%%%%%%%%%
%
%%%%%%%%%%%%%%%%%%%%%%%%%%%%%%%%%%%%%%%%%%%%%%%%%%%%%%%%%%%%%%%%%%%%%%%%%%%%%%%%
\begin{theo}\cite{regfup}
  \label{t:hfup-pres}
Fix $\delta>0$ and $C_R\geq 1$. Then there exists
\begin{equation}
  \label{e:hfup-pres-beta}
\beta=\beta(\delta,C_R)>{1\over 2}-\delta
\end{equation}
such that~\eqref{e:hfup} holds for all $X,Y\subset \mathbb R$ which are $\delta$-regular with constant $C_R$ on scales $h$ to $1$.
\end{theo}
%%%%%%%%%%%%%%%%%%%%%%%%%%%%%%%%%%%%%%%%%%%%%%%%%%%%%%%%%%%%%%%%%%%%%%%%%%%%%%%%
%
%%%%%%%%%%%%%%%%%%%%%%%%%%%%%%%%%%%%%%%%%%%%%%%%%%%%%%%%%%%%%%%%%%%%%%%%%%%%%%%%
\begin{theo}
\label{t:hfup-porous}
Fix $\nu>0$.
Then there exists $\beta=\beta(\nu)>0$ such that~\eqref{e:hfup} holds for
all $h$-dependent families of sets $X,Y\subset\mathbb R$ which are $\nu$-porous
on scales $h$ to $1$.
\end{theo}
%%%%%%%%%%%%%%%%%%%%%%%%%%%%%%%%%%%%%%%%%%%%%%%%%%%%%%%%%%%%%%%%%%%%%%%%%%%%%%%%
We give an informal explanation for how to reduce Theorem~\ref{t:hfup-0}
to the case of Fourier transform, Theorem~\ref{t:fup-0}.
(For Theorem~\ref{t:hfup-pres}, the argument in~\cite{regfup}
handles the case of a general phase directly.) The argument we give below gives
$\beta$ which depends on $\Phi$ in addition to $\delta,C_R$,
but it can be modified to remove this dependence. 

We first consider the case of the Fourier transform phase $\Phi(x,y)=-xy$
and arbitrary amplitude $b\in\CIc(\mathbb R^2)$. 
Fix $\rho\in (0,1)$ which is very close to~1. For each
$f\in L^2(\mathbb R)$, the function $\mathcal B_h \indic_Y f$
is localized to semiclassical frequencies in the neighborhood
$Y(h^\rho)$ in the following sense: for all $N$
\begin{equation}
  \label{e:localized}
\|\indic_{\mathbb R\setminus Y(h^\rho)}\mathcal F_h^{-1}\mathcal B_h \indic_Y f\|_{L^2(\mathbb R)}\leq C_Nh^N\|f\|_{L^2(\mathbb R)}.
\end{equation}
Indeed, the operator $\mathcal F_h^{-1}\mathcal B_h$ has integral kernel
$$
K(x,y)=(2\pi h)^{-1}\int_{\mathbb R} e^{i(x-y)w/h}b(w,y)\,dw.
$$
Repeated integration by parts shows
that $|K(x,y)|\leq C_N h^{-1}(1+|x-y|/h)^{-N}$
for all~$N$, which implies~\eqref{e:localized}.

Armed with~\eqref{e:localized}, we may replace
$u:=\mathcal B_h \indic_Y f$ by~$\mathcal F_h\indic_{Y(h^\rho)}\mathcal F_h^{-1}u$ in~\eqref{e:hfup},
which then reduces to
\begin{equation}
  \label{e:hfup-0}
\|\indic_X\mathcal F_h\indic_{Y(h^\rho)}\|_{L^2(\mathbb R)\to L^2(\mathbb R)}\leq Ch^{\beta}.
\end{equation}
Now~\eqref{e:hfup-0} with $\rho:=1$ follows from FUP for Fourier transform,
Theorem~\ref{t:fup-0}, since $Y(h)$ is still $\delta$-regular on scales $Ch$ to~1 similarly to
Proposition~\ref{l:regular-nbhd}. For $\rho<1$ we may write $Y(h^\rho)$
as a union of $\sim h^{\rho-1}$ shifted copies of the set~$Y(h)$, which
bounds the left-hand side of~\eqref{e:hfup-0} by $Ch^{\beta+\rho-1}$.
It remains to take $\rho$ close enough to~1 so that $\beta+\rho-1>0$.

\medskip

We now explain how to handle the case of a general phase $\Phi$ satisfying~\eqref{e:Phi-nondeg},
using almost orthogonality and a linearization argument. We first take $\rho<1$
close to~1 and replace the set $X$ by a smoothened version of
its neighborhood $X(h^{\rho/2})$ in~\eqref{e:hfup}. More precisely, take a function
$$
\chi\in C^\infty(\mathbb R;[0,1]),\quad
\supp\chi\subset X(h^{\rho/2}),\quad
\supp(1-\chi)\cap X=\emptyset
$$
which satisfies the derivative bounds for all~$N$
\begin{equation}
  \label{e:hfup-derbs-chi}
\sup|\partial^N \chi|\leq C_Nh^{-\rho N/2}.
\end{equation}
Then it suffices to show the bound
\begin{equation}
  \label{e:hfup-1}
\|\chi\mathcal B_h\indic_Y\|_{L^2(\mathbb R)\to L^2(\mathbb R)}\leq Ch^\beta.
\end{equation}
The bound~\eqref{e:hfup-1} is stronger than~\eqref{e:hfup} and thus appears harder to prove.
However, if the sets $Y_1,Y_2\subset\mathbb R$ are at distance $\geq h^{1/2}$, then
we have the almost orthogonality estimate for all~$N$
\begin{equation}
  \label{e:hfup-ao}
\|(\chi\mathcal B_h\indic_{Y_1})^*\chi\mathcal B_h\indic_{Y_2}\|_{L^2(\mathbb R)\to L^2(\mathbb R)}
\leq C_Nh^N.
\end{equation}
To show~\eqref{e:hfup-ao} we write the integral kernel of the operator
$\mathcal B_h^* \chi^2\mathcal B_h$,
$$
K_1(x,y)=(2\pi h)^{-1}\int_{\mathbb R} e^{i(\Phi(w,y)-\Phi(w,x))/h}
\overline{b(w,x)}b(w,y)\chi(w)^2\,dw.
$$
We repeatedly integrate by parts in~$w$.
Each integration by parts gives a gain of $h/|x-y|$ from the phase,
using the inequality
$$
|\partial_w(\Phi(w,x)-\Phi(w,y))|\geq C^{-1}|x-y|
$$
which is a consequence of~\eqref{e:Phi-nondeg}. On the other hand
differentiating the $\chi^2(w)$ factor we get a $h^{\rho/2}$ loss by~\eqref{e:hfup-derbs-chi}.
For $|x-y|\geq h^{1/2}$ we get an $h^{(1-\rho)/2}$ improvement with each integration
by parts, giving~\eqref{e:hfup-ao}.

Now, we split $Y$ into a disjoint union of clusters $Y_j$, each contained in an interval
of size $h^{1/2}$. Using~\eqref{e:hfup-ao} and the Cotlar--Stein Theorem~\cite[Theorem~C.5]{e-z}
we see that it suffices to show the norm bound for each individual cluster
\begin{equation}
\label{e:hfup-2}
\|\chi\mathcal B_h\indic_{Y_j}\|_{L^2(\mathbb R)\to L^2(\mathbb R)}\leq Ch^\beta.
\end{equation}
For simplicity we assume that $Y_j=Y\cap [0,h^{1/2}]$.
Composing $\mathcal B_h$ with the isometry $Tf(y)=h^{-1/4}f(y/h^{1/2})$ we get
the operator
$$
\widetilde{\mathcal B}_hf(x)=(2\pi\sqrt h)^{-1/2}\int_{\mathbb R}
e^{i\Phi(x,\sqrt h y)/h}b(x,\sqrt hy)f(y)\,dy
$$
and it suffices to show the norm bound
\begin{equation}
  \label{e:hfup-3}
\|\chi\widetilde{\mathcal B}_h\indic_{h^{-1/2}Y\cap [0,1]}\|_{L^2(\mathbb R)\to L^2(\mathbb R)}
\leq Ch^\beta.
\end{equation}
For $y\in [0,1]$ we write the Taylor expansion of the phase in $\widetilde{\mathcal B}_h$,
$$
{\Phi(x,\sqrt h y)\over h}={\Phi(x,0)\over h}+{\varphi(x)y\over\sqrt h}+\mathcal O(1),\quad
\varphi(x):=\partial_y\Phi(x,0).
$$
The first term on the right-hand side can be pulled out of the operator $\mathcal B_h$
without changing the norm~\eqref{e:hfup-3} and the $\mathcal O(1)$ remainder can be put into
the amplitude $b$. Thus~\eqref{e:hfup-3} follows from the bound
\begin{equation}
  \label{e:hfup-4}
\|\indic_{X(h^{\rho/2})}\mathcal B'_h\indic_{h^{-1/2}Y\cap [0,1]}\|_{L^2(\mathbb R)\to L^2(\mathbb R)}\leq Ch^\beta
\end{equation}
where we have for some amplitude $b'(x,y)$ with bounded derivatives
$$
\mathcal B'_h f(x)=(2\pi\sqrt h)^{-1/2}\int_{\mathbb R}e^{i\varphi(x)y/\sqrt h}b'(x,y)f(y)\,dy.
$$
Making the change of variables $x\mapsto -\varphi(x)$ (which is a diffeomorphism thanks to the nondegeneracy condition~\eqref{e:Phi-nondeg}) we reduce to an uncertainty estimate of the
form~\eqref{e:hfup} with the phase $-xy$, $h$ replaced by $\sqrt h$,
and the sets $X,Y$ replaced by $-\varphi(X(h^{\rho/2})),h^{-1/2}Y\cap [0,1]$.
Using that $\varphi(X(h^{1/2}))$ and $h^{-1/2}Y$ are $\delta$-regular on scales
$\sqrt h$ to~1 and taking $\rho$ close to~1, we finally get the bound~\eqref{e:hfup-4}
from the case of the phase $-xy$ handled above.

%%%%%%%%%%%%%%%%%%%%%%%%%%%%%%%%%%%%%%%%%%%%%%%%%%%%%%%%%%%%%%%%%%%%%%%%%%%%%%%%
%%%%%%%%%%%%%%%%%%%%%%%%%%%%%%%%%%%%%%%%%%%%%%%%%%%%%%%%%%%%%%%%%%%%%%%%%%%%%%%%
\section{Applications of FUP}
  \label{s:appl}

We now discuss applications of the fractal uncertainty principle
to quantum chaos, more precisely to lower bounds on mass of eigenfunctions (\S\ref{s:appl-closed})
and essential spectral gaps (\S\ref{s:appl-open}). The present review focuses
on the fractal uncertainty principle itself rather than on its applications,
thus we keep the discussion brief. In particular, we largely avoid discussing
\emph{microlocal analysis}, a mathematical theory behind classical/quantum and particle/wave
correspondences in physics which is essential in
obtaining applications of FUP. A more detailed presentation of the application
to eigenfunctions in~\S\ref{s:appl-closed} in the special
case of hyperbolic surfaces is available in~\cite{fwlEDP}. 

%%%%%%%%%%%%%%%%%%%%%%%%%%%%%%%%%%%%%%%%%%%%%%%%%%%%%%%%%%%%%%%%%%%%%%%%%%%%%%%%
\subsection{Control of eigenfunctions}
  \label{s:appl-closed}

Throughout this section we assume that $(M,g)$ is a compact connected Riemannian surface
of negative Gauss curvature.%
\footnote{The results in this section apply in fact to more general \emph{Anosov surfaces},
where the geodesic flow has a stable/unstable decomposition.}
An important special class is given by \emph{hyperbolic surfaces},
which have Gauss curvature~$-1$.
A standard object of study in quantum chaos is the collection of eigenfunctions of the
Laplace--Beltrami operator $-\Delta_g$ on $M$,
$$
u_k\in C^\infty(M),\quad
(-\Delta_g-\lambda_k^2)u_k=0,\quad
\|u_k\|_{L^2(M)}=1,
$$
where $u_k$ forms an orthonormal basis of $L^2(M)$. See Figure~\ref{f:stro}.
Our first application is a lower bound on mass of these eigenfunctions:
%%%%%%%%%%%%%%%%%%%%%%%%%%%%%%%%%%%%%%%%%%%%%%%%%%%%%%%%%%%%%%%%%%%%%%%%%%%%%%%%
\begin{figure}
\includegraphics[scale=1.1]{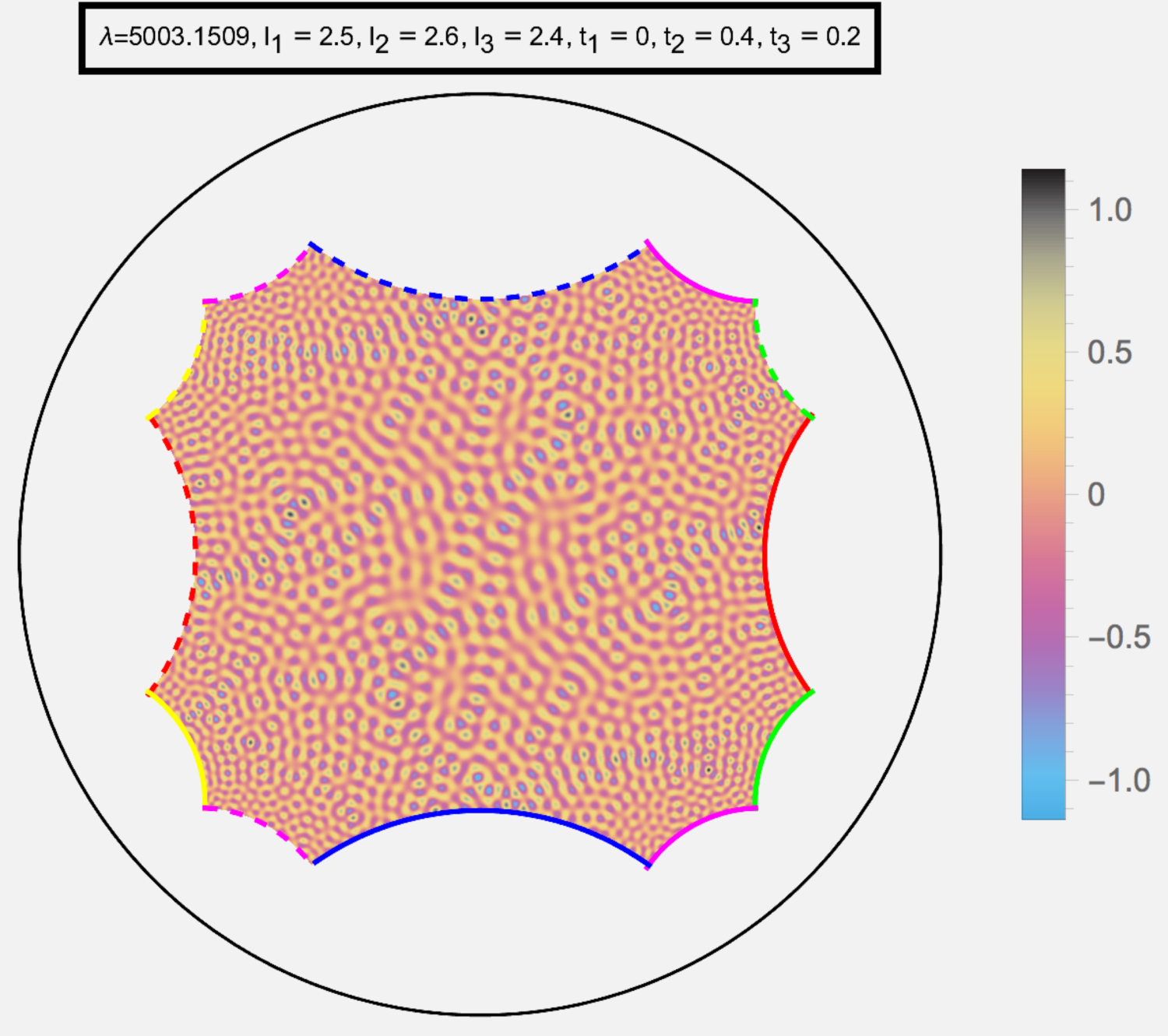}
\qquad
\includegraphics[scale=1.1]{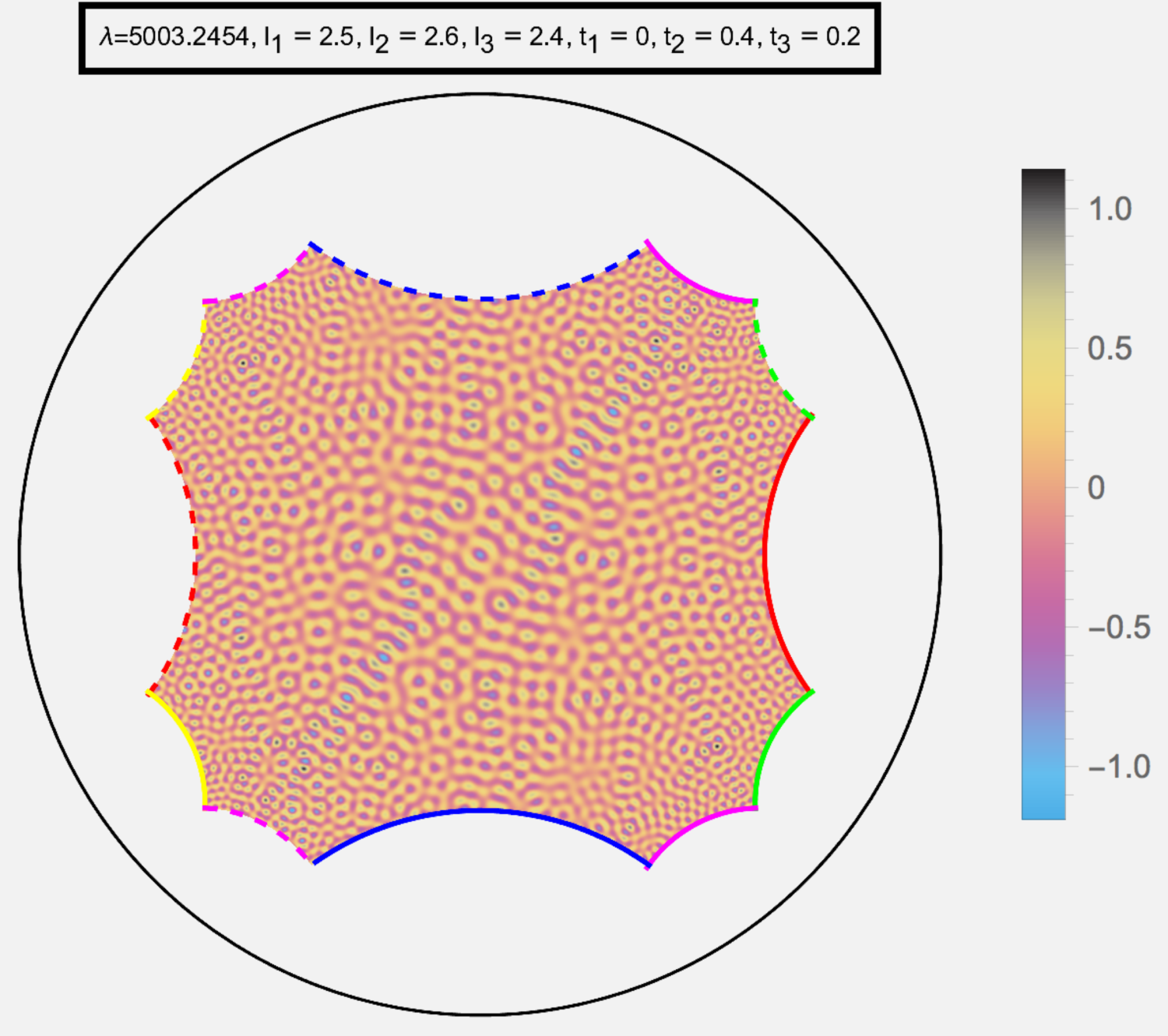}
\caption{Heat plots of two high frequency eigenfunctions of
the Laplacian on a genus~2 hyperbolic surface (on a fundamental domain
inside the Poincar\'e disk model of the hyperbolic plane), showing equidistribution
consistent with the Quantum Unique Ergodicity conjecture.
Pictures courtesy of Alexander Strohmaier,
produced using the method of Strohmaier--Uski~\cite{StrohmaierUski}.}
\label{f:stro}
\end{figure}
%%%%%%%%%%%%%%%%%%%%%%%%%%%%%%%%%%%%%%%%%%%%%%%%%%%%%%%%%%%%%%%%%%%%%%%%%%%%%%%%
\begin{theo}\cite{meassupp,varfup}
\label{t:appl-eig}
Fix a nonempty open set $\Omega\subset M$.
Then there exists $c_\Omega>0$ such that for all $k$
we have the lower bound
\begin{equation}
  \label{e:appl-eig}
\|\indic_\Omega u_k\|_{L^2(M)}\geq c_\Omega.
\end{equation}
\end{theo}
%%%%%%%%%%%%%%%%%%%%%%%%%%%%%%%%%%%%%%%%%%%%%%%%%%%%%%%%%%%%%%%%%%%%%%%%%%%%%%%%
We note that the paper~\cite{meassupp} handled the special case of hyperbolic
surfaces and the later paper~\cite{varfup}, the general case of surfaces
with Anosov geodesic flows.

We remark that the estimate~\eqref{e:appl-eig} with a constant which is allowed to
depend on~$k$ is true on any compact Riemannian manifold by the unique continuation
principle. However, in general this constant can go to~0 rapidly as $k\to\infty$.
For instance, if $M$ is the round sphere then one can construct a sequence
of eigenfunctions which are Gaussian beams centered on the equator, and
$\|\indic_\Omega u_k\|_{L^2(M)}$ is exponentially small in $\lambda_k$
for any $\Omega$ whose closure does not intersect the equator.
Thus the novelty of Theorem~\ref{t:appl-eig} is that it gives a bound uniform
in the \emph{high frequency limit} $k\to\infty$.

The key property of negatively curved
surfaces used in the proof is that the geodesic flow on $M$ is hyperbolic%
\footnote{The word `hyperbolic' is used in two different meanings:
for a surface, being hyperbolic means having curvature $-1$,
and for a flow, it means having a stable/unstable decomposition.}%
, or Anosov, in the sense that an infinitesimal perturbation of a geodesic diverges exponentially
fast from the original geodesic in at least one time direction.
This implies that this geodesic flow has chaotic behavior, making
negatively curved surfaces a standard model of chaotic systems and
the corresponding Laplacian eigenfunctions a standard model of quantum chaotic objects.

A motivation for Theorem~\ref{t:appl-eig} is given by the study of probability measures
$\mu$ which are \emph{weak limits} of high frequency sequences of eigenfunctions $u_{k_\ell}$ in the following sense:
$$
\int_M a(x)|u_{k_\ell}(x)|^2\,d\vol_g(x)\to \int_M a(x)\,d\mu(x)\quad\text{for all}\quad
a\in C^\infty(M).
$$
It is also natural to study the corresponding microlocal lifts, or \emph{semiclassical measures},
which are probability measures on the cosphere bundle $S^*M$ invariant under the geodesic flow,
and the results below are valid for these microlocal lifts as well
(see~\cite[Chapter~5]{e-z} and~\cite[\S1.2]{fwlEDP}).

We briefly review some results on weak limits of eigenfunctions
on negatively curved surfaces:
\begin{itemize}
\item Quantum Ergodicity, proved by Shnirelman~\cite{Shnirelman,Shnirelman2},
Zelditch~\cite{ZelditchQE}, and Colin de Verdi\`ere~\cite{CdV},
states that there exists a density~1 sequence $\{u_{k_\ell}\}$
whose weak limit is the volume measure on $M$. That is, most eigenfunctions
equidistribute in the high frequency limit.
\item The Quantum Unique Ergodicity conjecture of Rudnick--Sarnak~\cite{RudnickSarnak}
states that the volume measure is the only possible weak limit, that is
entire sequence of eigenfunctions equidistributes. So far this has
only been proved for Hecke eigenfunctions on arithmetic hyperbolic surfaces, by Lindenstrauss~\cite{Lindenstrauss}.
\item Entropy bounds of Anantharaman~\cite{Anantharaman} and Anantharaman--Nonnenmacher~\cite{AnantharamanNonnenmacher} give restrictions on possible weak limits:
the Kolmogorov--Sinai entropy of the corresponding microlocal lifts is $\geq c$
for some explicit constant $c>0$ depending only on the surface.
(For hyperbolic surfaces we have $c={1\over 2}$.)
In particular, this excludes the most degenerate situation when $\mu$
is supported on a single closed geodesic.
\item Theorem~\ref{t:appl-eig} implies that each weak limit has full support,
that is $\mu(\Omega)>0$ for any nonempty open $\Omega\subset M$. This also
excludes the case of $\mu$ supported on a single geodesic. The class of possible weak limits excluded by Theorem~\ref{t:appl-eig}
is different from the one excluded by~\cite{Anantharaman,AnantharamanNonnenmacher}
(neither is contained in the other).
\end{itemize}
Some of the above results are true in more general settings. In particular,
quantum ergodicity holds as long as the geodesic flow is ergodic;
neither quantum ergodicity nor entropy bounds require that the dimension of~$M$ be equal to~2.
We refer to the review articles by
Marklof~\cite{MarklofReview}, Zelditch~\cite{ZelditchReview},
and Sarnak~\cite{SarnakQUE} for a more detailed overview of the history of weak limits
of eigenfunctions.

We now give two more applications due to Jin~\cite{JinControl,JinDWE}
and Dyatlov--Jin--Nonnen\-macher~\cite{varfup},
building on Theorem~\ref{t:appl-eig} and its proof. The first of these is an observability estimate
for the Schr\"odinger equation (which immediately gives control for this equation by the HUM method
of Lions~\cite{Lions}):
%%%%%%%%%%%%%%%%%%%%%%%%%%%%%%%%%%%%%%%%%%%%%%%%%%%%%%%%%%%%%%%%%%%%%%%%%%%%%%%%
\begin{theo}\cite{JinControl,varfup}
  \label{t:appl-control}
For any $T>0$ and nonempty open $\Omega\subset M$ there exists $C_{T,\Omega}>0$ such that
for all $v\in L^2(M)$
$$
\|v\|_{L^2(M)}^2\leq C_{T,\Omega}\int_0^T\int_{\Omega}|e^{it\Delta}v(x)|^2\,d\vol_g(x)dt.
$$
\end{theo}
%%%%%%%%%%%%%%%%%%%%%%%%%%%%%%%%%%%%%%%%%%%%%%%%%%%%%%%%%%%%%%%%%%%%%%%%%%%%%%%%
The final application is exponential energy decay for the damped wave equation:
%%%%%%%%%%%%%%%%%%%%%%%%%%%%%%%%%%%%%%%%%%%%%%%%%%%%%%%%%%%%%%%%%%%%%%%%%%%%%%%%
\begin{theo}\cite{JinDWE,varfup}
  \label{t:appl-dwe}
Assume that $q\in C^\infty(M)$ satisfies $q\geq 0$ everywhere and $q\not\equiv 0$
(that is, there exists $x\in M$ such that $q(x)>0$).
Then every solution to the damped wave equation
$$
(\partial_t^2+ q(x)\partial_t-\Delta_g)w(t,x)=0,\quad
w(0,x)=f_0(x),\quad
\partial_t w(0,x)=f_1(x)
$$
with $f_j\in C^\infty(M)$ satisfies for some $\alpha>0,s>0,C>0$ depending only on $M,q$
\begin{equation}
  \label{e:xpd}
\|\nabla_{(t,x)}w(t)\|_{L^2(M)}\leq Ce^{-\alpha t}\big(\|f_0\|_{H^{s+1}(M)}+\|f_1\|_{H^s(M)}\big),\quad
t\geq 0.
\end{equation}
\end{theo}
%%%%%%%%%%%%%%%%%%%%%%%%%%%%%%%%%%%%%%%%%%%%%%%%%%%%%%%%%%%%%%%%%%%%%%%%%%%%%%%%
We remark that Theorems~\ref{t:appl-eig} and~\ref{t:appl-control}
(valid for any open $\Omega\neq\emptyset$)
were previously known only for the case when $M$ is a flat torus,
by Haraux~\cite{Haraux} and Jaffard~\cite{Jaffard},
and the corresponding weak limits for a torus were classified by Jakobson~\cite{JakobsonTorus}.
Theorem~\ref{t:appl-dwe}
is the first result of this kind (i.e. valid for any smooth nonnegative $q\not\equiv 0$) for any manifold. We refer the reader
to the introductions to~\cite{meassupp,JinControl,JinDWE,varfup}
for an overview of various previous results, in particular establishing exponential decay
bounds~\eqref{e:xpd} under various dynamical conditions on~$M,q$.

We now explain how Theorem~\ref{t:appl-eig} uses the fractal uncertainty principle,
restricting to the special case of hyperbolic surfaces considered in~\cite{meassupp}.
To keep our presentation brief, we ignore several subtle points in the argument,
referring to~\cite[\S\S2--3]{fwlEDP} for a more faithful exposition.
We argue by contradiction, assuming that $\|\indic_\Omega u_k\|_{L^2}$ is small.
The proof of Theorem~\ref{t:appl-eig}
uses \emph{semiclassical quantization}, which makes it possible to localize
$u$ in both position ($x$) and frequency ($\xi$) variables,
see~\cite{e-z} and~\cite[Appendix~E]{dizzy} for an introduction
to semiclassical analysis. Geometrically,
the pair $(x,\xi)$ gives a point in the cotangent bundle $T^*M$.
We define the semiclassical parameter as $h:=\lambda_k^{-1}$,
then in the semiclassical rescaling $u:=u_k$ is localized $h$-close
to the cosphere bundle $S^*M$.

By Egorov's Theorem, which is a form of the classical/quantum correspondence, the localization locus of
the function $u$ on $S^*M$ is invariant under the geodesic flow 
$\varphi_t:S^*M\to S^*M$ up to the `Ehrenfest time' $\log(1/h)$~-- see~\cite[\S2.2]{fwlEDP}.
From here and the smallness of $\|\indic_\Omega u\|_{L^2}$ we see that
$u$ is localized on the set
$\Gamma_+(\log(1/h))$ and also on the set $\Gamma_-(\log(1/h))$ where
$$
\Gamma_\pm(T):=\{(x,\xi)\in S^*M\mid \varphi_t(x,\xi)\notin\pi^{-1}(\Omega)\text{ for all }
t,\ 0\leq \mp t\leq T\}
$$
and $\pi:S^*M\to M$ is the projection map.
To make sense of this statement we define semiclassical pseudodifferential operators
$A_\pm$ which localize to $\Gamma_\pm(\log(1/h))$,
using the calculi introduced in~\cite{hgap}.
However these two operators lie in two incompatible pseudodifferential
calculi. The product $A_+A_-$ is not part of any pseudodifferential calculus.
Instead an application of the fractal uncertainty principle gives a norm bound for some $\beta>0$
\begin{equation}
  \label{e:fup-A}
\|A_+A_-\|_{L^2(M)\to L^2(M)}\leq Ch^\beta
\end{equation}
which tells us that no function can be localized on both $\Gamma_+(\log(1/h))$ and $\Gamma_-(\log(1/h))$.
This gives a contradiction, proving Theorem~\ref{t:appl-eig}.

%%%%%%%%%%%%%%%%%%%%%%%%%%%%%%%%%%%%%%%%%%%%%%%%%%%%%%%%%%%%%%%%%%%%%%%%%%%%%%%%
\begin{figure}
\hbox to\hsize{\hskip1.2cm
$\Gamma_+(0)$ \hss
$\Gamma_+(1)$ \hss
$\Gamma_+(2)$ \hss
$\Gamma_+(3)$ \hss
$\Gamma_+(4)$
\hskip.8cm}
\includegraphics[width=2.9cm]{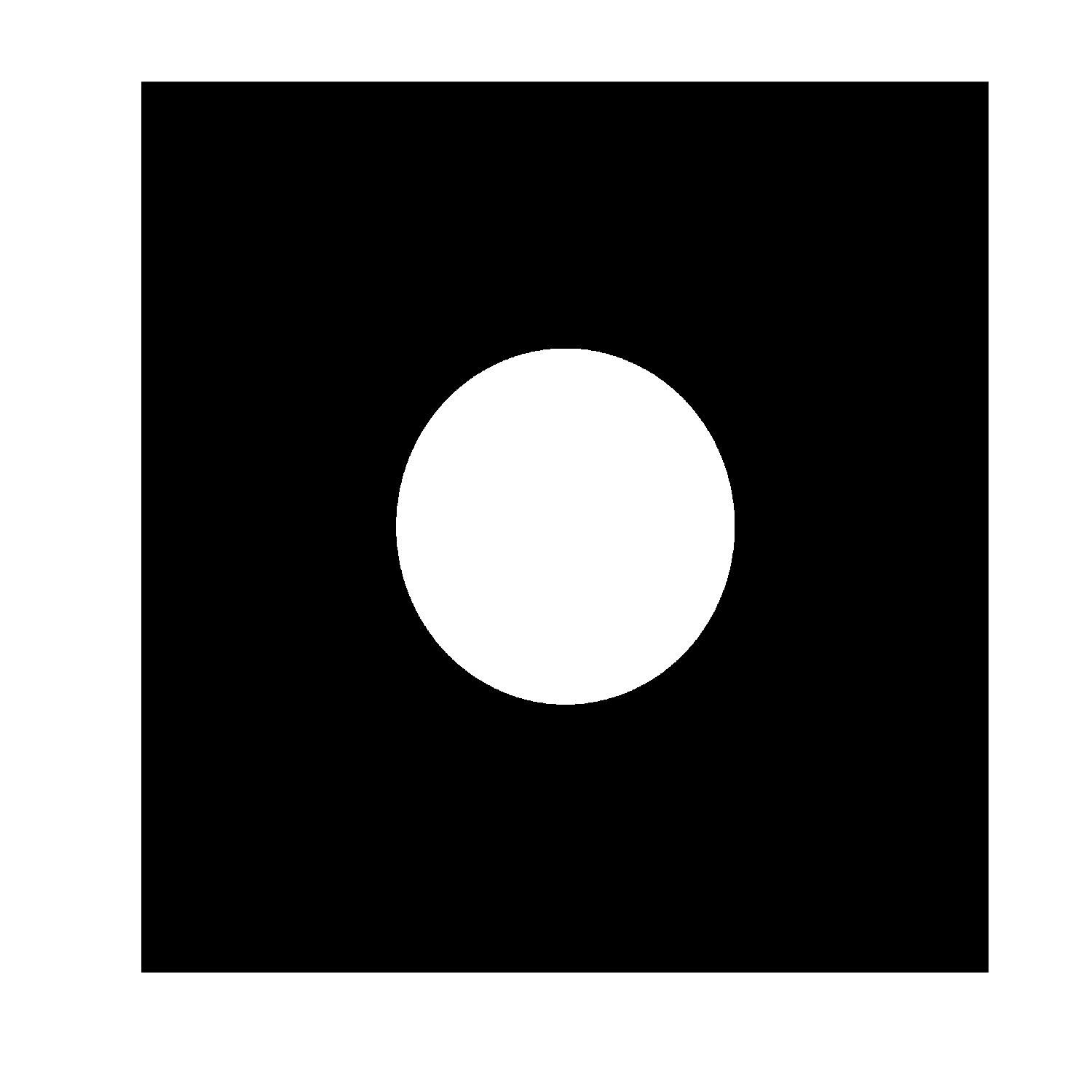}
\includegraphics[width=2.9cm]{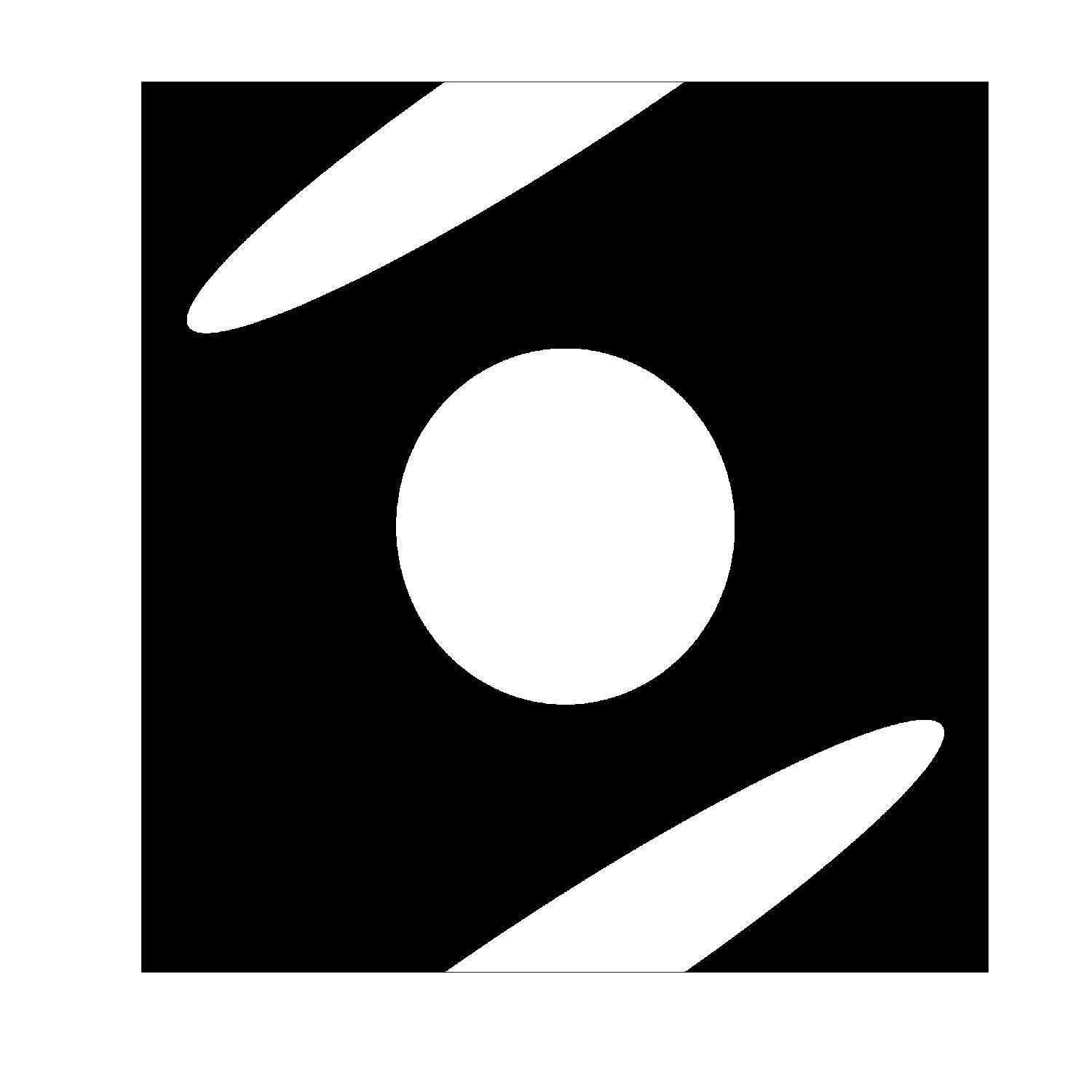}
\includegraphics[width=2.9cm]{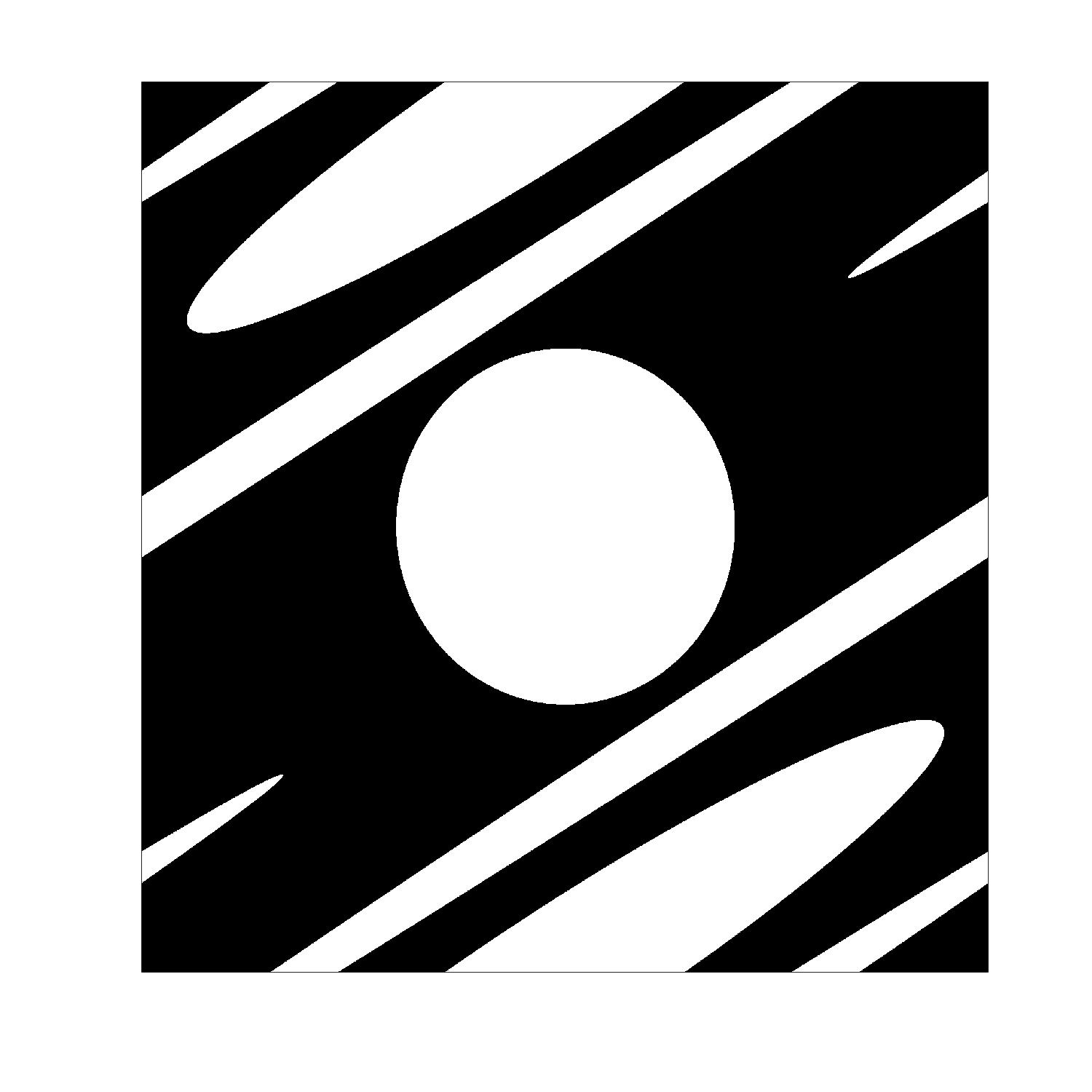}
\includegraphics[width=2.9cm]{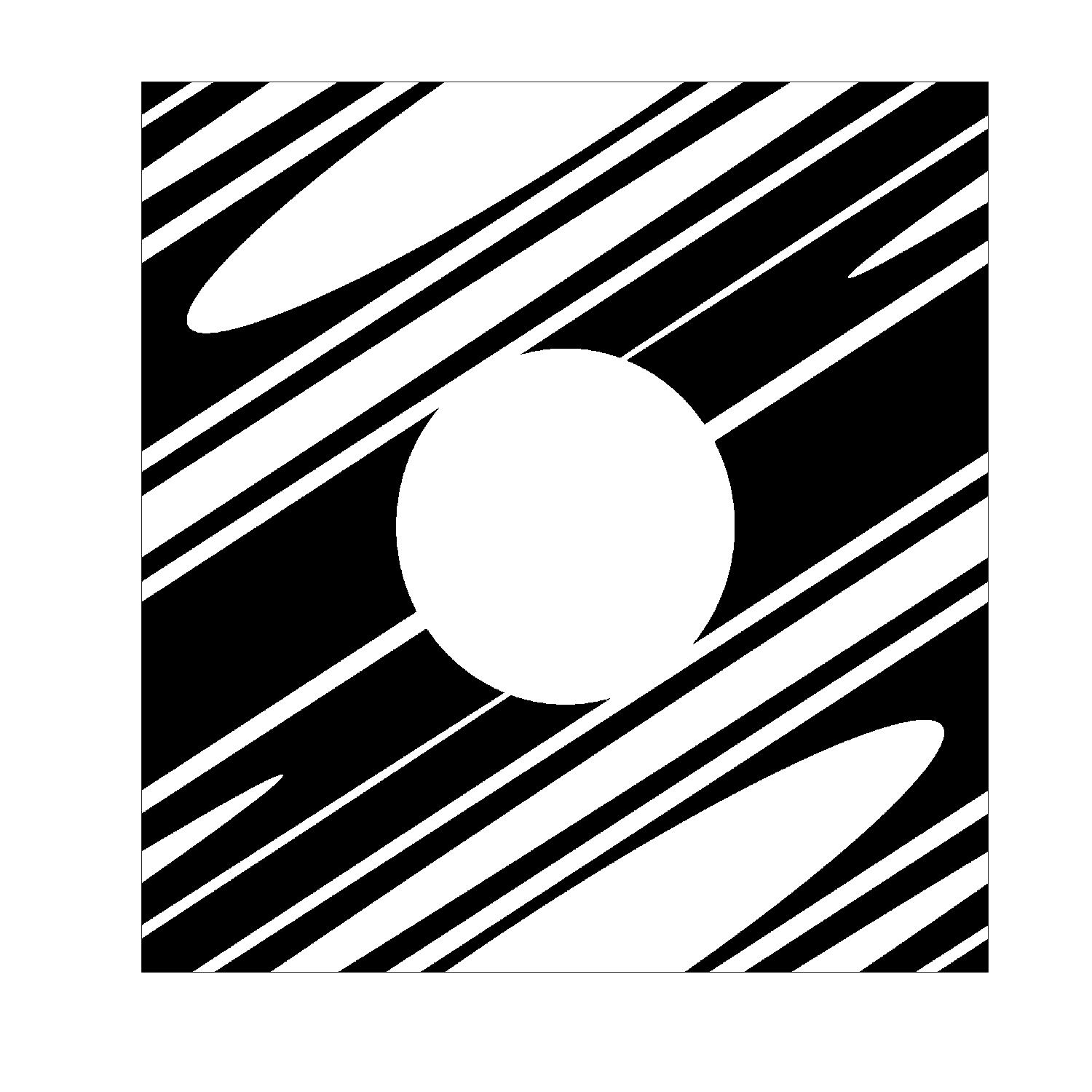}
\includegraphics[width=2.9cm]{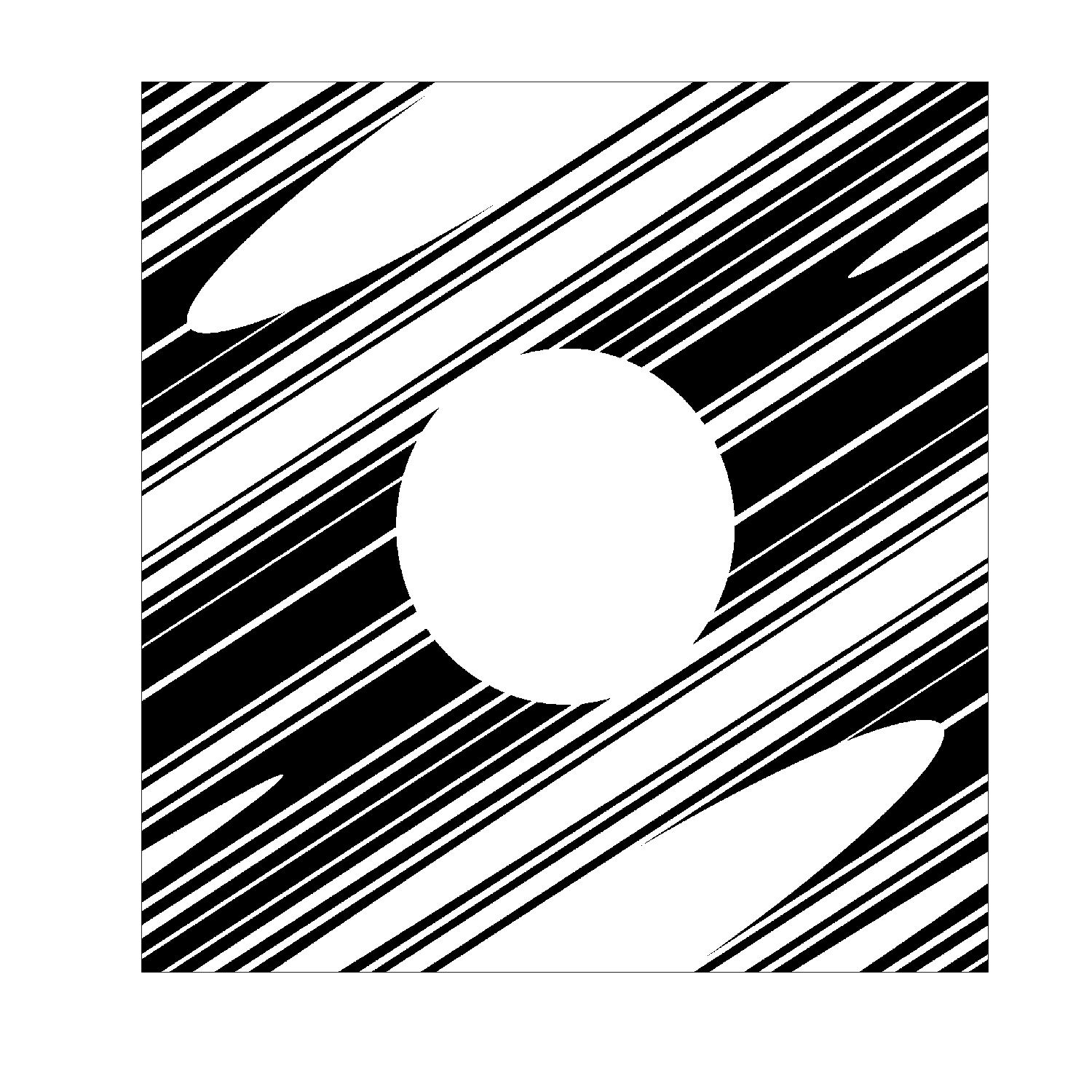}\\
\hbox to\hsize{\hskip1.2cm
$\Gamma_-(0)$ \hss
$\Gamma_-(1)$ \hss
$\Gamma_-(2)$ \hss
$\Gamma_-(3)$ \hss
$\Gamma_-(4)$
\hskip.8cm}
\includegraphics[width=2.9cm]{jpgfig/catpic0.jpg}
\includegraphics[width=2.9cm]{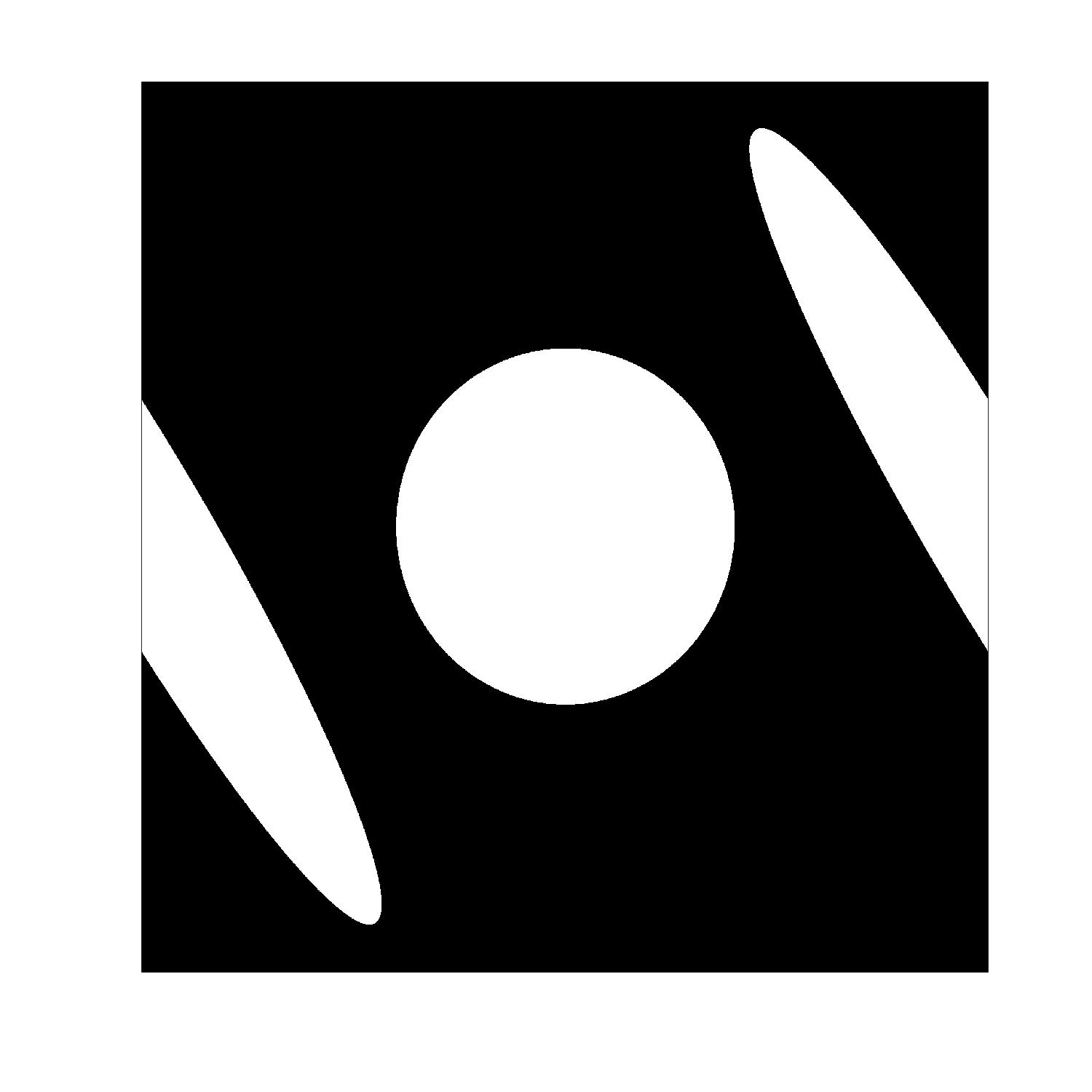}
\includegraphics[width=2.9cm]{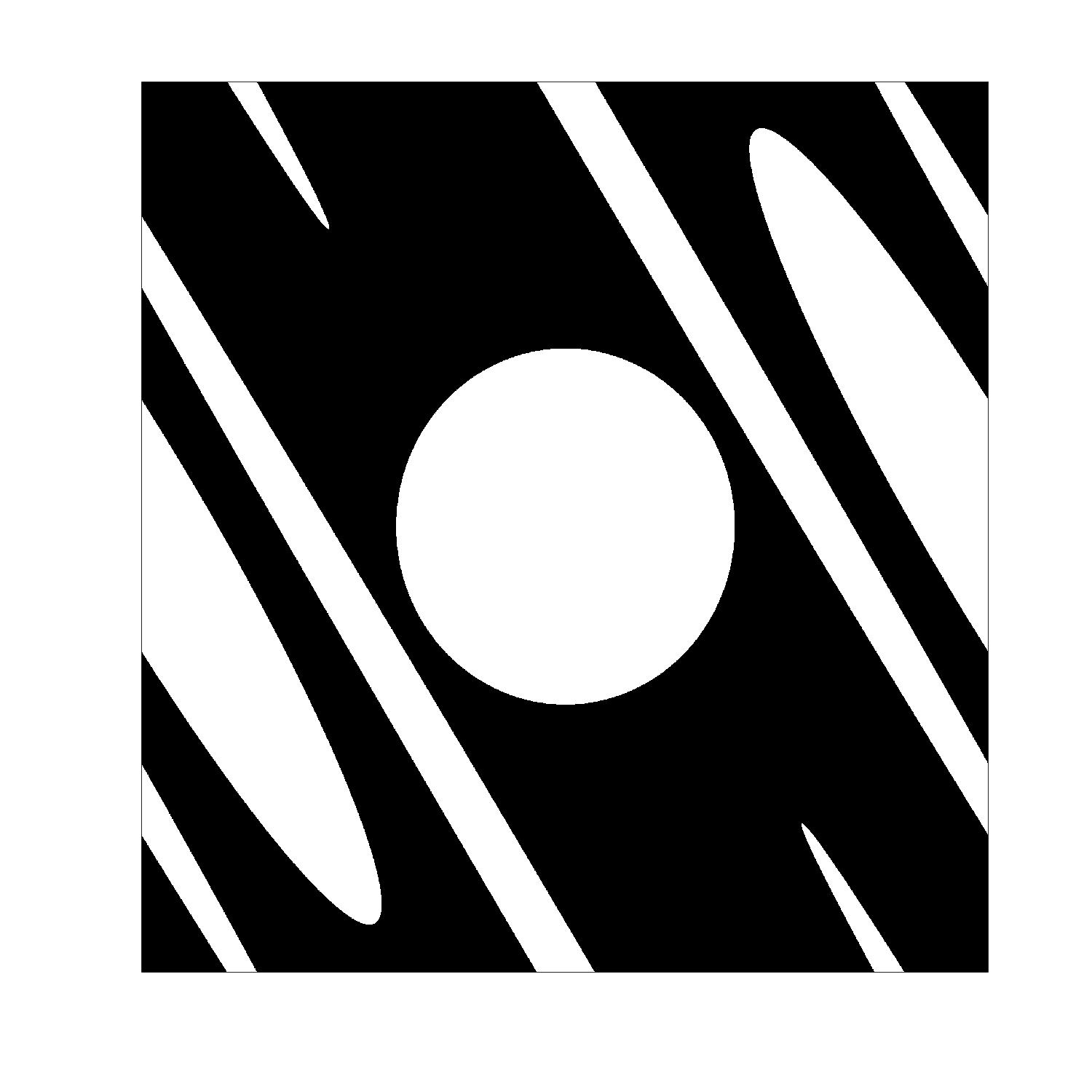}
\includegraphics[width=2.9cm]{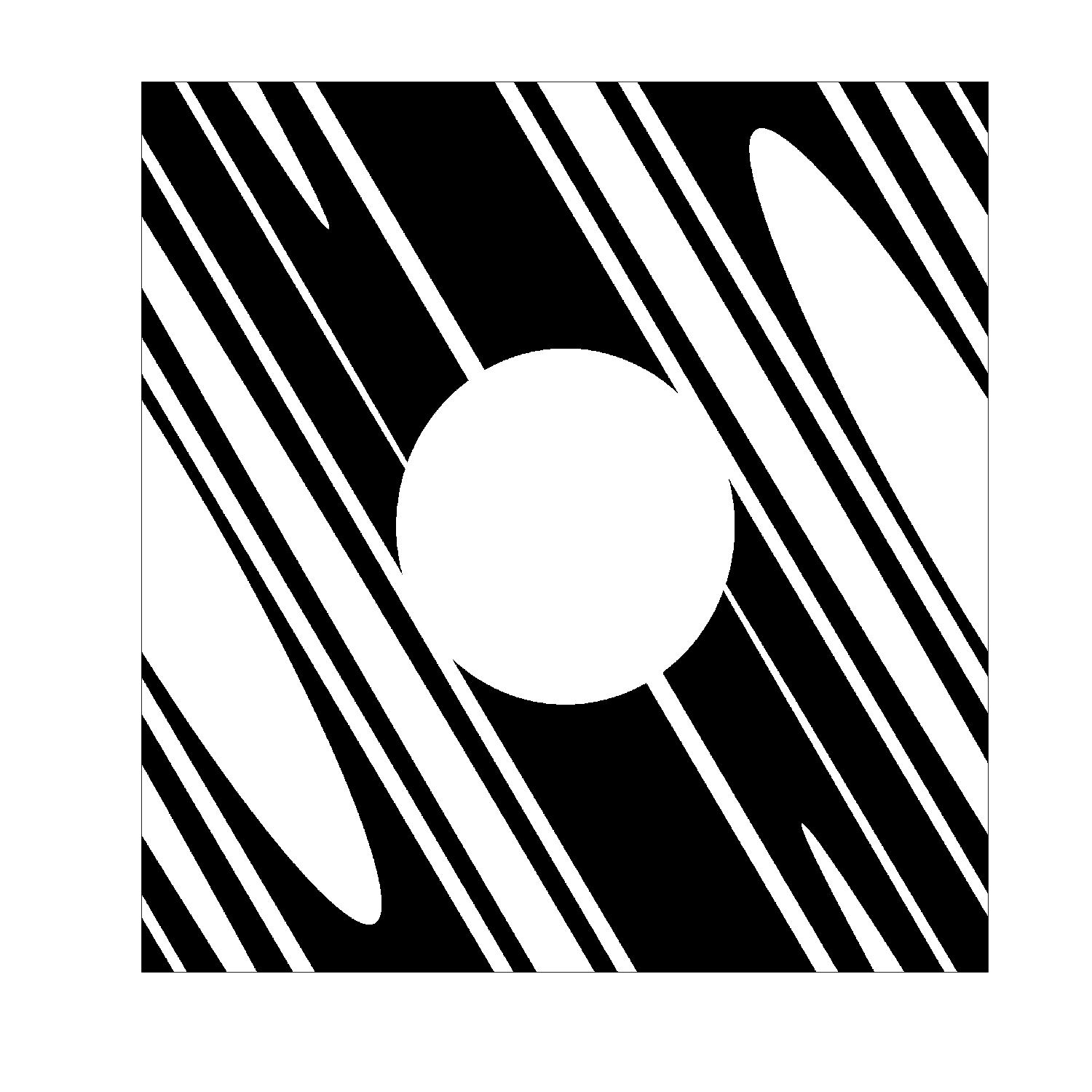}
\includegraphics[width=2.9cm]{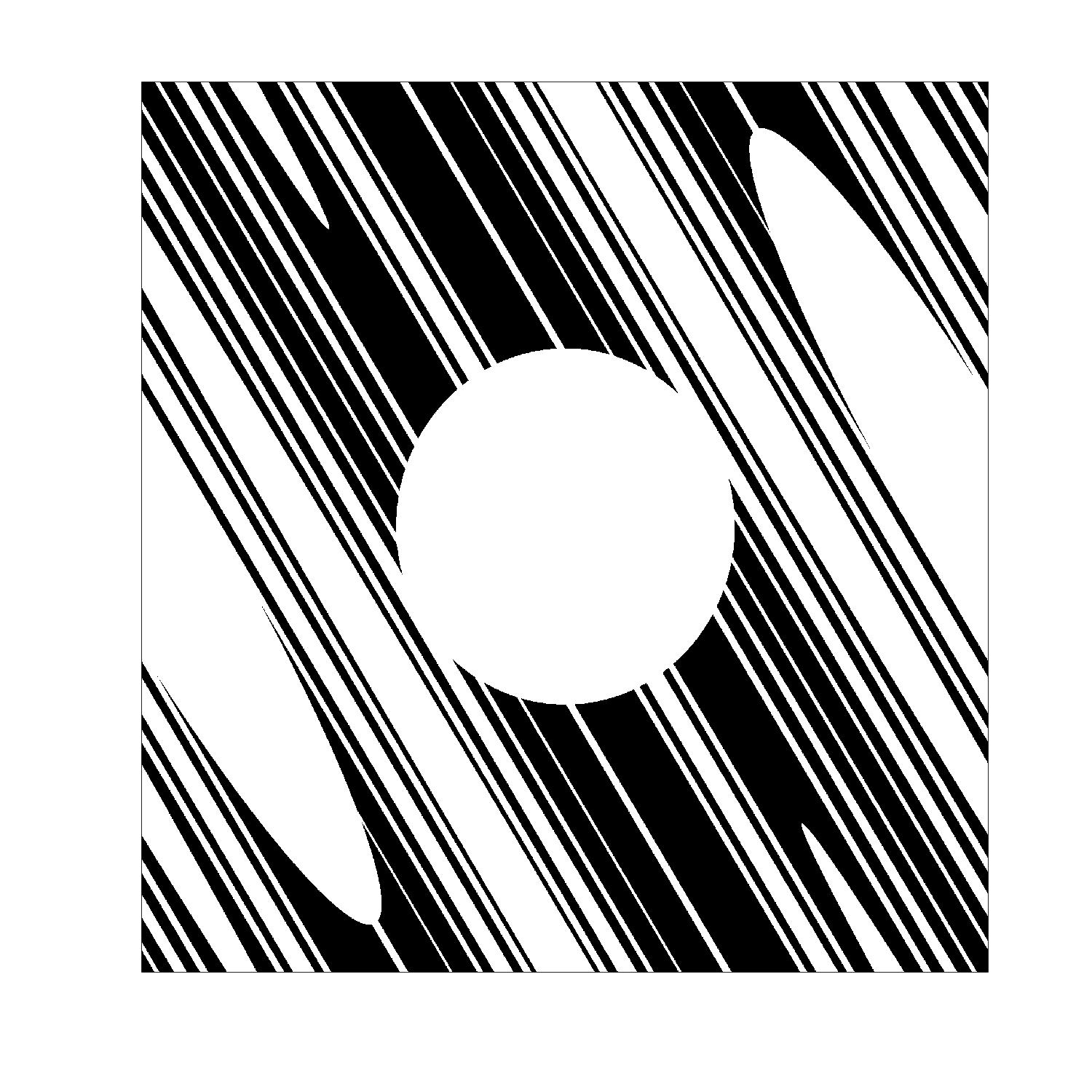}
\caption{The sets $\Gamma_\pm(T)$ (in black) with the flow direction removed.
This figure was produced using numerics for the closely related
Arnold cat map $(x,y)\mapsto (2x+y,x+y)$ on the torus $\mathbb R^2/\mathbb Z^2$.
The `hole' $\pi^{-1}(\Omega)$ is the white disk pictured on the leftmost figures.
We see that the sets $\Gamma_+(T)$ are smooth in the unstable
direction $\mathbb R(2,\sqrt{5}-1)$. Similarly the sets $\Gamma_-(T)$
are smooth in the stable direction $\mathbb R(2,-1-\sqrt 5)$.}
\label{f:porosity}
\end{figure}
%%%%%%%%%%%%%%%%%%%%%%%%%%%%%%%%%%%%%%%%%%%%%%%%%%%%%%%%%%%%%%%%%%%%%%%%%%%%%%%%
To obtain~\eqref{e:fup-A} from the fractal uncertainty principle we use the hyperbolicity
of the geodesic flow, which gives the \emph{stable/unstable decomposition} of the tangent
space to $S^*M$ at each point into
three subspaces: the space tangent to the flow, the stable space, and the unstable space.
The differential of the flow contracts vectors in the stable space
and expands those in the unstable space, with exponential rate $e^t$.

The stable/unstable decomposition implies that the set $\Gamma_+(T)$ is smooth along the unstable direction
and the flow direction, but it is $\nu$-porous on scales $e^{-T}$ to~1 in the stable direction
where porosity is understood similarly to Definition~\ref{d:porous}.
Same is true for the set $\Gamma_-(T)$, with the roles of stable/unstable directions reversed~--
see Figure~\ref{f:porosity}.
The pores at scale $\alpha\in [e^{-T},1]$ come from the restriction
that $\varphi_t(x,\xi)\notin\pi^{-1}(\Omega)$ when $\alpha\sim e^{-|t|}$, and the porosity constant $\nu$ depends on $\Omega$. Using Fourier integral operators
we can deduce the norm bound~\eqref{e:fup-A} from the hyperbolic FUP for porous sets,
Theorem~\ref{t:hfup-porous}.

The case of variable curvature considered in~\cite{varfup} presents many
additional challenges. First of all, the expansion rate of the flow is not constant,
which means that the propagation time $T$ has to depend on the base point.
Secondly, the stable/unstable foliations are not $C^\infty$, thus we cannot put the operators
$A_\pm$ into pseudodifferential calculi defined in~\cite{hgap}, and cannot use Fourier integral
operators to reduce~\eqref{e:fup-A} to an uncertainty estimate~\eqref{e:hfup}.
Finally, even if we could reduce to the estimate~\eqref{e:hfup},
the corresponding phase $\Phi$ would not be $C^\infty$ owing again to the lack of smoothness
of the stable/unstable foliations. The paper~\cite{varfup} thus employs a different strategy to
reduce~\eqref{e:hfup} to the uncertainty principle for the Fourier transform~\eqref{e:fup},
using $C^{1+}$ regularity of the stable/unstable foliations and a microlocal argument
in place of the proof of Theorem~\ref{t:hfup-0} which was described in~\S\ref{s:hfup}.

%%%%%%%%%%%%%%%%%%%%%%%%%%%%%%%%%%%%%%%%%%%%%%%%%%%%%%%%%%%%%%%%%%%%%%%%%%%%%%%%
\subsection{Spectral gaps}
  \label{s:appl-open}

We now give an application of FUP to open quantum chaos,
namely spectral gaps on noncompact hyperbolic surfaces.
Assume that $(M,g)$ is a connected complete noncompact hyperbolic surface
which is \emph{convex co-compact}, that is its infinite ends are funnels.
(See the book of Borthwick~\cite{BorthwickBook} for an introduction
to scattering on hyperbolic surfaces.)
Each such surface can be realized as a quotient
$M=\Gamma\backslash\mathbb H^2$ of the
Poincar\'e upper half-plane model of the hyperbolic space
$$
\mathbb H^2=\{z\in\mathbb C\mid \Im z>0\},\quad
g={|dz|^2\over |\Im z|^2}
$$
by a Schottky group $\Gamma\subset\SL_2(\mathbb R)$ constructed in~\S\ref{s:schottky}.
Here each $\gamma\in\SL_2(\mathbb R)$ defines an isometry of $(\mathbb H^2,g)$
by the formula~\eqref{e:mobius}, where $x\in\dot{\mathbb R}$ is replaced
by $z\in \mathbb H^2$. If $I_1,\dots,I_{2r}$ are the intervals used to define
the Schottky structure and $D_w\subset\mathbb C$, $w\in\{1,\dots,2r\}$, are disks with diameters $I_w$,
then $M$ can be obtained from the fundamental domain
\begin{equation}
  \label{e:fund-domain}
\mathcal F:=\mathbb H^2\setminus\bigsqcup_{w=1}^{2r} D_w
\end{equation}
by gluing each half-circle $\mathbb H^2\cap \partial D_w$ with $\mathbb H^2\cap\partial D_{\overline w}$
by the map $\gamma_w$.
See~\cite[\S15.1]{BorthwickBook} for more details
and Figure~\ref{f:3f} for an example. We assumed in~\S\ref{s:schottky}
that $r\geq 2$, this corresponds to $\Gamma$ being a nonelementary group;
equivalently, we assume that $M$ is neither the hyperbolic space nor a hyperbolic cylinder.
%%%%%%%%%%%%%%%%%%%%%%%%%%%%%%%%%%%%%%%%%%%%%%%%%%%%%%%%%%%%%%%%%%%%%%%%%%%%%%%%
\begin{figure}
\includegraphics[width=8cm]{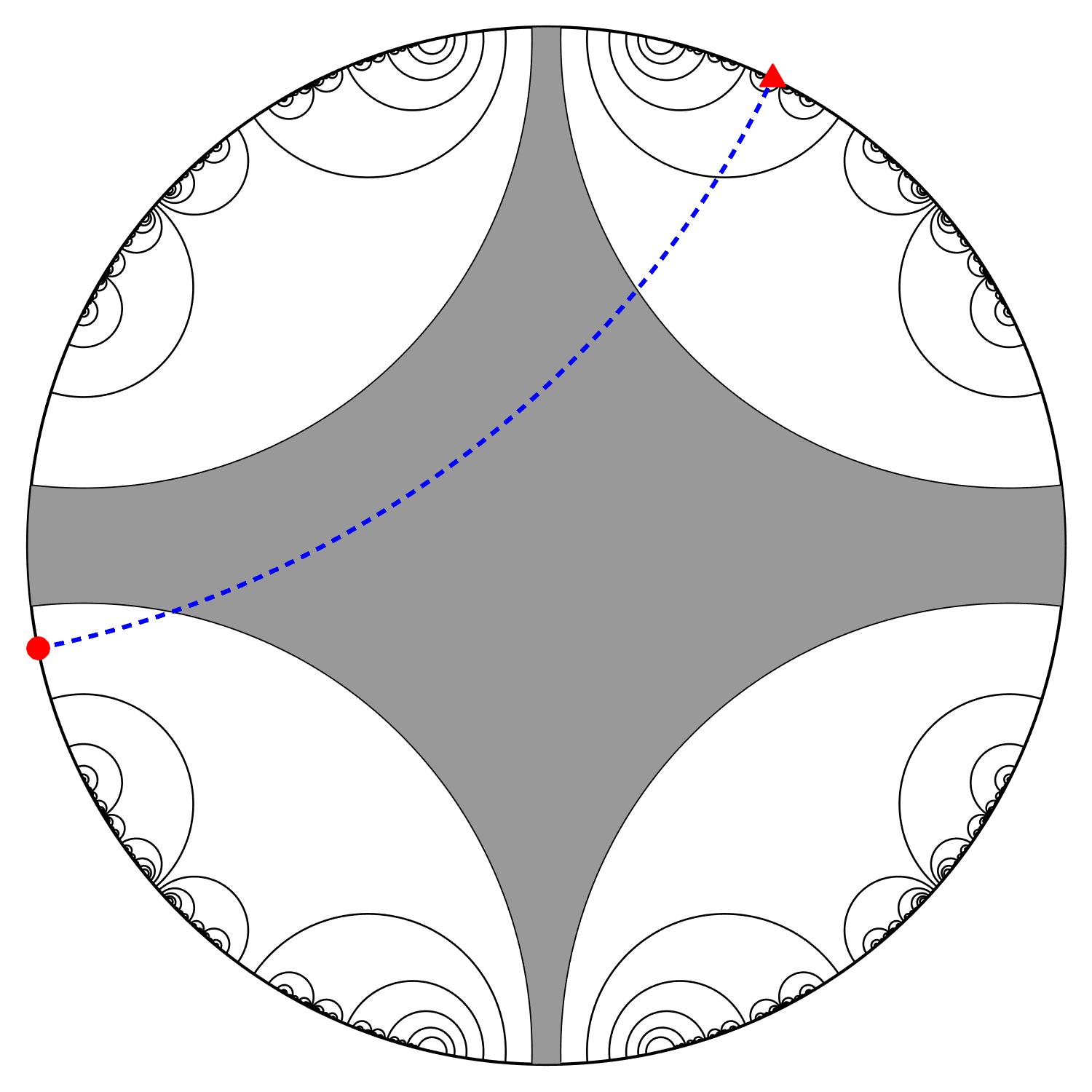}
\caption{A tessellation of the hyperbolic plane $\mathbb H^2$ (pictured here
in the Poincar\'e disk model, with $\dot{\mathbb R}$ pictured as the unit circle) by fundamental domains of a convex co-compact
group $\Gamma$ constructed using the procedure
in~\S\ref{s:schottky} with $r=2$. The quotient $\Gamma\backslash\mathbb H^2$
is a surface with three funnel ends. The fundamental domain
defined in~\eqref{e:fund-domain} is shaded. The dashed line is a geodesic
on $\mathbb H^2$ whose limiting starting point (denoted by the circle) does not
lie in the limit set $\Lambda_\Gamma$ but whose limiting ending point (denoted by the triangle)
lies in the limit set. The projection of this geodesic to the quotient
is trapped as $t\to\infty$ but not as $t\to-\infty$.}
\label{f:3f}
\end{figure}
%%%%%%%%%%%%%%%%%%%%%%%%%%%%%%%%%%%%%%%%%%%%%%%%%%%%%%%%%%%%%%%%%%%%%%%%%%%%%%%%
The limit set $\Lambda_\Gamma\subset\mathbb R$, defined in~\eqref{e:schottky-X},
determines the structure of trapped geodesics on $M$. More precisely,
we say a geodesic $\theta(t)$ on~$M$ is \emph{trapped as $t\to+\infty$} if
$\theta(t)$ does not go to an infinite end of~$M$ as $t\to +\infty$.
Similarly we define the notion of being trapped as $t\to -\infty$.
We lift $\theta$ to a geodesic on $\mathbb H^2$, which is
a half-circle starting at some point $\theta_-\in\dot{\mathbb R}$
and ending at some point $\theta_+\in\dot{\mathbb R}$.
Then
\begin{equation}
  \label{e:trapped-rel}
\theta\text{ is trapped as }t\to\pm\infty\quad\Longleftrightarrow\quad
\theta_\pm\in\Lambda_\Gamma.
\end{equation}
See Figure~\ref{f:3f} and~\cite[\S14.8]{BorthwickBook}.

We now define the `quantum' objects associated to the surface $M$,
called \emph{scattering resonances}. These are the poles of the meromorphic
continuation of the $L^2$ resolvent,
$$
R(\lambda)=\Big(-\Delta_g-\lambda^2-{1\over 4}\Big)^{-1}:
\begin{cases}
L^2(M)\to L^2(M),&\Im\lambda>0;\\
L^2_{\comp}(M)\to L^2_{\loc}(M),&\Im\lambda\leq 0.
\end{cases}
$$
See~\cite[Chapter~5]{dizzy} and~\cite[Chapter~8]{BorthwickBook} for the existence of this meromorphic continuation and an overview of hypebolic scattering,
and Figure~\ref{f:resonances} for an example. 

Resonances describe long time behavior of solutions to the (modified) wave equation
$(\partial_t^2-\Delta_g-{1\over 4})v(t,x)=0$
by the following resonance expansion%
\footnote{To simplify the formula~\eqref{e:resex}, we
assumed that $R(\lambda)$ has simple poles. Also, to prove a resonance expansion
one typically needs to make assumptions on high frequency behavior of $R(\lambda)$
such as the essential spectral gap which we study here.}
\begin{equation}
  \label{e:resex}
v(t,x)\sim \sum_{\lambda_j\text{ resonance}}e^{-it\lambda_j}v_j(x)\quad\text{as }
t\to\infty,\
x\text{ bounded}.
\end{equation}
See~\cite[Chapter~1]{dizzy} for more details and various applications of resonances.

%%%%%%%%%%%%%%%%%%%%%%%%%%%%%%%%%%%%%%%%%%%%%%%%%%%%%%%%%%%%%%%%%%%%%%%%%%%%%%%%
\begin{figure}
\includegraphics[width=15cm]{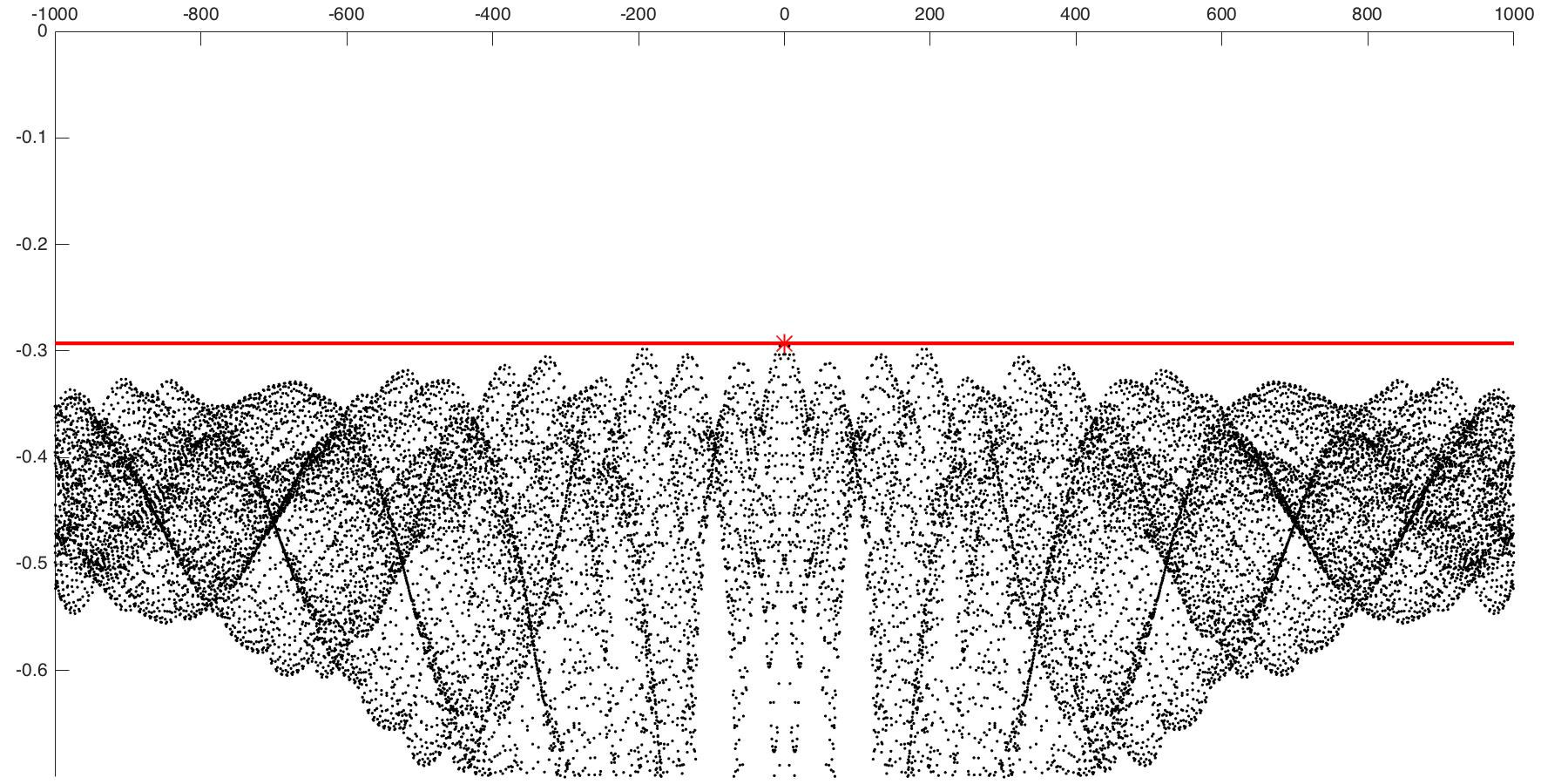}
\caption{Resonances for a three-funnel hyperbolic surface (similar to the
one on Figure~\ref{f:3f}), plotted using data provided by David Borthwick
and Tobias Weich. The topmost resonance $i(\delta-{1\over 2})$ is marked
by an asterisk and the red line is $\{\Im z=\delta-{1\over 2}\}$.
See~\cite{BorthwickNum,Borthwick-Weich}, and~\cite[Chapter~16]{BorthwickBook}
for more pictures of resonances of hyperbolic surfaces. As explained in these
works, numerically estimating the size of the essential spectral gap
is very challenging.}
\label{f:resonances}
\end{figure}
%%%%%%%%%%%%%%%%%%%%%%%%%%%%%%%%%%%%%%%%%%%%%%%%%%%%%%%%%%%%%%%%%%%%%%%%%%%%%%%%

The main topic of this section is the concept of an essential spectral gap:
%%%%%%%%%%%%%%%%%%%%%%%%%%%%%%%%%%%%%%%%%%%%%%%%%%%%%%%%%%%%%%%%%%%%%%%%%%%%%%%%
\begin{defi}
  \label{d:spectral-gap}
We say that $M$ has an \textbf{essential spectral gap} of size $\beta$,
if the half-plane $\{\Im\lambda\geq -\beta\}$ only has finitely many resonances.
(Such a gap is nontrivial only for $\beta>0$.)
\end{defi}
%%%%%%%%%%%%%%%%%%%%%%%%%%%%%%%%%%%%%%%%%%%%%%%%%%%%%%%%%%%%%%%%%%%%%%%%%%%%%%%%
In the expansion~\eqref{e:resex}, the real part of a resonance $\lambda_j$ gives the
rate of oscillation of the function $e^{-it\lambda_j}$, and the (negative)
imaginary part gives the rate of decay.
Thus an essential spectral gap of size $\beta>0$ implies exponential decay $\mathcal O(e^{-\beta t})$
of solutions to the wave equation, modulo a finite dimensional space
corresponding to resonances with $\Im\lambda_j\geq -\beta$.

We emphasize that resonances can be defined for a variety of quantum open systems
(for instance, obstacle scattering or wave equations on black holes) and having an essential spectral
gap is equivalent to exponential local energy decay of high frequency waves,
see for instance~\cite[Theorems~2.9 and~5.40]{dizzy}. This in particular has applications
to nonlinear equations, such as black hole stability (see Hintz--Vasy~\cite{HintzVasyGreat})
and Strichartz estimates (see Burq--Guillarmou--Hassell~\cite{BGH}
and Wang~\cite{WangJian}).

Existence of an essential spectral gap depends on the structure of trapped classical trajectories (see~\cite[Chapter~6]{dizzy}). For a convex co-compact hyperbolic surface,
the set of all trapped geodesics has fractal structure (by~\eqref{e:trapped-rel})
and the geodesic flow has hyperbolic behavior on this set (namely it has a stable/unstable decomposition). Thus convex co-compact hyperbolic surfaces
serve as a model for more general systems with fractal hyperbolic trapped sets.
The latter class includes scattering by several convex obstacles (see Figure~\ref{f:obstacles}),
where spectral gaps have been observed in microwave scattering experiments
by Barkhofen et al.~\cite{ZworskiPRL}. We refer to the reviews of Nonnenmacher~\cite{Nonnenmacher} and Zworski~\cite{ZworskiReview} for an overview of results
on spectral gaps for open quantum chaotic systems.
%%%%%%%%%%%%%%%%%%%%%%%%%%%%%%%%%%%%%%%%%%%%%%%%%%%%%%%%%%%%%%%%%%%%%%%%%%%%%%%%
\begin{figure}
\includegraphics[height=5cm]{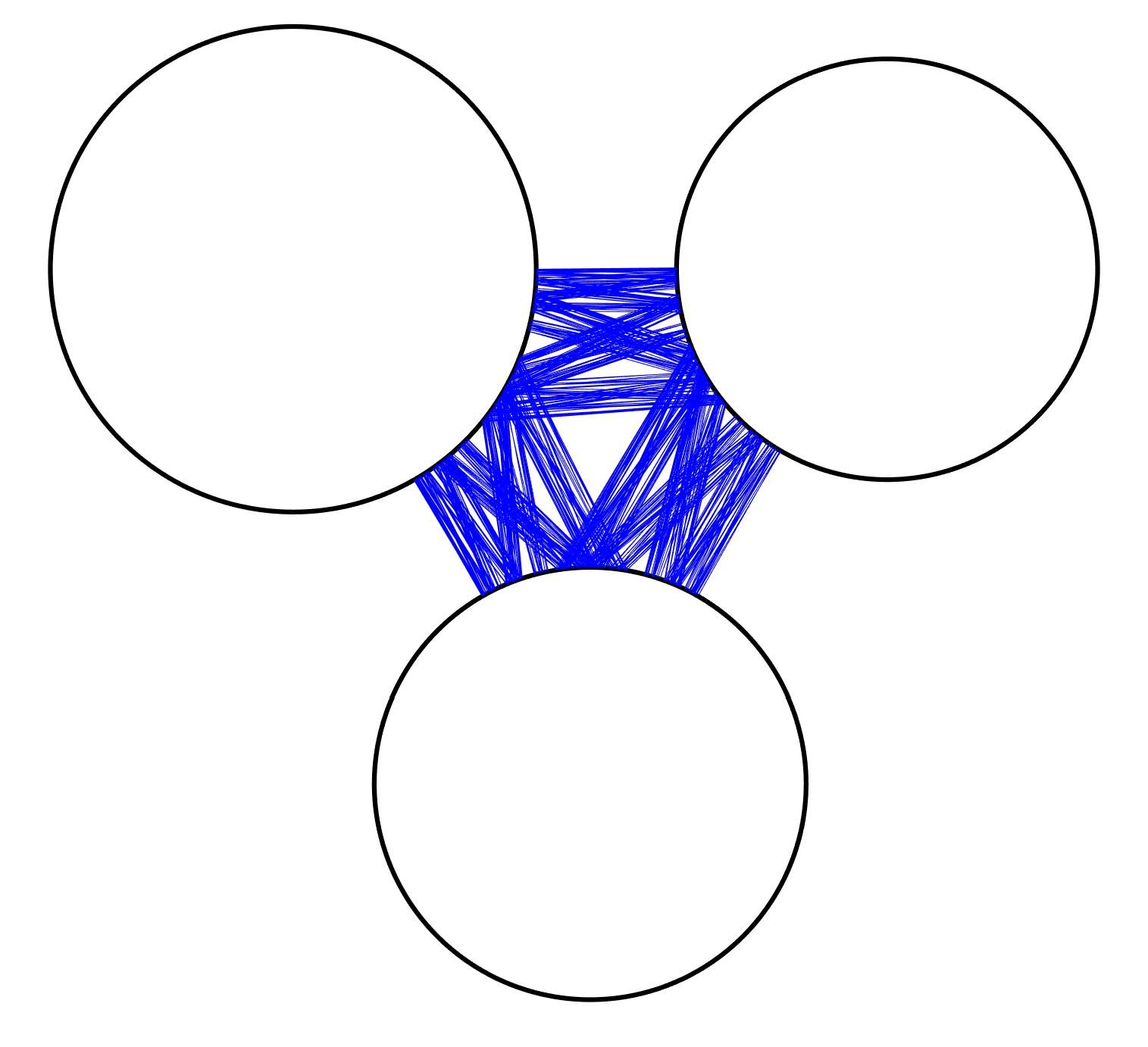}
\includegraphics[height=5cm]{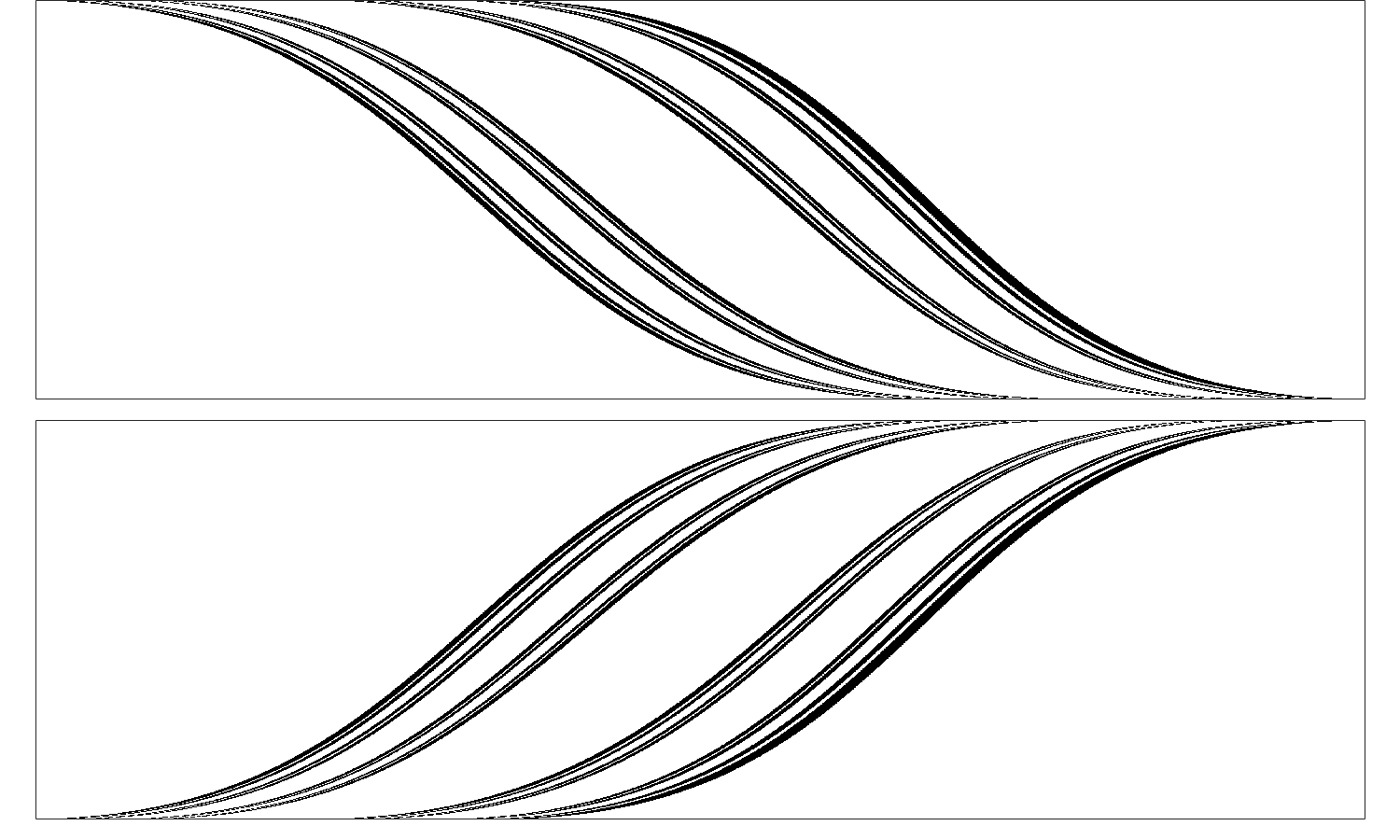}
\caption{A `real-world' example of a system with fractal hyperbolic trapping:
the billiard ball map on the exterior of three disks in~$\mathbb R^2$. The trapped billiard
trajectories are shown on the left. On the right is the set of trajectories
trapped backward (top) and forward (bottom)
restricted to the boundary of one disk, where the horizontal variable is the point on the boundary
and the vertical variable is the angle of the tangent vector to the boundary.
(See~\cite[\S5.2]{stunnote} for details.)
The fractal uncertainty principle in this setting would say that no function can be localized
in the position-frequency space to both forward and backward trapped sets.}
\label{f:obstacles}
\end{figure}
%%%%%%%%%%%%%%%%%%%%%%%%%%%%%%%%%%%%%%%%%%%%%%%%%%%%%%%%%%%%%%%%%%%%%%%%%%%%%%%%

Coming back to hyperbolic surfaces, it is well-known that
there is an essential spectral gap of size $\beta=0$.
In fact, resonances with $\Im\lambda>0$ correspond to the (finitely many) $L^2$ eigenvalues
of~$-\Delta_g$ in $[0,{1\over 4})$.
There is also the \emph{Patterson--Sullivan gap} $\beta={1\over 2}-\delta$
(see~\cite{Patterson3,Sullivan})
where $\delta\in (0,1)$ is the dimension of the limit set (see~\S\ref{s:schottky}).
In fact, the resonance with the largest imaginary part
is given by $\lambda=i(\delta-{1\over 2})$, see~\cite[Theorem~14.15]{BorthwickBook}.
Thus we have an essential spectral gap of the size $\beta=\max(0,{1\over 2}-\delta)$.

The application of FUP to spectral gaps is based on the following
%%%%%%%%%%%%%%%%%%%%%%%%%%%%%%%%%%%%%%%%%%%%%%%%%%%%%%%%%%%%%%%%%%%%%%%%%%%%%%%%
\begin{theo}\cite{hgap,tug}
\label{t:fup-open-rel}
Let $M=\Gamma\backslash\mathbb H^2$ be a convex co-compact
hyperbolic surface and $\Lambda_\Gamma\subset\mathbb R$ be the limit
set of the group~$\Gamma$. Denote by $\Lambda_\Gamma(h)$ the $h$-neighborhood
of~$\Lambda_\Gamma$.

Assume that $X=Y=\Lambda_\Gamma(h)$ satisfies the
hyperbolic uncertainty principle~\eqref{e:hfup} with some exponent $\beta>0$,
for the phase function $\Phi(x,y)=\log|x-y|$ from~\eqref{e:Phi-hyp}
and every choice of the amplitude~$b\in\CIc(\mathbb R^2\setminus \{x=y\})$. Then $M$ has an essential spectral
gap of size $\beta-\varepsilon$ for each $\varepsilon>0$.
\end{theo}
%%%%%%%%%%%%%%%%%%%%%%%%%%%%%%%%%%%%%%%%%%%%%%%%%%%%%%%%%%%%%%%%%%%%%%%%%%%%%%%%
Two different proofs of Theorem~\ref{t:fup-open-rel} are given in~\cite{hgap} and~\cite{tug}.
The proof in~\cite{hgap} uses microlocal methods
similar to the proof of Theorem~\ref{t:appl-eig}. Roughly speaking,
if $\lambda$ is a resonance with $|\Re\lambda|=h^{-1}\gg 1$ and $\Im\lambda=-\nu$, then
there exists a resonant state which is a solution $u$ to the equation
$(-\Delta_g-\lambda^2-{1\over 4})u=0$ satisfying a certain outgoing condition
at the infinite ends of~$M$.
Next, $u$ is microlocalized $h$-close to the set of backward
trapped trajectories, and it has mass at least $h^{2\nu}$ on the $h$-neighborhood
of the set of forward trapped trajectories (here mass is the square of the $L^2$ norm). The fractal uncertainty principle
then implies that $h^\nu\leq h^\beta$, that is $\nu\geq \beta$. Here
the limit set enters via the description of trapped trajectories in~\eqref{e:trapped-rel}.
Compared to the compact setting described in~\S\ref{s:appl-closed}
a key additional ingredient is the work of Vasy~\cite{Vasy-AH1,Vasy-AH2}
on effective meromorphic continuation of the scattering resolvent.

The other proof of Theorem~\ref{t:fup-open-rel}, given in~\cite{tug},
proceeds by bounding the spectral radius of the transfer operator of the Bowen--Series map.
That proof is much shorter than~\cite{hgap} but the method is less likely to
be applicable to more general open hyperbolic systems.

Combining Theorem~\ref{t:fup-open-rel} with the fractal uncertainty principle
from Theorems~\ref{t:hfup-0}--\ref{t:hfup-pres} we obtain
%%%%%%%%%%%%%%%%%%%%%%%%%%%%%%%%%%%%%%%%%%%%%%%%%%%%%%%%%%%%%%%%%%%%%%%%%%%%%%%%
\begin{theo}
\label{t:fup-open}
Let $M,\Lambda_\Gamma$ be as in Theorem~\ref{t:fup-open-rel}
and $\delta\in (0,1)$ be the dimension of~$\Lambda_\Gamma$. Then
$M$ has an essential spectral gap of size $\beta$ for some
$\beta>\max(0,{1\over 2}-\delta)$. See Figure~\ref{f:gaps-hyp}.
\end{theo}
%%%%%%%%%%%%%%%%%%%%%%%%%%%%%%%%%%%%%%%%%%%%%%%%%%%%%%%%%%%%%%%%%%%%%%%%%%%%%%%%
\begin{figure}
\includegraphics{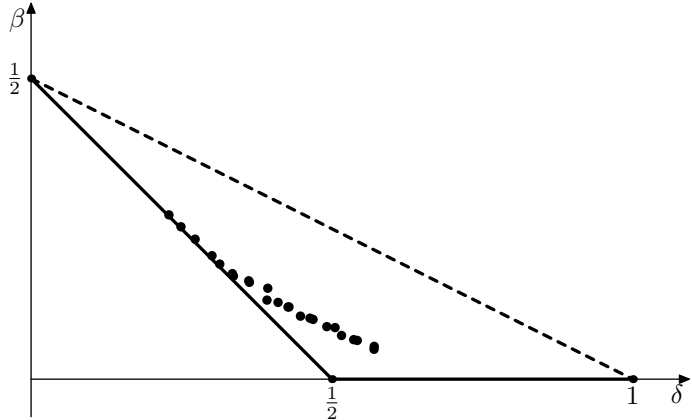}
\caption{Numerically computed essential spectral gaps $\beta$ for symmetric 3-funneled and 4-funneled surfaces from~\cite[Figure~14]{Borthwick-Weich} (specifically, $G_{100}^{I_1}$ in the notation
of~\cite{Borthwick-Weich}). Each point corresponds to one surface and has
coordinates $(\delta,\beta)$. The solid line is the standard gap $\beta=\max(0,{1\over 2}-\delta)$.
The dashed line is the Jakobson--Naud conjecture $\beta={1-\delta\over 2}$.}
\label{f:gaps-hyp}
\end{figure}
%%%%%%%%%%%%%%%%%%%%%%%%%%%%%%%%%%%%%%%%%%%%%%%%%%%%%%%%%%%%%%%%%%%%%%%%%%%%%%%%
An essential spectral gap of size $\beta>{1\over 2}-\delta$
when $0<\delta\leq{1\over 2}$ was previously established by Naud~\cite{NaudGap}.
This improvement over the Patterson--Sullivan gap was used to
get an asymptotic formula for the number $\mathcal N(L)$ of primitive closed geodesics
of period $\leq L$ of the form
$$
\mathcal N(L)=\li(e^{\delta T})+\mathcal O(e^{(\delta-\varepsilon)T}),\quad
\li(x):=\int_2^x{dt\over\log t}
$$
for some $\varepsilon>0$, see~\cite[Theorem~1.4]{NaudGap}.
Spectral gaps with $\beta>{1\over 2}-\delta$ also have important applications
to diophantine problems in number theory, see Bourgain--Gamburd--Sarnak~\cite{BGS},
Magee--Oh--Winter~\cite{MOW}, and the review of Sarnak~\cite{SarnakThin}.
A spectral gap $\beta>{1\over2}-\delta$ depending only on the dimension~$\delta$ of the limit set is given in Theorem~\ref{t:fdec-used} below.

Jakobson--Naud~\cite{Jakobson-Naud} conjectured an essential spectral gap of size $\beta={1-\delta\over 2}$,
see Figure~\ref{f:gaps-hyp}. This conjecture corresponds
to the upper limit of possible results that could be proved using FUP:
indeed, by applying $\indic_{\Lambda_\Gamma(h)} \mathcal B_h\indic_{\Lambda_\Gamma(h)}$ to a function localized in an $h$-sized
interval inside $\Lambda_\Gamma(h)$ and using that $\vol(\Lambda_\Gamma(h))\sim h^{1-\delta}$ we see that
if~\eqref{e:hfup} holds with some value of $\beta$, then we necessarily have $\beta\leq {1-\delta\over 2}$.
While the Jakobson--Naud conjecture is out of reach of current methods,
its analogue is known to hold in certain special cases
in the `toy model' setting of open quantum cat maps, see~\cite[\S3.5]{oqm}.

For more general open systems with hyperbolic
trapping, an essential spectral gap was known under a \emph{pressure condition}
which generalizes the inequality $\delta<{1\over 2}$,
by Ikawa~\cite{Ikawa}, Gaspard--Rice~\cite{GaspardRice}, and Nonnenmacher--Zworski~\cite{NonnenmacherZworskiActa}. In some cases there exists a gap
strictly larger than the pressure gap: see Petkov--Stoyanov~\cite{PetkovStoyanov}
and Stoyanov~\cite{Stoyanov1,Stoyanov2}, in addition to the work of Naud mentioned above.

In contrast with the pressure gap and improvements over it, Theorem~\ref{t:fup-open}
gives an essential spectral gap $\beta>0$ for all convex co-compact hyperbolic surfaces.
This makes it a special case of the conjecture of Zworski~\cite[\S3.2, Conjecture~3]{ZworskiReview}
that \emph{every open hyperbolic system has an essential spectral gap}.

%%%%%%%%%%%%%%%%%%%%%%%%%%%%%%%%%%%%%%%%%%%%%%%%%%%%%%%%%%%%%%%%%%%%%%%%%%%%%%%%
%%%%%%%%%%%%%%%%%%%%%%%%%%%%%%%%%%%%%%%%%%%%%%%%%%%%%%%%%%%%%%%%%%%%%%%%%%%%%%%%
\section{FUP for discrete Cantor sets}
  \label{s:cantor}

We now discuss FUP for a special class of regular fractal sets, namely
discrete Cantor sets. In this setting we provide a complete proof of the
fractal uncertainty principle of Theorems~\ref{t:fup-0}--\ref{t:fup-pres}.
In~\cite{oqm} this special case of FUP was applied to obtain an essential
spectral gap for the `toy model' of \emph{quantum open baker's maps},
similarly to the application to convex co-compact hyperbolic surfaces
discussed in~\S\ref{s:appl-open}. We refer to~\cite{oqm} for a discussion
of these quantum maps and more qualitative information on FUP for Cantor sets.

A discrete Cantor set is a subset of $\mathbb Z_N:=\{0,\dots,N-1\}$
of the form
\begin{equation}
  \label{e:discrete-cantor}
\mathcal C_k:=\big\{ a_0+a_1M+\cdots+a_{k-1}M^{k-1}\mid
a_0,\dots,a_{k-1}\in\mathcal A\big\},\quad
N:=M^k
\end{equation}
where $k$ (called the \emph{order} of the set) is a large natural number and we fixed
\begin{itemize}
\item an integer $M\geq 3$, called the \emph{base}, and
\item a nonempty subset $\mathcal A\subset \{0,\dots,M-1\}$,
called the \emph{alphabet}.
\end{itemize}
In other words, $\mathcal C_k$ is the set of numbers of length $k$ in base $M$
with all digits in $\mathcal A$. Note that
$|\mathcal C_k|=|\mathcal A|^k=N^\delta$ where the dimension $\delta$
is defined by
\begin{equation}
  \label{e:delta-discrete}
\delta:={\log|\mathcal A|\over \log M}\in [0,1].
\end{equation}
We have $0<\delta<1$ except in the trivial cases $|\mathcal A|=1$ and $|\mathcal A|=M$.
The number $\delta$ is the dimension of the limiting Cantor set
\begin{equation}
  \label{e:continuous-cantor}
\mathcal C_\infty:=\bigcap_{k\geq 1}\bigcup_{j\in\mathcal C_k} \Big[{j\over M^k},{j+1\over M^k}\Big]\subset [0,1].
\end{equation}
More precisely, $\mathcal C_\infty$ is $\delta$-regular on scales~0 to~1 similarly
to Example~\ref{x:cantor-3},
see~\cite[Lemma~5.4]{regfup} for more details. The middle third Cantor set corresponds to $M=3$, $\mathcal A=\{0,2\}$.

The main result of this section is the following discrete version
of FUP:
%%%%%%%%%%%%%%%%%%%%%%%%%%%%%%%%%%%%%%%%%%%%%%%%%%%%%%%%%%%%%%%%%%%%%%%%%%%%%%%%
\begin{theo}
  \label{t:discreteFUP}
Let $\mathcal C_k\subset \mathbb Z_N$, $N=M^k$, be the Cantor set from~\eqref{e:discrete-cantor}
for some choice of $M,\mathcal A$. Define the unitary discrete Fourier transform
\begin{equation}
  \label{e:discrete-fourier}
\mathcal F_N:\mathbb C^N\to\mathbb C^N,\quad
\mathcal F_Nu(j)={1\over\sqrt N}\sum_{\ell=0}^{N-1}\exp\Big(-{2\pi i j\ell\over N}\Big)
u(\ell).
\end{equation}
Let $\delta$ be defined in~\eqref{e:delta-discrete} and assume that $0<\delta<1$.
Then there exist constants
\begin{equation}
  \label{e:fup-discrete-beta}
\beta>\max\Big(0,{1\over 2}-\delta\Big)
\end{equation}
and~$C$, both depending only on $M,\mathcal A$,
such that the set $\mathcal C_k$ satisfies the discrete uncertainty principle
\begin{equation}
  \label{e:fup-discrete}
\|\indic_{\mathcal C_k}\mathcal F_N\indic_{\mathcal C_k}\|_{\mathbb C^N\to\mathbb C^N}
\leq CN^{-\beta}.
\end{equation}
\end{theo}
%%%%%%%%%%%%%%%%%%%%%%%%%%%%%%%%%%%%%%%%%%%%%%%%%%%%%%%%%%%%%%%%%%%%%%%%%%%%%%%%
\begin{figure}
\includegraphics[scale=0.5125]{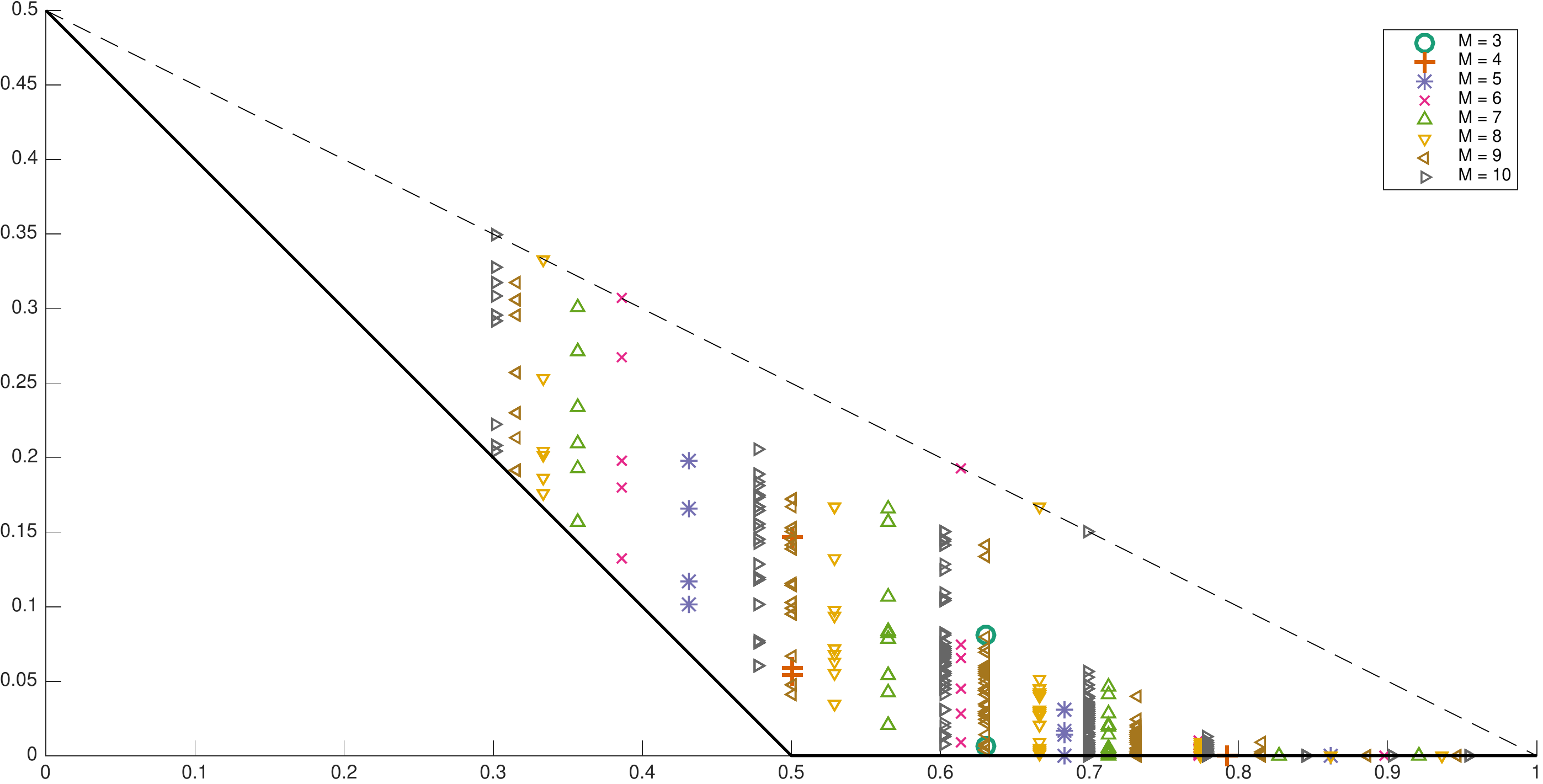}
\caption{Numerically approximated fractal uncertainty exponents for all possible
alphabets with $M\leq 10$ and $0<\delta<1$. Here the horizontal axis represents $\delta$
and the vertical axis represents $\beta$.
The solid black line is $\beta=\max(0,{1\over 2}-\delta)$
and the dashed line is $\beta={1-\delta\over 2}$.
See~\cite[Figure~3]{oqm} for details.}
\label{f:fupcloud}
\end{figure}
%%%%%%%%%%%%%%%%%%%%%%%%%%%%%%%%%%%%%%%%%%%%%%%%%%%%%%%%%%%%%%%%%%%%%%%%%%%%%%%%
\begin{rema}
The discrete uncertainty principle~\eqref{e:fup-discrete}
is related to the continuous uncertainty principle~\eqref{e:fup}
for $X=Y=\mathcal C_\infty+[0,h]$ and $h=(2\pi N)^{-1}$,
see~\cite[Proposition~5.8]{regfup}.
\end{rema}
%%%%%%%%%%%%%%%%%%%%%%%%%%%%%%%%%%%%%%%%%%%%%%%%%%%%%%%%%%%%%%%%%%%%%%%%%%%%%%%%
\begin{rema}
\label{r:fup-discrete-easy}
It is easy to see that~\eqref{e:fup-discrete} holds
with $C=1$ and $\beta=\max(0,{1\over 2}-\delta)$. Indeed,
since $\mathcal F_N$ is unitary, the left-hand side of~\eqref{e:fup-discrete}
is bounded above by~1. On the other hand, denoting by $\|\bullet\|_{\HS}$
the Hilbert--Schmidt norm, we have
\begin{equation}
  \label{e:HS}
\|\indic_{\mathcal C_k}\mathcal F_N\indic_{\mathcal C_k}\|_{\mathbb C^N\to\mathbb C^N}
\leq \|\indic_{\mathcal C_k}\mathcal F_N\indic_{\mathcal C_k}\|_{\HS}
=N^{\delta-1/2}.
\end{equation}
\end{rema}
%%%%%%%%%%%%%%%%%%%%%%%%%%%%%%%%%%%%%%%%%%%%%%%%%%%%%%%%%%%%%%%%%%%%%%%%%%%%%%%%

A natural question to ask is the dependence of the 
largest exponent $\beta$ for which~\eqref{e:fup-discrete}
holds on the alphabet $\mathcal A$. 
This dependence can be quite complicated,
see Figure~\ref{f:fupcloud}. There exist various lower and upper bounds
on $\beta$ depending on $M,\delta$, see~\cite[\S3]{oqm}.
In particular, for each $\delta\in (0,{1\over 2}]$ the improvement
in~\eqref{e:fup-discrete-beta} may be arbitrarily small, namely
there exists
a sequence $(M_j,\mathcal A_j)$ such that the corresponding dimensions
$\delta_j$ converge to $\delta$ and the FUP exponents $\beta_j$ converge
to ${1\over 2}-\delta$~-- see~\cite[Proposition~3.17]{oqm}. For $\delta>{1\over 2}$ numerics suggest
that $\beta$ could be exponentially small in $M$,
supporting the following

%%%%%%%%%%%%%%%%%%%%%%%%%%%%%%%%%%%%%%%%%%%%%%%%%%%%%%%%%%%%%%%%%%%%%%%%%%%%%%%%
\begin{conj}
  \label{c:improve-0}
Fix $\delta\in (1/2,1)$. Then there exists a sequence of pairs $(M_j,\mathcal A_j)$ such that
$$
\delta(M_j,\mathcal A_j)\to\delta,\quad
\beta(M_j,\mathcal A_j)\to 0.
$$
Here $\delta(M,\mathcal A)=\log_M |\mathcal A|$ and
$$
\beta(M,\mathcal A):=-\limsup_{k\to\infty}{\log \|\indic_{\mathcal C_k}\mathcal F_N\indic_{\mathcal C_k}\|_{\mathbb C^N\to\mathbb C^N}\over \log N},\quad
N:=M^k.
$$
\end{conj}
%%%%%%%%%%%%%%%%%%%%%%%%%%%%%%%%%%%%%%%%%%%%%%%%%%%%%%%%%%%%%%%%%%%%%%%%%%%%%%%%
As follows from the above discussion and illustrated by Figure~\ref{f:fupcloud},
we expect that $\beta-\max(0,{1\over 2}-\delta)$ may be very small for some choice
of $M,\mathcal A$. However, the following conjecture states
that if we dilate one of the sets $\mathcal C_k$ by a generic factor, then
FUP holds with a larger value of $\beta$, depending only on the dimension~$\delta$:
%%%%%%%%%%%%%%%%%%%%%%%%%%%%%%%%%%%%%%%%%%%%%%%%%%%%%%%%%%%%%%%%%%%%%%%%%%%%%%%%
\begin{conj}
\label{c:genericdil}
Fix $M,\mathcal A$ with $0<\delta<1$, take $\alpha\in [1,M]$, and consider the dilated
Fourier transform
\begin{equation}
  \label{e:F-N-alpha}
\mathcal F_{N,\alpha}:\mathbb C^N\to\mathbb C^N,\quad
\mathcal F_{N,\alpha}u(j)={1\over\sqrt N}\sum_{\ell=0}^{N-1} \exp\Big(-{2\pi i\alpha j\ell\over N}\Big)u(\ell).
\end{equation}
Show that there exists $\beta>\max(0,{1\over 2}-\delta)$ \textbf{depending only on $\delta$}
such that for a \textbf{generic} choice of $\alpha\in [1,M]$
we have as $k\to\infty$
\begin{equation}
  \label{e:genericdil}
\|\indic_{\mathcal C_k}\mathcal F_{N,\alpha}\indic_{\mathcal C_k}\|_{\mathbb C^N\to\mathbb C^N}=\mathcal O(N^{-\beta}),\quad
N:=M^k.
\end{equation}
\end{conj}
%%%%%%%%%%%%%%%%%%%%%%%%%%%%%%%%%%%%%%%%%%%%%%%%%%%%%%%%%%%%%%%%%%%%%%%%%%%%%%%%
We note that existence of $\beta$ depending on $M,\mathcal A$
follows from the general FUP in Theorems~\ref{t:fup-0}--\ref{t:fup-pres}.
Note also that while we do not in general have $\|\mathcal F_{N,\alpha}\|\leq 1$,
by applying Schur's Lemma to the matrix $\mathcal F_{N,\alpha}^*\mathcal F_{N,\alpha}$
we see that $\|\mathcal F_{N,\alpha}\|\leq C\sqrt{\log N}$.

If true (with a sufficiently good understanding of what it means for $\alpha$ to be generic),
Conjecture~\ref{c:genericdil} is likely to give a spectral gap depending only on $\delta$
for an open quantum baker's map with generic size of the matrix.
We refer the reader to~\cite[\S5]{regfup} for details.
More precisely, if the size of the open quantum baker's map matrix
is given by $\alpha M^k$ where $1\leq\alpha\leq M$, then
the left-hand side of~\cite[(5.8)]{regfup}
is the norm of the matrix
\begin{equation}
  \label{e:matlator}
\bigg({1\over \sqrt{\alpha M^k}}\exp\Big(-{2\pi ib_jb_\ell\over \alpha M^k}\Big)\bigg)_{j,\ell\in\mathcal C_k}
\end{equation}
where $b_j$ is an integer chosen arbitrarily in the interval $\big[\alpha j,\alpha(j+1)\big)$.
If we forget about the requirement that $b_j$ be an integer, then we may take
$b_j:=\alpha j$, in which case the matrix~\eqref{e:matlator} has entries
$\big({1\over\sqrt {\alpha M^k}}e^{-2\pi i \alpha j\ell /M^k}\big)_{j,\ell\in\mathcal C_k}$,
and its norm equals the left-hand side of~\eqref{e:genericdil} up to the constant factor
$\sqrt\alpha$.

One possible approach to Conjecture~\ref{c:genericdil} would be to use
the following corollary of Schur's Lemma (see the proof of~\cite[Lemma~3.8]{oqm})
\begin{equation}
  \label{e:gendale}
\begin{gathered}
\|\indic_{\mathcal C_k}\mathcal F_{N,\alpha}\indic_{\mathcal C_k}\|_{\mathbb C^N\to\mathbb C^N}\leq \tilde r_{k,\alpha}\quad\text{where}\\
\tilde r_{k,\alpha}^2=\max_{j\in\mathcal C_k}\sum_{\ell\in\mathcal C_k}|F_{k,\alpha}(j-\ell)|,\quad
F_{k,\alpha}(j)={1\over N}\sum_{r\in\mathcal C_k}\exp\Big(-{2\pi i \alpha rj\over N}\Big).
\end{gathered}
\end{equation}
For numerical evidence in support of Conjecture~\ref{c:genericdil},
see Figures~\ref{f:genericdil0}--\ref{f:genericdil}.
%%%%%%%%%%%%%%%%%%%%%%%%%%%%%%%%%%%%%%%%%%%%%%%%%%%%%%%%%%%%%%%%%%%%%%%%%%%%%%%%
\begin{figure}
\includegraphics[width=15cm]{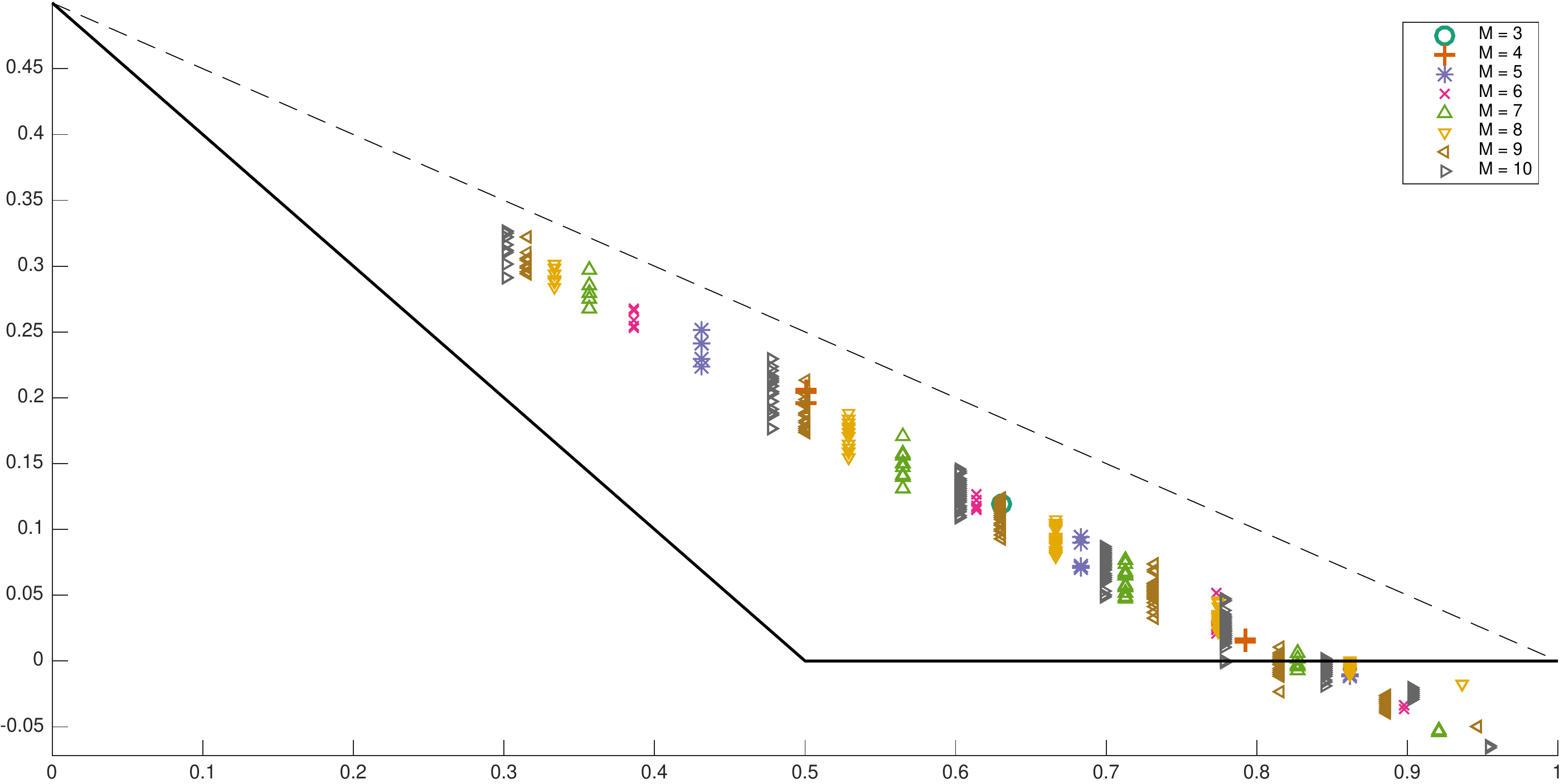}
\caption{Numerically approximated fractal uncertainty exponents for all possible
alphabets with $M\leq 10$ and $0<\delta<1$, using the dilated Fourier transform $\mathcal F_{N,\alpha}$ from~\eqref{e:F-N-alpha} in place of $\mathcal F_N$,
with $\alpha:=(1+\sqrt 5)/2$, taking $k$ large
and assuming constant $1$ in the $\mathcal O(\bullet)$ in~\eqref{e:genericdil}.
We see that the worst FUP exponents are generally larger than
those on Figure~\ref{f:fupcloud}, supporting Conjecture~\ref{c:genericdil}.
This data has to be interpreted with more caution than the one on Figure~\ref{f:fupcloud}.
In particular, some points have $\beta<0$, because the operator $\mathcal F_{N,\alpha}$
no longer has norm bounded by~1.
}
\label{f:genericdil0}
\end{figure}
%%%%%%%%%%%%%%%%%%%%%%%%%%%%%%%%%%%%%%%%%%%%%%%%%%%%%%%%%%%%%%%%%%%%%%%%%%%%%%%%

%%%%%%%%%%%%%%%%%%%%%%%%%%%%%%%%%%%%%%%%%%%%%%%%%%%%%%%%%%%%%%%%%%%%%%%%%%%%%%%%
\begin{figure}
\includegraphics[width=15cm]{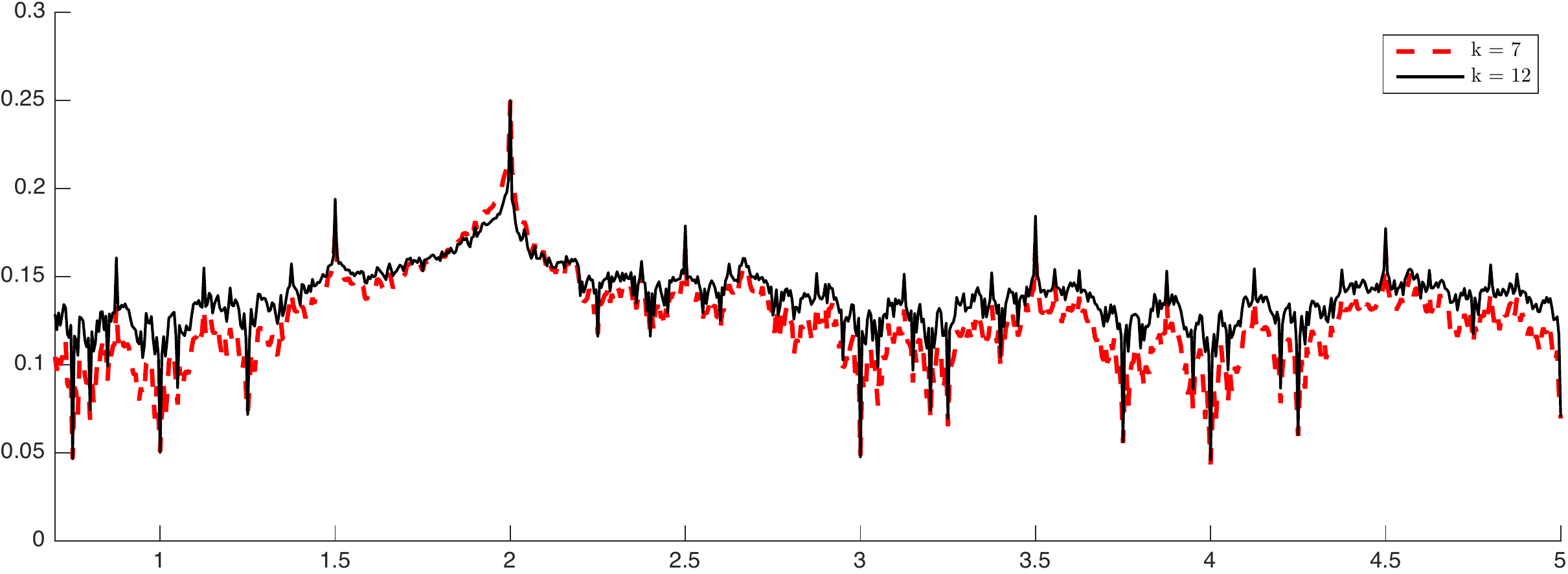}
\caption{Graphs of the quantities $\beta_{k,\alpha}:=-\log \tilde r_{k,\alpha}/\log N$ as functions of~$\alpha$ for two values of~$k$ where
$\tilde r_{k,\alpha}$ is defined in~\eqref{e:gendale}.
The parameters of the Cantor set are $M=4$, $\mathcal A=\{0,1\}$,
and in this case $|F_{k,\alpha}(j)|=2^{-k}\prod_{q=1}^k|\cos(4^{-q}\pi\alpha j)|$. We note the
dips at $\alpha=1,3,4$, with $\beta_{k,\alpha}$ being much larger for generic $\alpha$
than at the dips. For $\alpha=2$ we have $\beta_{k,\alpha}={1\over 4}$
due to the choice of the alphabet, this is similar to the special
alphabets considered in~\cite[\S3.5]{oqm}.
}
\label{f:genericdil}
\end{figure}
%%%%%%%%%%%%%%%%%%%%%%%%%%%%%%%%%%%%%%%%%%%%%%%%%%%%%%%%%%%%%%%%%%%%%%%%%%%%%%%%

%%%%%%%%%%%%%%%%%%%%%%%%%%%%%%%%%%%%%%%%%%%%%%%%%%%%%%%%%%%%%%%%%%%%%%%%%%%%%%%%
\subsection{Proof of discrete FUP}
  \label{s:fup-cantor-proof}

We now give a proof of Theorem~\ref{t:discreteFUP}, following~\cite[\S3]{oqm}.
Compared to the general case the proof is greatly simplified by the following
submultiplicative property which
uses the special structure of Cantor sets:
%%%%%%%%%%%%%%%%%%%%%%%%%%%%%%%%%%%%%%%%%%%%%%%%%%%%%%%%%%%%%%%%%%%%%%%%%%%%%%%%
\begin{lemm}
  \label{l:submultiplicativity}
Put
$$
r_k:=\|\indic_{\mathcal C_k}\mathcal F_N\indic_{\mathcal C_k}\|_{\mathbb C^N\to\mathbb C^N},\quad
N:=M^k.
$$
Then for all $k_1,k_2$ we have
$$
r_{k_1+k_2}\leq r_{k_1}\cdot r_{k_2}.
$$
\end{lemm}
%%%%%%%%%%%%%%%%%%%%%%%%%%%%%%%%%%%%%%%%%%%%%%%%%%%%%%%%%%%%%%%%%%%%%%%%%%%%%%%%
\begin{proof}
Denote
$$
k:=k_1+k_2,\quad
N_j:=M^{k_j},\quad
N:=M^k.
$$
We define the space
$$
\ell^2(\mathcal C_k)=\{u\in \mathbb C^N\mid \supp u\subset \mathcal C_k\}.
$$
Then $r_k$ is the norm of the operator
$$
\mathcal G_k:\ell^2(\mathcal C_k)\to\ell^2(\mathcal C_k),\quad
\mathcal G_k u=\indic_{\mathcal C_k}\mathcal F_Nu.
$$
We will write $\mathcal G_k$ in terms of $\mathcal G_{k_1},\mathcal G_{k_2}$
using a procedure similar to the one used in the Fast Fourier Transform (FFT) algorithm.
Take
$$
u\in\ell^2(\mathcal C_k),\quad
v:=\mathcal G_k u.
$$
We associate to $u,v$ the $|\mathcal A|^{k_1}\times |\mathcal A|^{k_2}$
matrices $U,V$ defined as follows:
$$
\begin{aligned}
U_{ab}&=u(N_1\cdot b + a)\\
V_{ab}&=v(N_2\cdot a + b)
\end{aligned}
$$
for all $a\in\mathcal C_{k_1}$ and $b\in\mathcal C_{k_2}$.
Here we use the fact that
$$
\mathcal C_k=N_2\cdot \mathcal C_{k_1}+\mathcal C_{k_2}
=N_1\cdot \mathcal C_{k_2}+\mathcal C_{k_1}.
$$
Note that the norms of $u,v$ are equal to the Hilbert--Schmidt norms of $U,V$:
$$
\|u\|^2=\sum_{a,b}|U_{ab}|^2,\quad
\|v\|^2=\sum_{a,b}|V_{ab}|^2.
$$
We now write the identity $v=\mathcal G_k u$ in terms of
the matrices $U,V$:
$$
V_{ab}={1\over\sqrt N}\sum_{p\in\mathcal C_{k_1}\atop q\in\mathcal C_{k_2}}
\exp\Big(-{2\pi i (N_2\cdot a+b)(N_1\cdot q+p)\over N}\Big)U_{pq}.
$$
Here is where a small miracle happens: the product of $N_2\cdot a$
and $N_1\cdot q$ is divisible by~$N$, so it can be removed from the exponential.
That is,
$$
V_{ab}={1\over\sqrt N}\sum_{p,q}
\exp\Big(-{2\pi i ap\over N_1}\Big)
\exp\Big(-{2\pi ibp\over N}\Big)
\exp\Big(-{2\pi ibq\over N_2}\Big)
U_{pq}.
$$
It follows that the matrix $V$ can be obtained from $U$ in the following
three steps:
\begin{enumerate}
\item Replace each row of $U$ by its Fourier transform $\mathcal G_{k_2}$,
obtaining the matrix
$$
U'_{pb}={1\over\sqrt{N_2}}\sum_{q\in\mathcal C_{k_2}} \exp\Big(-{2\pi i bq\over N_2}\Big)U_{pq}.
$$
\item Multiply the entries of $U'$ by twist factors, obtaining
the matrix
$$
V'_{pb}=\exp\Big(-{2\pi ibp\over N}\Big)U'_{pb}.
$$
\item Replace each column of $V'$ by its Fourier transform $\mathcal G_{k_1}$,
obtaining the matrix
$$
V_{ab}={1\over\sqrt{N_1}}\sum_{p\in\mathcal C_{k_1}} \exp\Big(-{2\pi iap\over N_1}\Big)V'_{pb}.
$$
\end{enumerate}
Now, we have
$$
\|U'\|_{\HS}\leq r_{k_2}\|U\|_{\HS},\quad
\|V'\|_{\HS}=\|U'\|_{\HS},\quad
\|V\|_{\HS}\leq r_{k_1}\|V'\|_{\HS},
$$
giving
$$
\|v\|\leq r_{k_1}\cdot r_{k_2}\cdot \|u\|
$$
which finishes the proof.
\end{proof}
%%%%%%%%%%%%%%%%%%%%%%%%%%%%%%%%%%%%%%%%%%%%%%%%%%%%%%%%%%%%%%%%%%%%%%%%%%%%%%%%
Given Lemma~\ref{l:submultiplicativity}, we see by Fekete's Lemma that
\begin{equation}
  \label{e:betabest}
\lim_{k\to\infty}{\log r_k\over k\log M}
=\inf_{k\geq 1}{\log r_k\over k\log M}.
\end{equation}
Thus to prove Theorem~\ref{t:discreteFUP} it suffices
to obtain the strict inequality
\begin{equation}
  \label{e:discrete-improve}
r_k:=\|\indic_{\mathcal C_k}\mathcal F_N\indic_{\mathcal C_k}\|_{\mathbb C^N\to\mathbb C^N}
<\min(1,N^{\delta-1/2})
\end{equation}
for just \emph{one} value of $k$. 

The inequality~\eqref{e:discrete-improve} consists of two parts, proved below:
%%%%%%%%%%%%%%%%%%%%%%%%%%%%%%%%%%%%%%%%%%%%%%%%%%%%%%%%%%%%%%%%%%%%%%%%%%%%%%%%
\begin{lemm}
  \label{l:improve-0}
There exists $k$ such that $r_k<1$.  
\end{lemm}
%%%%%%%%%%%%%%%%%%%%%%%%%%%%%%%%%%%%%%%%%%%%%%%%%%%%%%%%%%%%%%%%%%%%%%%%%%%%%%%%
\begin{proof}
Since $\mathcal F_N$ is unitary we have
$r_k\leq 1$. 
We argue by contradiction. Assume that $r_k=1$.
Then there exists
$$
u\in\mathbb C^N\setminus \{0\},\quad
\|\indic_{\mathcal C_k}\mathcal F_N\indic_{\mathcal C_k}u\|=\|u\|.
$$
This implies that
\begin{align}
  \label{e:i0-1}
\supp u&\subset \mathcal C_k,\\
  \label{e:i0-2}
\supp (\mathcal F_N u)&\subset\mathcal C_k.
\end{align}
We now use the fact that discrete Fourier transform evaluates polynomials at roots
of unity. Define the polynomial
$$
p(z):=\sum_{\ell=0}^{N-1} u(\ell)z^\ell.
$$
Then
$$
\mathcal F_Nu(j)={1\over\sqrt{N}}\,p(e^{-2\pi i j/N}).
$$
By~\eqref{e:i0-2} for each $j\in \mathbb Z_N\setminus \mathcal C_k$
we have $\mathcal F_Nu(j)=0$.
It follows that the number of roots of $p$ is bounded below by
(here we use that $\delta<1$)
$$
N-|\mathcal C_k|\geq M^k-(M-1)^k.
$$
On the other hand, the set $\mathbb Z_N\setminus\mathcal C_k$
contains $M^{k-1}$ consecutive numbers
(specifically $aM^{k-1},\dots,(a+1)M^{k-1}-1$ where
$a\in\mathbb Z_M\setminus\mathcal A$; this corresponds
to porosity). We shift $\mathcal C_k$ circularly
(which does not change the norm $r_k$) to map these numbers
to $(M-1)M^{k-1},\dots,M^k-1$. Then
the degree of $p$ is smaller than $(M-1)M^{k-1}$.

Now, for $k$ large enough we have
$$
M^k-(M-1)^k\geq (M-1)M^{k-1}.
$$
Then the number of roots of $p$ is larger than its degree, giving
a contradiction.
\end{proof}
%%%%%%%%%%%%%%%%%%%%%%%%%%%%%%%%%%%%%%%%%%%%%%%%%%%%%%%%%%%%%%%%%%%%%%%%%%%%%%%%

%%%%%%%%%%%%%%%%%%%%%%%%%%%%%%%%%%%%%%%%%%%%%%%%%%%%%%%%%%%%%%%%%%%%%%%%%%%%%%%%
\begin{lemm}
  \label{l:improve-1}
For $k\geq 2$ we have $r_k<N^{\delta-1/2}$.  
\end{lemm}
%%%%%%%%%%%%%%%%%%%%%%%%%%%%%%%%%%%%%%%%%%%%%%%%%%%%%%%%%%%%%%%%%%%%%%%%%%%%%%%%
\begin{proof}
Recall from~\eqref{e:HS} that
$N^{\delta-1/2}$ is the Hilbert--Schmidt norm of $\indic_{\mathcal C_k}\mathcal F_N\indic_{\mathcal C_k}$,
while $r_k$ is its operator norm. We again argue by contradiction, assuming
that $r_k=N^{\delta-1/2}$. Then $\indic_{\mathcal C_k}\mathcal F_N\indic_{\mathcal C_k}$
is a rank 1 operator; indeed, the sum of the squares of its singular
values is equal to the square of the maximal singular value.
It follows that each rank~2 minor of $\indic_{\mathcal C_k}\mathcal F_N\indic_{\mathcal C_k}$ is equal to zero,
namely
$$
\det\begin{pmatrix}
e^{-2\pi ij\ell/N} & e^{-2\pi i j\ell'/N} \\
e^{-2\pi ij'\ell/N} & e^{-2\pi i j'\ell'/N}
\end{pmatrix}
=0\quad\text{for all }j,j',\ell,\ell'\in\mathcal C_k.
$$
Computing the determinant we see that
$$
(j-j')(\ell-\ell')\in N\mathbb Z\quad\text{for all }j,j',\ell,\ell'\in\mathcal C_k.
$$
However, if $k\geq 2$ we may take $j=\ell,j'=\ell'\in\mathcal C_k$ such that (here we use that $\delta>0$)
$$
0<|j-j'|<M\leq \sqrt N,
$$
giving a contradiction.
\end{proof}
%%%%%%%%%%%%%%%%%%%%%%%%%%%%%%%%%%%%%%%%%%%%%%%%%%%%%%%%%%%%%%%%%%%%%%%%%%%%%%%%

%%%%%%%%%%%%%%%%%%%%%%%%%%%%%%%%%%%%%%%%%%%%%%%%%%%%%%%%%%%%%%%%%%%%%%%%%%%%%%%%
%%%%%%%%%%%%%%%%%%%%%%%%%%%%%%%%%%%%%%%%%%%%%%%%%%%%%%%%%%%%%%%%%%%%%%%%%%%%%%%%
\section{Relation to Fourier decay and additive energy}
\label{s:relator}

We now explain how a fractal uncertainty principle can be proved
if we have a Fourier decay bound or an additive energy bound on one of the sets $X,Y$.
While this does not give new results (compared to Theorems~\ref{t:fup-0}--\ref{t:fup-pres})
in the general setting, it leads to improvements in special cases.

Let $X,Y\subset [0,1]$ be two $h$-dependent closed sets which are
$\delta$-regular on scales $h$ to~1 with some $h$-independent regularity constant $C_R$
(see Definition~\ref{d:delta-regular}). In particular by~\eqref{e:vol-x} we have
\begin{equation}
  \label{e:vol-x-0}
\vol(X),\vol(Y)\leq Ch^{1-\delta}.
\end{equation}

To estimate the norm on the left-hand side of the uncertainty principle~\eqref{e:fup},
we use the $T^*T$ argument:
$$
\|\indic_X\mathcal F_h\indic_Y\|_{L^2\to L^2}^2=
\|\indic_Y\mathcal F_h^*\indic_X\mathcal F_h\indic_Y\|_{L^2\to L^2}.
$$
We write $\mathcal F_h^*\indic_X\mathcal F_h$ as an integral operator:
$$
\mathcal F_h^*\indic_X\mathcal F_h f(y)=\int_{\mathbb R} \mathcal K_X(y-y')f(y')\,dy'
$$
where
\begin{equation}
  \label{e:K-X-def}
\mathcal K_X(y)=(2\pi h)^{-1}\int_X e^{ixy/h}\,dx.
\end{equation}
Note that $\mathcal K_X(y)$ is just the rescaled Fourier transform of the
indicator function of~$X$.

By Schur's inequality applied to $\mathcal K_X(y-y')$ we see that
\begin{equation}
  \label{e:fup-schur}
\|\indic_X\mathcal F_h\indic_Y\|_{L^2\to L^2}^2\leq
\sup_{y'\in Y}\int_Y |\mathcal K_X(y-y')|\,dy.
\end{equation}
If we combine this with the basic bound (which follows from~\eqref{e:vol-x-0}) 
\begin{equation}
  \label{e:ft-basic-sup}
\sup|\mathcal K_X|\leq Ch^{-\delta}
\end{equation}
then we recover the bound~\eqref{e:fup}
with the standard exponent $\beta={1\over 2}-\delta$.

We now explore two possible conditions on $X$ where~\eqref{e:fup-schur} gives
an uncertainty principle with $\beta>\max(0,{1\over 2}-\delta)$.

%%%%%%%%%%%%%%%%%%%%%%%%%%%%%%%%%%%%%%%%%%%%%%%%%%%%%%%%%%%%%%%%%%%%%%%%%%%%%%%%
\subsection{Fourier decay}
  \label{s:fdec}

We first impose the condition that the Fourier transform $\mathcal K_X$ has a decay bound,
with the baseline given by the upper bound~\eqref{e:ft-basic-sup}:
namely for some $\beta_F>0$
\begin{equation}
  \label{e:fourier-dec}
|\mathcal K_X(y)|=\mathcal O(h^{\beta_F/2-\delta}|y|^{-\beta_F/2})\quad\text{for }h\leq |y|\leq 2.
\end{equation}
If we assume that $\vol(X)\sim h^{1-\delta}$, then this is 
equivalent to the Fourier transform $\widehat{\mu_X}(\xi)$ of the natural probability
measure $\mu_X(U)=\vol(X\cap U)/\vol(X)$ having $\mathcal O(|\xi|^{-\beta_F/2})$ decay
for $|\xi|\lesssim h^{-1}$. This is the finite scale version of requiring that
\emph{$X$ has Fourier dimension at least $\beta_F$}. In particular, it is natural
to assume that $\beta_F\leq \delta$ (since the Fourier dimension of a set is always bounded above by its Hausdorff dimension).

%%%%%%%%%%%%%%%%%%%%%%%%%%%%%%%%%%%%%%%%%%%%%%%%%%%%%%%%%%%%%%%%%%%%%%%%%%%%%%%%
\begin{prop}
  \label{l:fourier-dec-fup}
Assume that~\eqref{e:fourier-dec} holds for some $\beta_F\in (0,\delta]$. Then
the uncertainty principle~\eqref{e:fup} holds with
\begin{equation}
  \label{e:FUP-Fourier}
\beta={1\over 2}-\delta+{\beta_F\over 4}.
\end{equation}
\end{prop}
%%%%%%%%%%%%%%%%%%%%%%%%%%%%%%%%%%%%%%%%%%%%%%%%%%%%%%%%%%%%%%%%%%%%%%%%%%%%%%%%
\begin{proof}
From $\delta$-regularity of $Y$ (similarly to~\eqref{e:vol-x}) we see that for any interval
$I$ with $h\leq |I|\leq 2$
$$
\vol(Y\cap I)\leq C h^{1-\delta} |I|^\delta.
$$
Breaking the integral below into dyadic pieces centered
at $y'$, we see that~\eqref{e:fourier-dec} implies
$$
\sup_{y'\in Y}\int_Y |\mathcal K_X(y-y')|\,dy\leq C h^{1-2\delta+\beta_F/2}.
$$
It remains to apply~\eqref{e:fup-schur}.
\end{proof}
%%%%%%%%%%%%%%%%%%%%%%%%%%%%%%%%%%%%%%%%%%%%%%%%%%%%%%%%%%%%%%%%%%%%%%%%%%%%%%%%
In general we do not have the Fourier decay property~\eqref{e:fourier-dec}.
For example, if $X$ is the $h/2$-neighborhood of the middle third Cantor set
where $h:=3^{-k}$ and $\delta=\log_32$
then an explicit computation shows that
$$
|\mathcal K_X(y)|=h^{-\delta}\bigg|{\sin(2y)\over 2\pi y}\bigg|\cdot\prod_{j=1}^{k-1}|\cos(3^jy)|.
$$
For $y:=2\pi\cdot 3^{-\ell}$ where $\ell\in [1,k-1]$ is an integer, we
have $|\mathcal K_X(y)|\sim h^{-\delta}$, contradicting~\eqref{e:fourier-dec}
for any $\beta_F>0$.

However, if $X$ is a Schottky limit set then
we have the following Fourier decay statement whose proof uses sum-product inequalities
and the nonlinear structure of the transformations generating~$X$:
%%%%%%%%%%%%%%%%%%%%%%%%%%%%%%%%%%%%%%%%%%%%%%%%%%%%%%%%%%%%%%%%%%%%%%%%%%%%%%%%
\begin{theo}\cite[Theorem~2]{hyperfup}
  \label{t:fdec}
Assume that $\Lambda_\Gamma\subset\mathbb R$ is a Schottky limit set of dimension $\delta>0$
and $\mu$ is the Patterson--Sullivan measure on $\Lambda_\Gamma$
(see~\S\ref{s:schottky}). Then there exists $\beta_F>0$ depending only on~$\delta$
such that for each phase function $\varphi\in C^2(\mathbb R;\mathbb R)$ with $\varphi'>0$ everywhere
and each amplitude $a\in C^1(\mathbb R)$ we have the generalized Fourier decay bound
(with $C$ depending on $\Gamma,\varphi,a$)
\begin{equation}
  \label{e:fdec}
\bigg|\int_{\Lambda_\Gamma}e^{i\xi\varphi(x)}a(x)\,d\mu(x)\bigg|\leq C|\xi|^{-\beta_F/2}\quad\text{for all $\xi$,}\quad
|\xi|\geq 1.
\end{equation}
\end{theo}
%%%%%%%%%%%%%%%%%%%%%%%%%%%%%%%%%%%%%%%%%%%%%%%%%%%%%%%%%%%%%%%%%%%%%%%%%%%%%%%%
See Figure~\ref{f:fourier} for numerical evidence supporting Theorem~\ref{t:fdec}.
Fourier decay statements similar to~\eqref{e:fdec} have been obtained
for Gibbs measures for the Gauss map by Jordan--Sahlsten~\cite{JordanSahlsten},
for limit sets of sufficiently nonlinear iterated function systems
by Sahlsten--Stevens~\cite{SahlstenStevens},
and in some higher dimensional cases by Li~\cite{LiJialun} and
Li--Naud--Pan~\cite{LiNaudPan}.

Arguing similarly to Proposition~\ref{l:fourier-dec-fup} we obtain the generalized FUP~\eqref{e:hfup}
for $X=Y=\Lambda_\Gamma(h)$ with the exponent~\eqref{e:FUP-Fourier}. Combining this with Theorem~\ref{t:fup-open-rel}
we obtain the following application to spectral gaps of convex co-compact hyperbolic surfaces
which uses that the exponent in Theorem~\ref{t:fdec} depends only on~$\delta$:
%%%%%%%%%%%%%%%%%%%%%%%%%%%%%%%%%%%%%%%%%%%%%%%%%%%%%%%%%%%%%%%%%%%%%%%%%%%%%%%%
\begin{figure}
\includegraphics[width=15cm]{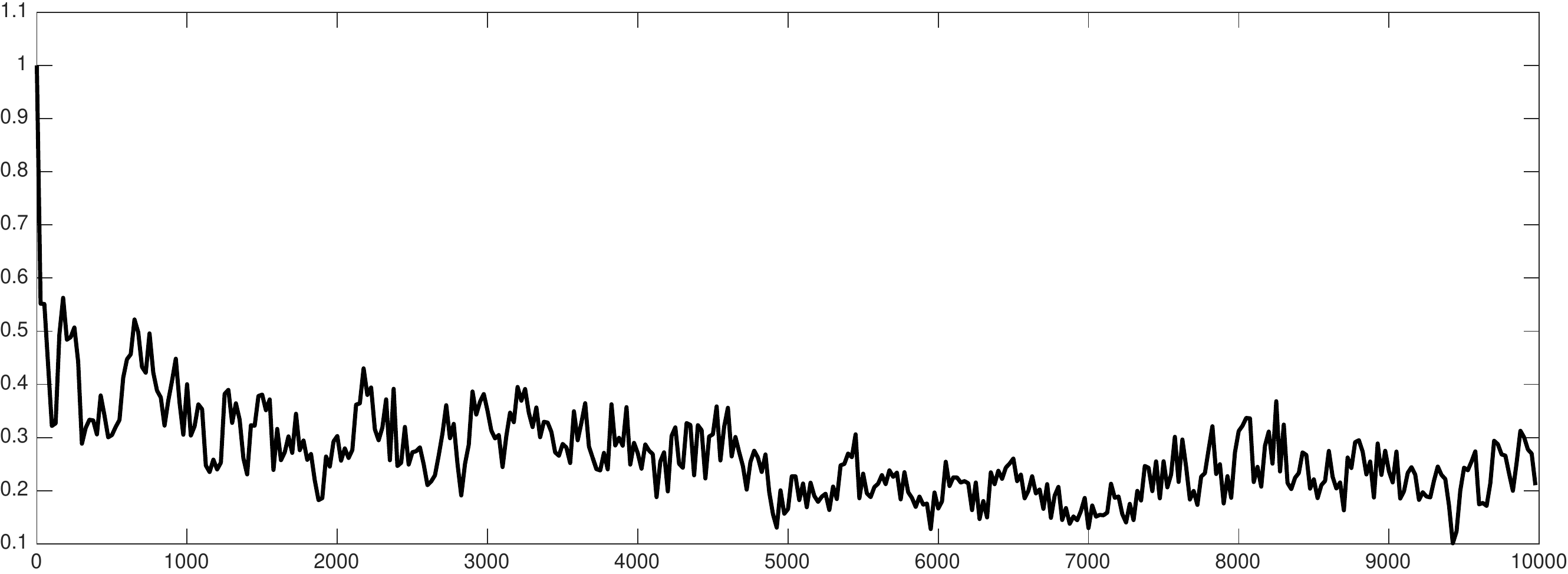}
\includegraphics[width=15cm]{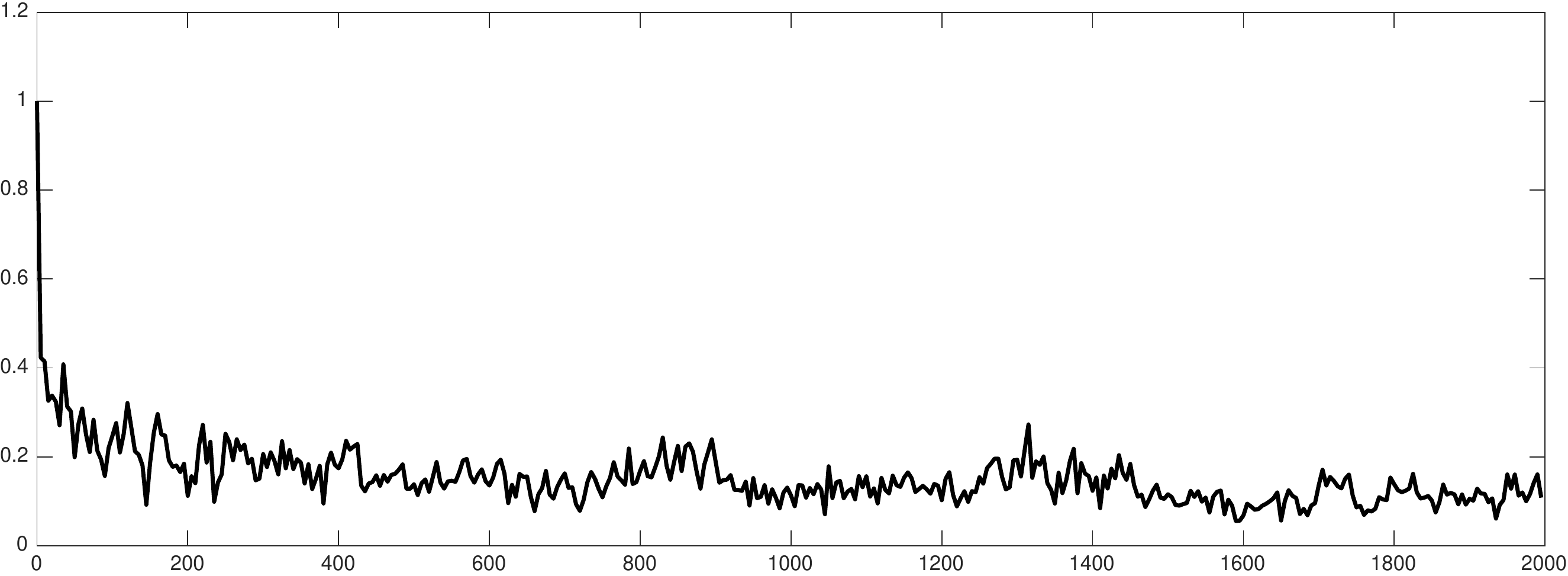}
\includegraphics[width=15cm]{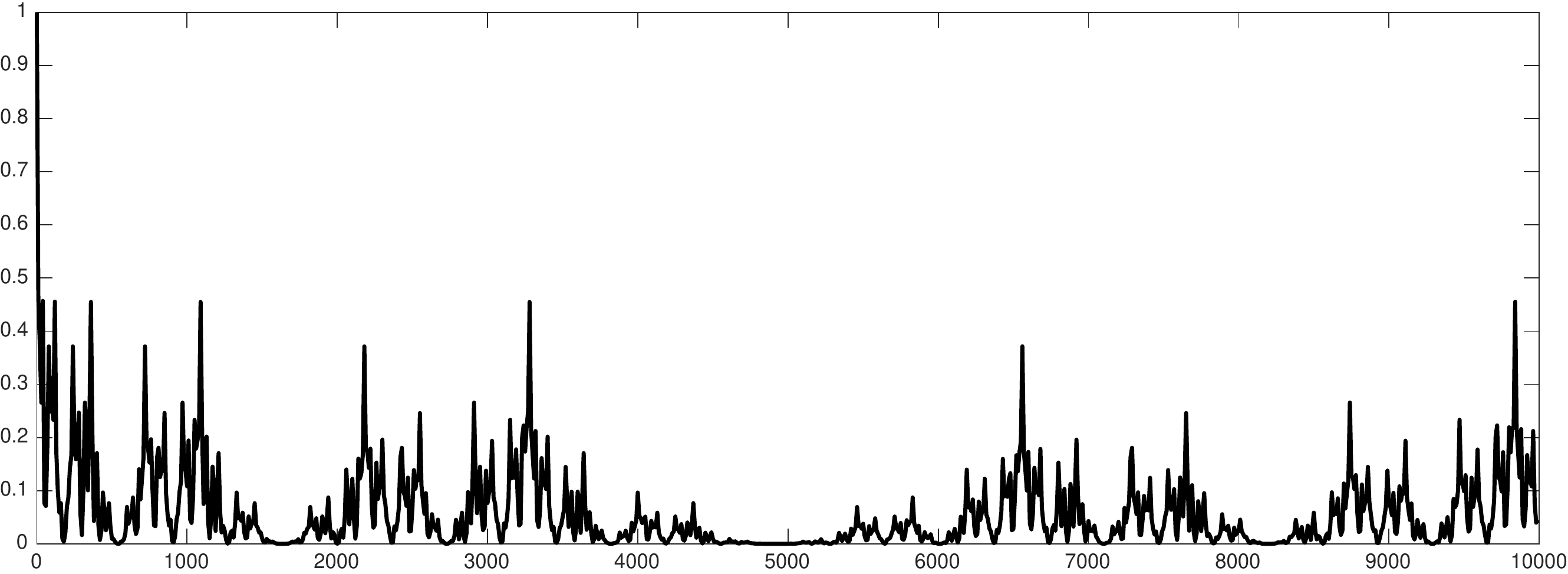}
\caption{Top two graphs: the absolute value of the Fourier transform $|\widehat\mu(\xi)|$
of the Patterson--Sullivan measure on a Schottky limit set. The top graph corresponds to the
same Schottky data as Figure~\ref{f:sch1}, with $\delta\approx 0.31038$.
The second graph corresponds to the three-funnel surface
with neck lengths $2,3,3$, with $\delta\approx 0.46932$. Both go to~0 polynomially fast
as $\xi\to\infty$ by Theorem~\ref{t:fdec} though this decay is not evident on the plots.
The bottom graph is the Fourier transform of the natural probability
measure on the middle third Cantor set, which does not go to~0. Since the Fourier transform is highly
oscillating we plot here its envelope, with each point corresponding
to the maximum of $|\hat\mu|$ over an interval of size 25 (top), 5 (middle),
and 10 (bottom).
}
\label{f:fourier}
\end{figure}
%%%%%%%%%%%%%%%%%%%%%%%%%%%%%%%%%%%%%%%%%%%%%%%%%%%%%%%%%%%%%%%%%%%%%%%%%%%%%%%%
\begin{theo}\cite[Theorem~1]{hyperfup}
\label{t:fdec-used}
Let $M$ be a convex co-compact hyperbolic surface with $\delta>0$. Then $M$ has an essential
spectral gap of size ${1\over 2}-\delta+\varepsilon(\delta)$ in the sense of Definition~\ref{d:spectral-gap},
where $\varepsilon(\delta)>0$ depends only on~$\delta$.
\end{theo}
%%%%%%%%%%%%%%%%%%%%%%%%%%%%%%%%%%%%%%%%%%%%%%%%%%%%%%%%%%%%%%%%%%%%%%%%%%%%%%%%

%%%%%%%%%%%%%%%%%%%%%%%%%%%%%%%%%%%%%%%%%%%%%%%%%%%%%%%%%%%%%%%%%%%%%%%%%%%%%%%%
\subsection{Additive energy}
We now give an FUP which follows from an improved additive energy bound on the
set~$X$. As before we assume here that $X,Y\subset [0,1]$
are $\delta$-regular on scales $h$ to~1.
We define additive energy as
\begin{equation}
  \label{e:ae-def}
\mathcal E_A(X) = \vol\{(x_1,x_2,x_3,x_4)\in X^4\mid x_1+x_2=x_3+x_4\}
\end{equation}
where we use the volume form on the hypersurface $\{x_1+x_2=x_3+x_4\}\subset\mathbb R^4$
induced by the standard volume form in the $(x_1,x_2,x_3)$ variables.
It follows immediately from~\eqref{e:vol-x-0} that
\begin{equation}
  \label{e:ae-trivial}
\mathcal E_A(X)\leq \vol(X)^3\leq C h^{3(1-\delta)}.
\end{equation}
%%%%%%%%%%%%%%%%%%%%%%%%%%%%%%%%%%%%%%%%%%%%%%%%%%%%%%%%%%%%%%%%%%%%%%%%%%%%%%%%
\begin{prop}
  \label{l:ae-used}
Assume that $X$ satisfies the improved additive energy bound for some $\beta_A>0$,
\begin{equation}
  \label{e:ae-improved}
\mathcal E_A(X)\leq Ch^{3(1-\delta)+\beta_A}.
\end{equation}
Then the fractal uncertainty principle~\eqref{e:fup}
holds with
\begin{equation}
  \label{e:FUP-AE}
\beta={3\over 4}\Big({1\over 2}-\delta\Big)+{\beta_A\over 8}.
\end{equation}
\end{prop}
%%%%%%%%%%%%%%%%%%%%%%%%%%%%%%%%%%%%%%%%%%%%%%%%%%%%%%%%%%%%%%%%%%%%%%%%%%%%%%%%
\begin{proof}
By H\"older's inequality using the volume bound~\eqref{e:vol-x-0} on $Y$ we get
\begin{equation}
  \label{e:ae-imp-1}
\sup_{y'\in Y}\int_Y |\mathcal K_X(y-y')|\,dy\leq C h^{{3\over 4}(1-\delta)} \|\mathcal K_X\|_{L^4(\mathbb R)}.
\end{equation}
Recalling the definition~\eqref{e:K-X-def} of~$\mathcal K_X$ we compute
\begin{equation}
  \label{e:ae-imp-2}
\|\mathcal K_X\|^4_{L^4(\mathbb R)}=(2\pi h)^{-3}\mathcal E_A(X)\leq Ch^{-3\delta+\beta_A}.
\end{equation}
It remains to combine~\eqref{e:ae-imp-1} and~\eqref{e:ae-imp-2} with~\eqref{e:fup-schur}.
\end{proof}
%%%%%%%%%%%%%%%%%%%%%%%%%%%%%%%%%%%%%%%%%%%%%%%%%%%%%%%%%%%%%%%%%%%%%%%%%%%%%%%%
The next result shows that $\delta$-regular sets with $0<\delta<1$ satisfy
an improved additive energy bound:
%%%%%%%%%%%%%%%%%%%%%%%%%%%%%%%%%%%%%%%%%%%%%%%%%%%%%%%%%%%%%%%%%%%%%%%%%%%%%%%%
\begin{theo}\cite[\S6, Theorem~6]{hgap}
  \label{t:ae-improved}
Assume that $X\subset [0,1]$ is $\delta$-regular on scales $h$ to 1 with constant
$C_R$, and $0<\delta<1$. Then there exists $\beta_A=\beta_A(\delta,C_R)>0$ such that~\eqref{e:ae-improved} holds.
\end{theo}
%%%%%%%%%%%%%%%%%%%%%%%%%%%%%%%%%%%%%%%%%%%%%%%%%%%%%%%%%%%%%%%%%%%%%%%%%%%%%%%%
Note that~\cite[Theorem~6]{hgap} was formulated in slightly different
terms (similar to~\eqref{e:ae-schottky} below), see~\cite[\S7.2, proof of Theorem~5]{hgap}
for the version which uses~\eqref{e:ae-def}.
%%%%%%%%%%%%%%%%%%%%%%%%%%%%%%%%%%%%%%%%%%%%%%%%%%%%%%%%%%%%%%%%%%%%%%%%%%%%%%%%
\begin{exam}
\label{x:cantor-3-ae}
If $X$ is the $h$-neighborhood of the middle third Cantor set (see Example~\ref{x:cantor-3})
then~\eqref{e:ae-improved} holds with $\beta_A=\log_3(4/3)$, as follows
by an application of~\cite[Lemma~3.10]{oqm}.
\end{exam}
%%%%%%%%%%%%%%%%%%%%%%%%%%%%%%%%%%%%%%%%%%%%%%%%%%%%%%%%%%%%%%%%%%%%%%%%%%%%%%%%
We remark that if $X$ is sufficiently large (i.e. it contains a disjoint union of $\sim h^{-\delta}$
many intervals of length $h$ each) then any $\beta_A$ in~\eqref{e:ae-improved} has to satisfy
\begin{equation}
  \label{e:ae-max}
\beta_A\leq\min(\delta,1-\delta).
\end{equation}
See~\cite[(3.20)--(3.21)]{oqm} for a proof in the closely related discrete case.

Combining Theorem~\ref{t:ae-improved} with Proposition~\ref{l:ae-used}, we recover
the fractal uncertainty principle of Theorems~\ref{t:fup-0}--\ref{t:fup-pres} when
$\delta={1\over 2}$. On the other hand, when $\delta$ is far from $1\over 2$,
the exponent~\eqref{e:FUP-AE} does not improve over the standard exponent
$\beta=\max(0,{1\over 2}-\delta)$. In particular, even the best possible
additive energy improvement~\eqref{e:ae-max} does not give an improved FUP exponent
when $\delta\notin ({1\over 3},{4\over 7})$.

A similar statement is true for the
fractal uncertainty principle with a general phase~\eqref{e:hfup}, with the exponent in~\eqref{e:FUP-AE} divided by~2 (due to the fact that we use the argument presented at the end of~\S\ref{s:hfup}).
Also, the set $X$ should be replaced in~\eqref{e:ae-def} by its images under
certain diffeomorphisms determined by the phase. We refer the reader to~\cite[Theorem~4]{hgap}
and Conjecture~\ref{c:ae-schottky} below for details in the case of the hyperbolic FUP used in Theorem~\ref{t:fup-open-rel}.

%%%%%%%%%%%%%%%%%%%%%%%%%%%%%%%%%%%%%%%%%%%%%%%%%%%%%%%%%%%%%%%%%%%%%%%%%%%%%%%%
\begin{figure}
\includegraphics[width=7.5cm]{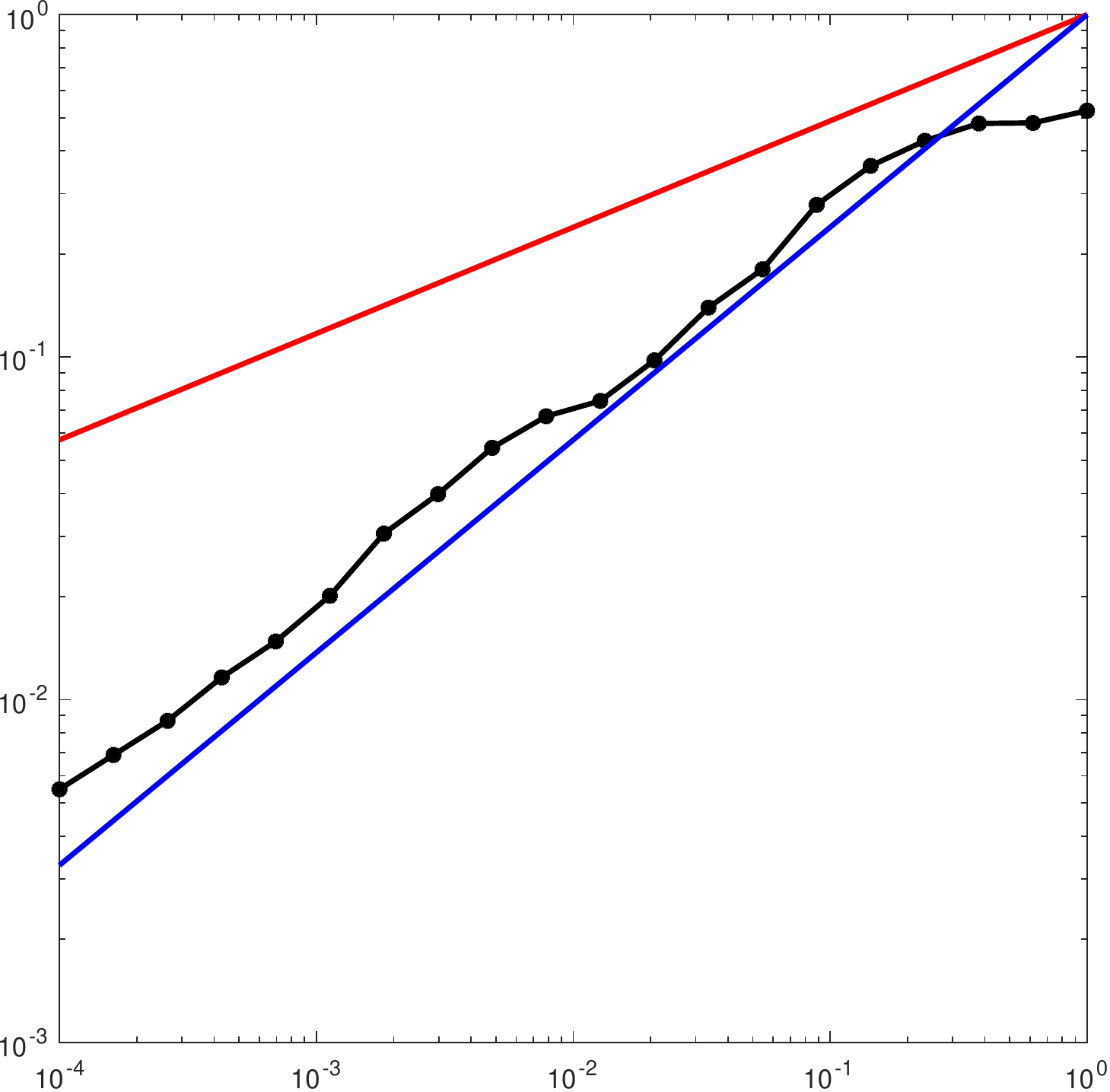}\quad
\includegraphics[width=7.5cm]{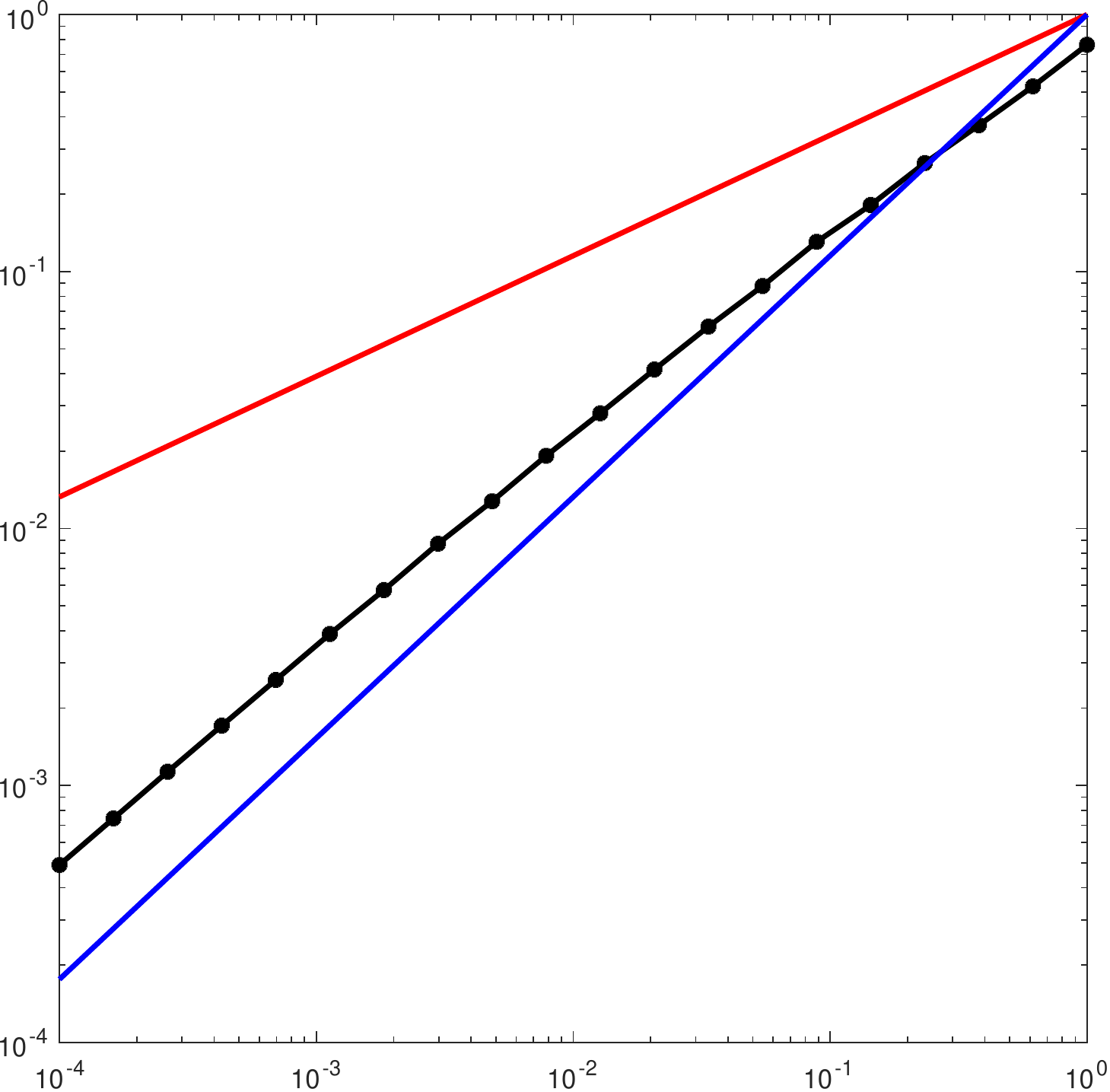}
\caption{Log-log plot of the Patterson--Sullivan additive energy~\eqref{e:ae-schottky}
as a function of~$h$ for some choice of $y\in\Lambda_\Gamma$ and two
different Schottky limit sets:
the Schottky data used in Figure~\ref{f:sch1}, with $\delta\approx 0.31038$
(left), and the three-funnel surface
with neck lengths $2,3,3$, with $\delta\approx 0.46932$ (right).
The (red) higher straight line corresponds to the asymptotic upper bound $h^{\delta}$
and the (blue) lower straight line, to the bound $h^{2\delta}$ of Conjecture~\ref{c:ae-schottky}.}
\label{f:ae}
\end{figure}
%%%%%%%%%%%%%%%%%%%%%%%%%%%%%%%%%%%%%%%%%%%%%%%%%%%%%%%%%%%%%%%%%%%%%%%%%%%%%%%%

While arguments based on additive energy only give an FUP when $\delta\approx {1\over 2}$,
for these values of~$\delta$
they may give a larger FUP exponent than other techniques. For instance,
for discrete Cantor sets considered in~\S\ref{s:cantor}, using additive energy
one can show an FUP with $\beta\sim 1/\log M$ where $M$ is the base of the Cantor set
in the special case $\delta\approx {1\over 2}$ (see~\cite[Proposition~3.12]{oqm}) while
the improvements $\beta-\max(0,{1\over 2}-\delta)$ obtained using
other methods decay polynomially ($\delta<{1\over 2}$) or exponentially ($\delta>{1\over 2}$) in $M$ (see~\cite[Corollaries~3.5 and~3.7]{oqm}). The better improvement for $\delta\approx {1\over 2}$ is
visible in the numerics on Figure~\ref{f:fupcloud}.

For the case of convex co-compact hyperbolic surfaces, numerics on Figure~\ref{f:gaps-hyp} also show
a better improvement in the size of the spectral gap when $\delta\approx {1\over 2}$.
This could be explained if one were to prove the following
%%%%%%%%%%%%%%%%%%%%%%%%%%%%%%%%%%%%%%%%%%%%%%%%%%%%%%%%%%%%%%%%%%%%%%%%%%%%%%%%
\begin{conj}
  \label{c:ae-schottky}
Let $\Lambda_\Gamma\subset\mathbb R$ be a Schottky limit set and $\mu$ the Patterson--Sullivan measure on~$\Lambda_\Gamma$, see~\S\ref{s:schottky}. Then $\Lambda_\Gamma$
has an additive energy estimate with improvement $\min(\delta,1-\delta)-$
in the following sense: for each $\varepsilon>0$ there exists $C_\varepsilon$ such that
for all $y\in \Lambda_\Gamma$ and all $h\in (0,1]$
\begin{equation}
  \label{e:ae-schottky}
\begin{gathered}
\mu^4\big\{ (x_1,x_2,x_3,x_4)\in(\Lambda_\Gamma\setminus I_w)^4\colon
|\gamma_y(x_1)+\gamma_y(x_2)-\gamma_y(x_3)-\gamma_y(x_4)|\leq h\big\}
\\\leq C_\varepsilon h^{\delta+\min(\delta,1-\delta)-\varepsilon}.
\end{gathered}
\end{equation}
Here $w\in\{1,\dots,2r\}$ is chosen so that $y$ is contained in the Schottky interval
$I_w$ (see~\S\ref{s:schottky}) and $\gamma_y(x)$ is the stereographic projection
of $x$ centered at~$y$, defined by
$$
\gamma_y(x)={1+xy\over y-x}.
$$
\end{conj}
%%%%%%%%%%%%%%%%%%%%%%%%%%%%%%%%%%%%%%%%%%%%%%%%%%%%%%%%%%%%%%%%%%%%%%%%%%%%%%%%
Note that similarly to~\eqref{e:ae-trivial},
the left-hand side of~\eqref{e:ae-schottky} is trivially bounded above by $Ch^\delta$.
For an explanation for why the transformation $\gamma_y$ appears,
see~\cite[Definition~1.3]{hgap}.
For the relation of the Patterson--Sullivan additive energy in~\eqref{e:ae-schottky}
to the additive energy defined in~\eqref{e:ae-def} see~\cite[\S7.2, proof of Theorem~5]{hgap}.
Numerical evidence in support of Conjecture~\ref{c:ae-schottky} is given on Figure~\ref{f:ae}.

%%%%%%%%%%%%%%%%%%%%%%%%%%%%%%%%%%%%%%%%%%%%%%%%%%%%%%%%%%%%%%%%%%%%%%%%%%%%%%%%
%%%%%%%%%%%%%%%%%%%%%%%%%%%%%%%%%%%%%%%%%%%%%%%%%%%%%%%%%%%%%%%%%%%%%%%%%%%%%%%%
\section{FUP in higher dimensions}
\label{s:high-dim}

We finally discuss generalizations of Theorems~\ref{t:fup-0}--\ref{t:fup-pres},\ref{t:hfup-0}--\ref{t:hfup-pres} to fractal sets in~$\mathbb R^n$.
The case of $n\geq 2$ is currently not well-understood, with the known results
not general enough to be able to extend the applications (Theorems~\ref{t:appl-eig}--\ref{t:appl-dwe},\ref{t:fup-open}) from the setting of surfaces to the case of higher dimensional manifolds.
We discuss both the general FUP and the two-dimensional version of FUP for discrete Cantor sets
(see~\S\ref{s:cantor}), presenting the known results and formulating several open problems.

%%%%%%%%%%%%%%%%%%%%%%%%%%%%%%%%%%%%%%%%%%%%%%%%%%%%%%%%%%%%%%%%%%%%%%%%%%%%%%%%
\subsection{The continuous case}
  \label{s:hd-cont}

We first extend the definitions of uncertainty principle and fractal set to
the case of higher dimensions. The unitary semiclassical Fourier transform
on $\mathbb R^n$ is defined by the following generalization of~\eqref{e:F-h}:
\begin{equation}
  \label{e:F-h-hd}
\mathcal F_hf(\xi)=(2\pi h)^{-n/2}\int_{\mathbb R^n}e^{-i\langle x,\xi\rangle/h}f(x)\,dx.
\end{equation}
The notion of a $\delta$-regular subset of $\mathbb R^n$, where $\delta\in [0,n]$, is introduced similarly to Definition~\ref{d:delta-regular}, where we replace intervals in $\mathbb R$
with balls in~$\mathbb R^n$ and the length of an interval by the diameter of a ball.
Similarly to Definition~\ref{d:porous} we define what it means for a
subset of $\mathbb R^n$ to be $\nu$-porous. Regular and porous sets are related
by the following analogue of Proposition~\ref{l:regular-porous}:
porous sets are subsets of $\delta$-regular sets with $\delta<n$.

The higher dimensional version
of the question stated in the beginning of~\S\ref{s:fup-statement}
is as follows: given $\delta,C_R$, what is the largest value of $\beta$ such that
\begin{equation}
  \label{e:hd-fup}
\|\indic_X\mathcal F_h\indic_Y\|_{L^2(\mathbb R^n)\to L^2(\mathbb R^n)}=\mathcal O(h^\beta)\quad\text{as }h\to 0
\end{equation}
for all $h$-dependent sets $X,Y\subset B_{\mathbb R^n}(0,1)$ which are $\delta$-regular on scales $h$ to~1?

Similarly to~\eqref{e:easy-fup} we see that~\eqref{e:hd-fup} holds with
the basic FUP exponent
\begin{equation}
  \label{e:easy-fup-hd}
\beta_0=\max\Big(0,{n\over 2}-\delta\Big).
\end{equation}
Unfortunately, in dimensions $n\geq 2$ one cannot obtain an uncertainty principle~\eqref{e:hd-fup} with an exponent larger than~\eqref{e:easy-fup-hd} in the entire range
$\delta\in (0,n)$.
In dimension~2 this is illustrated by the following example,
taking $X,Y$ to be $h$-neighborhoods of two orthogonal line segments:
%%%%%%%%%%%%%%%%%%%%%%%%%%%%%%%%%%%%%%%%%%%%%%%%%%%%%%%%%%%%%%%%%%%%%%%%%%%%%%%%
\begin{exam}
  \label{x:cross}
Let $n=2$, $X=[0,1]\times [0,h]$, $Y=[0,h]\times[0,1]$. Then
$X,Y\subset\mathbb R^2$ are $1$-regular on scales $h$ to~1 with constant~10,
and they are ${1\over 10}$-porous on scales $10h$ to~$\infty$.
However, we have
$$
\|\indic_X \mathcal F_h\indic_Y\|_{L^2(\mathbb R^2)\to L^2(\mathbb R^2)}\sim 1\quad\text{as }h\to 0,
$$
as can be verified by applying the operator on the left-hand side to the function
$f(x_1,x_2)=\chi(x_1/h)\chi(x_2)$ where $\chi\in \CIc((0,1))$.
\end{exam}
%%%%%%%%%%%%%%%%%%%%%%%%%%%%%%%%%%%%%%%%%%%%%%%%%%%%%%%%%%%%%%%%%%%%%%%%%%%%%%%%
The currently known statements on FUP in higher dimensions thus make strong assumptions
on the structure of one or both sets. For simplicity we present those in dimension~2
and restrict ourselves to obtaining exponent $\beta>0$ for porous sets
(similarly to Theorem~\ref{t:fup-porous}):
%%%%%%%%%%%%%%%%%%%%%%%%%%%%%%%%%%%%%%%%%%%%%%%%%%%%%%%%%%%%%%%%%%%%%%%%%%%%%%%%
\begin{itemize}
\item Define $\pi_1(x_1,x_2)=x_1$, $\pi_2(x_1,x_2)=x_2$. If
both the projection $\pi_1(X)$ and each intersection $Y_y=Y\cap \pi_2^{-1}(y)\subset\mathbb R$,
$y\in\mathbb R$, are $\nu$-porous,
then~\eqref{e:hd-fup} holds with some $\beta=\beta(\nu)>0$.
Indeed, denote $X_1:=\pi_1(X)$; we may assume that $X=X_1\times \mathbb R$.
Denote by $\mathcal F_h^{(1)}$ and~$\mathcal F_h^{(2)}$ the unitary semiclassical Fourier transforms
in the first and the second variable respectively, then
$$
\begin{aligned}
\|\indic_X\mathcal F_h\indic_Y\|&=\|\indic_{X_1\times\mathbb R}\mathcal F_h^{(2)}\mathcal F_h^{(1)}\indic_Y\|=
\|\mathcal F_h^{(2)}\indic_{X_1\times\mathbb R}\mathcal F_h^{(1)}\indic_Y\|\\
&=\|\indic_{X_1\times\mathbb R}\mathcal F_h^{(1)}\indic_Y\|
\leq\sup_y\|\indic_{X_1}\mathcal F_h\indic_{Y_y}\|_{L^2(\mathbb R)\to L^2(\mathbb R)}\leq Ch^\beta
\end{aligned}
$$
where the last inequality follows from Theorem~\ref{t:fup-porous}.
\item A much more involved result is that if $X$ is $\nu$-porous
and both projections
$\pi_1(Y),\pi_2(Y)$ are $\nu$-porous
then~\eqref{e:hd-fup} holds with some $\beta=\beta(\nu)>0$.
A more general version of this statement was proved by Han--Schlag~\cite[Theorem~1.2]{HanSchlag}.
\end{itemize}
%%%%%%%%%%%%%%%%%%%%%%%%%%%%%%%%%%%%%%%%%%%%%%%%%%%%%%%%%%%%%%%%%%%%%%%%%%%%%%%%
See Propositions~\ref{l:hds-1}--\ref{l:hds-2} below for analogues of the above two statements
for discrete Cantor sets.

Similarly to~\S\ref{s:hfup} we may generalize~\eqref{e:hd-fup} to the estimate
\begin{equation}
  \label{e:hd-hfup}
\|\indic_X\mathcal B_h\indic_Y\|_{L^2(\mathbb R^n)\to L^2(\mathbb R^n)}=\mathcal O(h^\beta)\quad\text{as }h\to 0
\end{equation}
featuring a Fourier integral operator $\mathcal B_h$ defined by
$$
\mathcal B_h f(x)=(2\pi h)^{-n/2}\int_{\mathbb R^n} e^{i\Phi(x,y)/h}b(x,y)f(y)\,dy
$$
where $\Phi\in C^\infty(U;\mathbb R)$ is a phase function satisfying the nondegeneracy
condition
\begin{equation}
  \label{e:Phi-nondeg-hd}
\det(\partial^2_{x_jy_k}\Phi(x,y))_{j,k=1}^n\neq 0\quad\text{for all }(x,y)\in U,
\end{equation}
$U\subset\mathbb R^{2n}$ is an open set, and $b\in \CIc(U)$. In generalizations
of applications in~\S\ref{s:appl} (replacing hyperbolic surfaces with higher dimensional
hyperbolic manifolds) one would use~\eqref{e:hd-hfup} with the phase
\begin{equation}
  \label{e:hd-hPhi}
\Phi(x,y)=\log|x-y|,\quad
U=\{(x,y)\in\mathbb R^{2n}\mid x\neq y\}
\end{equation}
where $|x-y|$ denotes the Euclidean distance between $x,y\in\mathbb R^n$.

One can reduce the generalized FUP~\eqref{e:hd-hfup} to the FUP for Fourier transform,
\eqref{e:hd-fup}, similarly to the argument at the end of~\S\ref{s:hfup}.
However, in higher dimensions this reduction might be disadvantageous.
In fact, Example~\ref{x:cross} cannot be generalized to the phase~\eqref{e:hd-hPhi},
prompting the following
%%%%%%%%%%%%%%%%%%%%%%%%%%%%%%%%%%%%%%%%%%%%%%%%%%%%%%%%%%%%%%%%%%%%%%%%%%%%%%%%
\begin{conj}
\label{c:fup-curved}
For each $\nu>0$ there exists $\beta=\beta(\nu)>0$ such that the generalized FUP~\eqref{e:hd-hfup}
holds for each $X,Y\subset B_{\mathbb R^2}(0,1)$ which are $\nu$-porous on scales $h$ to~1
and each $b\in \CIc(U)$, assuming that the phase $\Phi$ is given by~\eqref{e:hd-hPhi}.
\end{conj}
%%%%%%%%%%%%%%%%%%%%%%%%%%%%%%%%%%%%%%%%%%%%%%%%%%%%%%%%%%%%%%%%%%%%%%%%%%%%%%%%
If proved, Conjecture~\ref{c:fup-curved} would be a key component in generalizing
the applications of FUP to hyperbolic surfaces (Theorems~\ref{t:appl-eig}--\ref{t:appl-dwe},
\ref{t:fup-open}) to the setting of higher dimensional hyperbolic manifolds.

%%%%%%%%%%%%%%%%%%%%%%%%%%%%%%%%%%%%%%%%%%%%%%%%%%%%%%%%%%%%%%%%%%%%%%%%%%%%%%%%
\subsection{The discrete setting}
\label{s:hd-cantor}

We now discuss a two-dimensional generalization of FUP for discrete Cantor sets
presented in~\S\ref{s:cantor}. We fix
\begin{itemize}
\item the base $M\geq 2$,
\item and two nonempty alphabets $\mathcal A,\mathcal B\subset \mathbb Z_M^2$
where $\mathbb Z_M:=\{0,\dots,M-1\}$.
\end{itemize}
For $k\geq 1$, define the Cantor set
$$
\mathcal C_{k,\mathcal A}:=\{a_0+Ma_1+\dots+M^{k-1}a_{k-1}\mid a_0,\dots,a_{k-1}\in \mathcal A\}\subset\mathbb Z_N^2,\quad
N:=M^k
$$
and similarly define $\mathcal C_{k,\mathcal B}$. The corresponding dimensions are
$$
\delta_A = {\log |\mathcal A|\over\log M},\quad
\delta_B = {\log |\mathcal B|\over\log M}.
$$
The two-dimensional unitary discrete Fourier transform is given by $\mathcal F_{N\times N}:\mathbb C^{N\times N}\to\mathbb C^{N\times N}$
where $\mathbb C^{N\times N}$ is the space of $N\times N$ matrices%
\footnote{One should think of these matrices as rank~2 tensors rather
than as operators.}
with the $\ell^2$ norm on the entries (i.e. the Hilbert--Schmidt norm)
and
$$
\mathcal F_{N\times N}u(j)={1\over N}\sum_{\ell\in\mathbb Z_N^2} \exp\Big(-{2\pi i \langle j,\ell\rangle\over N}\Big) u(\ell).
$$
The fractal uncertainty principle for two-dimensional discrete Cantor sets then takes the form
\begin{equation}
  \label{e:hd-fup-discrete}
\|\indic_{\mathcal C_{k,\mathcal A}}\mathcal F_{N\times N}\indic_{\mathcal C_{k,\mathcal B}}\|_{\mathbb C^{N\times N}\to\mathbb C^{N\times N}}
\leq CN^{-\beta}.
\end{equation}
Similarly to Remark~\ref{r:fup-discrete-easy} by using the unitarity of the Fourier transform and bounding the operator
norm in~\eqref{e:hd-fup-discrete} by the Hilbert--Schmidt norm we get~\eqref{e:hd-fup-discrete} with $C=1$ and
\begin{equation}
  \label{e:hd-discrete-fup-easy}
\beta=\max\Big(0,1-{\delta_A+\delta_B\over 2}\Big).
\end{equation}
The question is then:
$$
\begin{gathered}
\text{For which alphabets $\mathcal A,\mathcal B$ does the estimate~\eqref{e:hd-fup-discrete} hold}\\
\text{with some }\beta>\max\Big(0,1-{\delta_A+\delta_B\over 2}\Big)\,?
\end{gathered}
$$
We henceforth assume that $\delta_A,\delta_B\in (0,2)$ since otherwise~\eqref{e:hd-discrete-fup-easy} is sharp.
Unlike the one-dimensional case discussed in~\S\ref{s:cantor}, there exist other situations where~\eqref{e:hd-discrete-fup-easy} is sharp,
similarly to Example~\ref{x:cross}:
%%%%%%%%%%%%%%%%%%%%%%%%%%%%%%%%%%%%%%%%%%%%%%%%%%%%%%%%%%%%%%%%%%%%%%%%%%%%%%%%
\begin{exam}
  \label{x:cross-discrete}
1. Assume that $\mathcal A\supset \mathcal A_0$ and $\mathcal B\supset\mathcal B_0$
where
$$
\mathcal A_0:=\{0,\dots,M-1\}\times \{0\},\quad
\mathcal B_0:=\{0\}\times\{0,\dots,M-1\}.
$$
Then the norm in~\eqref{e:hd-fup-discrete} is equal to~1.
(Note that $\delta_A+\delta_B\geq 2$ in this case.)

2. Assume that $\mathcal A\subset\mathcal A_0$ and $\mathcal B\subset \mathcal B_0$.
Then the norm in~\eqref{e:hd-fup-discrete} is equal to
$N^{-\beta}$ with $\beta=1-{\delta_A+\delta_B\over 2}$.
(Note that $\delta_A+\delta_B\leq 2$ in this case.)
\end{exam}
%%%%%%%%%%%%%%%%%%%%%%%%%%%%%%%%%%%%%%%%%%%%%%%%%%%%%%%%%%%%%%%%%%%%%%%%%%%%%%%%
Similarly to Lemma~\ref{l:submultiplicativity} we have a submultiplicativity property:
%%%%%%%%%%%%%%%%%%%%%%%%%%%%%%%%%%%%%%%%%%%%%%%%%%%%%%%%%%%%%%%%%%%%%%%%%%%%%%%%
\begin{lemm}
  \label{l:hd-submul}
Put
$$
r_k:=\|\indic_{\mathcal C_{k,\mathcal A}}\mathcal F_{N\times N}\indic_{\mathcal C_{k,\mathcal B}}\|_{\mathbb C^{N\times N}\to\mathbb C^{N\times N}}.
$$
Then for all $k_1,k_2$ we have
$$
r_{k_1+k_2}\leq r_{k_1}\cdot r_{k_2}.
$$
\end{lemm}
%%%%%%%%%%%%%%%%%%%%%%%%%%%%%%%%%%%%%%%%%%%%%%%%%%%%%%%%%%%%%%%%%%%%%%%%%%%%%%%%
Using this and arguing similarly to Lemma~\ref{l:improve-1} we obtain a condition under which one can prove~\eqref{e:hd-fup-discrete}
with $\beta>1-{\delta_A+\delta_B\over 2}$:
%%%%%%%%%%%%%%%%%%%%%%%%%%%%%%%%%%%%%%%%%%%%%%%%%%%%%%%%%%%%%%%%%%%%%%%%%%%%%%%%
\begin{prop}
  \label{l:hd-fup-disc-1}
Assume that there exist
$$
a,a'\in\mathcal A,\quad
b,b'\in\mathcal B,\quad
\langle a-a',b-b'\rangle\neq 0.
$$
(Here the inner product is an element of $\mathbb Z$ rather than $\mathbb Z/N\mathbb Z$.)
Then~\eqref{e:hd-fup-discrete} holds for some $\beta>1-{\delta_A+\delta_B\over 2}$.
\end{prop}
%%%%%%%%%%%%%%%%%%%%%%%%%%%%%%%%%%%%%%%%%%%%%%%%%%%%%%%%%%%%%%%%%%%%%%%%%%%%%%%%
\begin{rema}
If the condition of Proposition~\ref{l:hd-fup-disc-1} fails, then
we have $\langle j-j',\ell-\ell'\rangle=0$ for all $j,j'\in\mathcal C_{k,\mathcal A}$
and $\ell,\ell'\in\mathcal C_{k,\mathcal B}$, in which case it is easy
to check that the left-hand side of~\eqref{e:hd-fup-discrete}
is equal to $N^{-\beta}$ where $\beta=1-{\delta_A+\delta_B\over 2}\geq 0$.
\end{rema}
%%%%%%%%%%%%%%%%%%%%%%%%%%%%%%%%%%%%%%%%%%%%%%%%%%%%%%%%%%%%%%%%%%%%%%%%%%%%%%%%
On the other hand, there is no known criterion for when~\eqref{e:hd-fup-discrete} holds with some $\beta>0$.
We make the following
%%%%%%%%%%%%%%%%%%%%%%%%%%%%%%%%%%%%%%%%%%%%%%%%%%%%%%%%%%%%%%%%%%%%%%%%%%%%%%%%
\begin{figure}
\includegraphics[width=6cm]{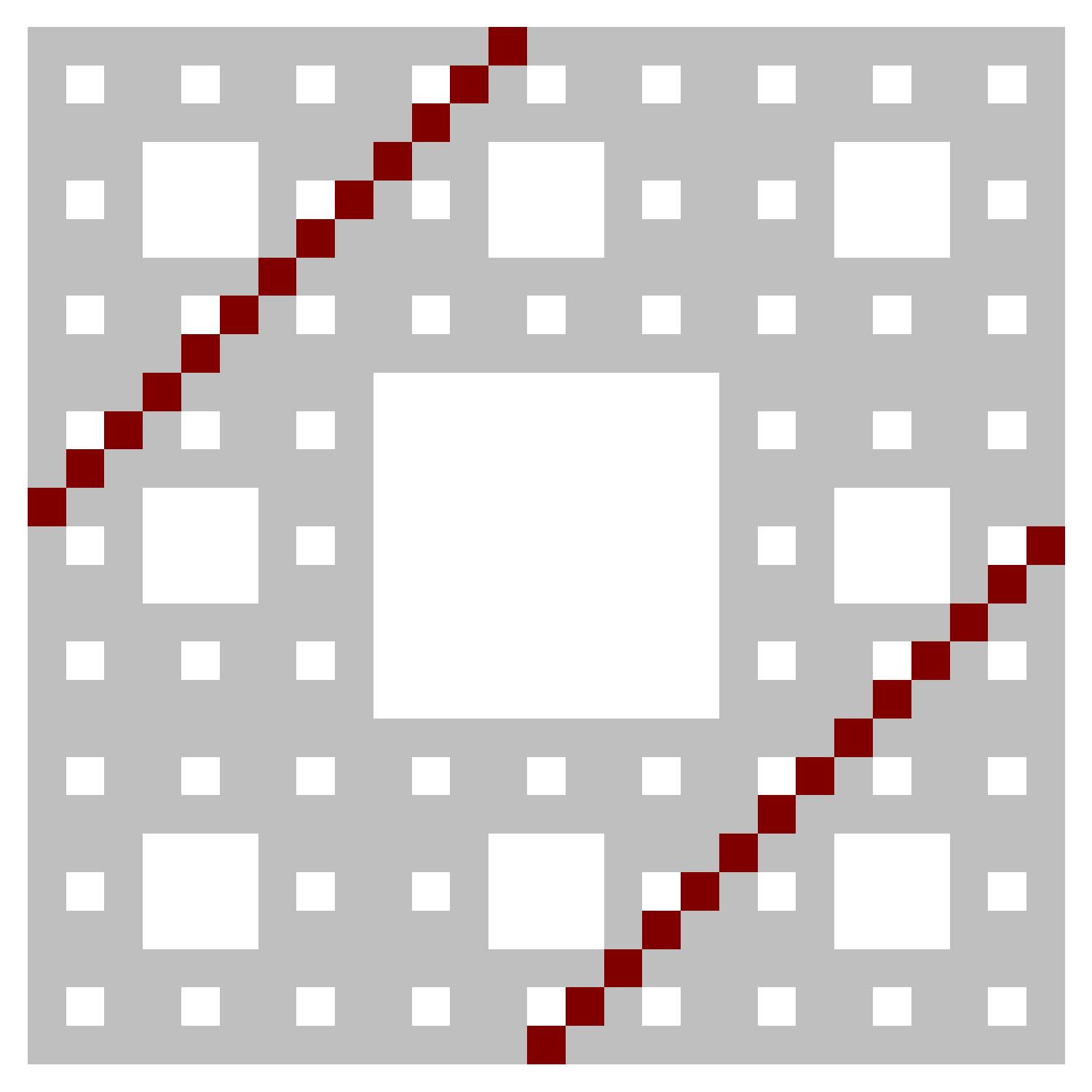}
\caption{A discrete Sierpi\'nski carpet $\mathcal C_{k,\mathcal A}$, with
$M:=3$, $\mathcal A:=\{0,1,2\}^2\setminus \{(1,1)\}$; in the picture we have $k=3$.
The dark red line is the diagonal $\{(j,j+{3^k-1\over 2})\mid j=0,\dots,3^k-1\}$
contained in $\mathcal C_{k,\mathcal A}$.}
\label{f:sierpinski}
\end{figure}
%%%%%%%%%%%%%%%%%%%%%%%%%%%%%%%%%%%%%%%%%%%%%%%%%%%%%%%%%%%%%%%%%%%%%%%%%%%%%%%%
%
%%%%%%%%%%%%%%%%%%%%%%%%%%%%%%%%%%%%%%%%%%%%%%%%%%%%%%%%%%%%%%%%%%%%%%%%%%%%%%%%
\begin{conj}
  \label{c:hd-fup-disc}
The bound~\eqref{e:hd-fup-discrete} holds with some $\beta>0$
(by Lemma~\ref{l:hd-submul} this is the same as saying that
the left-hand side of~\eqref{e:hd-fup-discrete} is $<1$ for some
value of~$k$) unless one of the following situations happens:
\begin{enumerate}
\item one of the sets $\mathcal A,\mathcal B$ contains a horizontal line
and the other set contains a vertical line, or
\item for each $k$, one of the sets $\mathcal C_{k,\mathcal A},\mathcal C_{k,\mathcal B}$ contains a diagonal line and the other set contains an antidiagonal line.
\end{enumerate}
\end{conj}
%%%%%%%%%%%%%%%%%%%%%%%%%%%%%%%%%%%%%%%%%%%%%%%%%%%%%%%%%%%%%%%%%%%%%%%%%%%%%%%%
Here a horizontal line in $\mathbb Z_M^2$ is defined as a set
of the form $\{(j,s)\mid j\in\mathbb Z_M\}$ for some $s\in \mathbb Z_M$;
a vertical line is defined similarly, replacing $(j,s)$ with $(s,j)$.
A diagonal line in $\mathbb Z_N^2$ is defined as a set of the form
$\{(j,(j+s)\bmod N)\mid j\in\mathbb Z_N\}$ for some $s\in\mathbb Z_N$;
an antidiagonal line is defined similarly, replacing $j+s$ by $s-j$.
If either case~(1) or case~(2) above hold, then one can show
that the norm in~\eqref{e:hd-fup-discrete} is equal to~1.
We note that the case~(2) in Conjecture~\ref{c:hd-fup-disc} can
arise in a non-obvious way, see Figure~\ref{f:sierpinski}.

We finish this section with two conditions under which~\eqref{e:hd-fup-discrete}
is known to hold with some $\beta>0$. The first one says that the complement of $\mathcal A$ contains
a vertical line, while $\mathcal B$ contains no horizontal line:
%%%%%%%%%%%%%%%%%%%%%%%%%%%%%%%%%%%%%%%%%%%%%%%%%%%%%%%%%%%%%%%%%%%%%%%%%%%%%%%%
\begin{prop}
  \label{l:hds-1}
Assume that (see Figure~\ref{f:hds-1}):
\begin{enumerate}
\item there exists $s\in\mathbb Z_M$ such that $(s,j)\notin\mathcal A$
for all $j\in\mathbb Z_M$, and
\item for each $t\in\mathbb Z_M$, the set $\{\ell\mid (\ell,t)\in \mathcal B\}$
is not equal to the entire $\mathbb Z_M$.
\end{enumerate}
Then~\eqref{e:hd-fup-discrete} holds with some $\beta>0$.
\end{prop}
%%%%%%%%%%%%%%%%%%%%%%%%%%%%%%%%%%%%%%%%%%%%%%%%%%%%%%%%%%%%%%%%%%%%%%%%%%%%%%%%
\begin{figure}
\includegraphics[width=5cm]{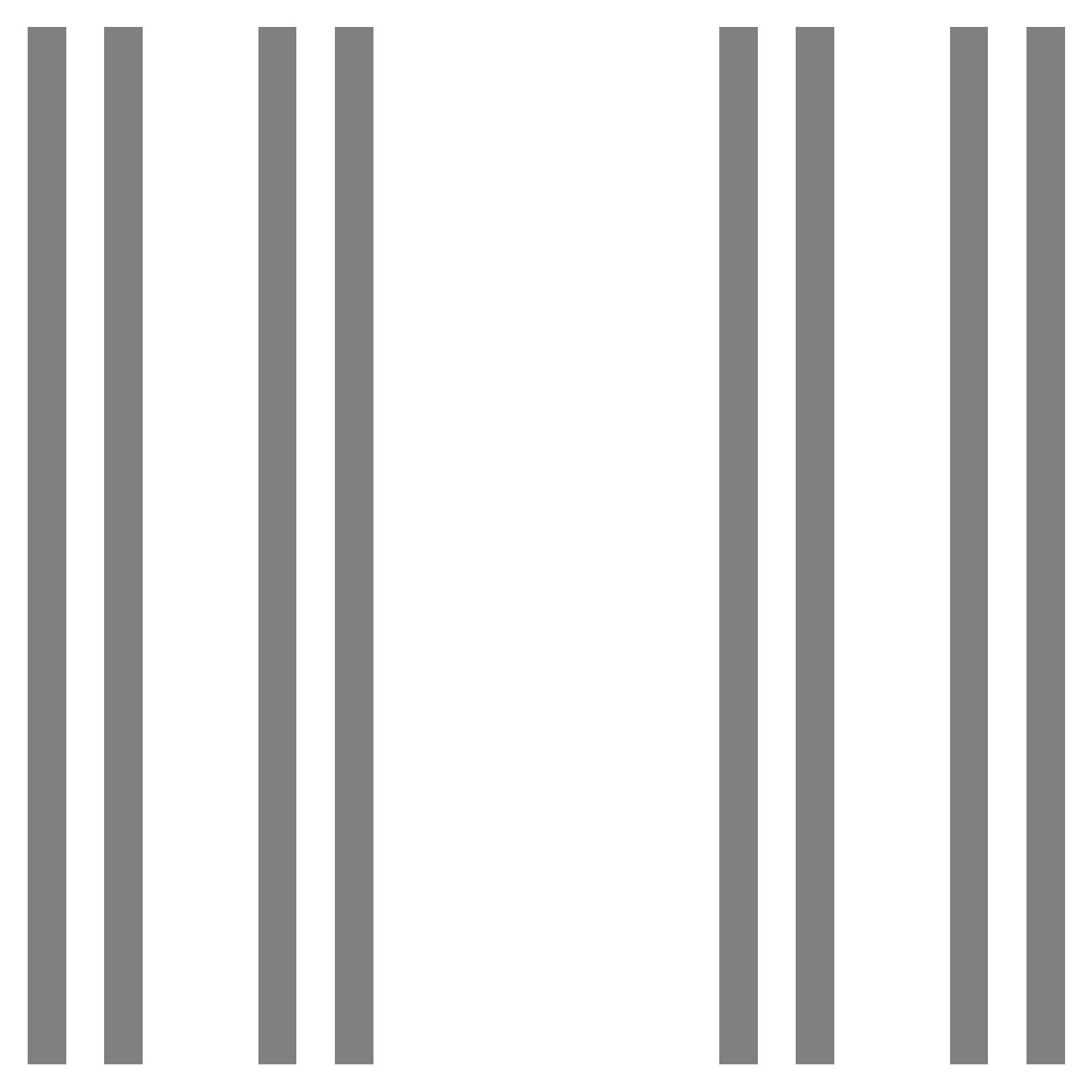}
\qquad\qquad
\includegraphics[width=5cm]{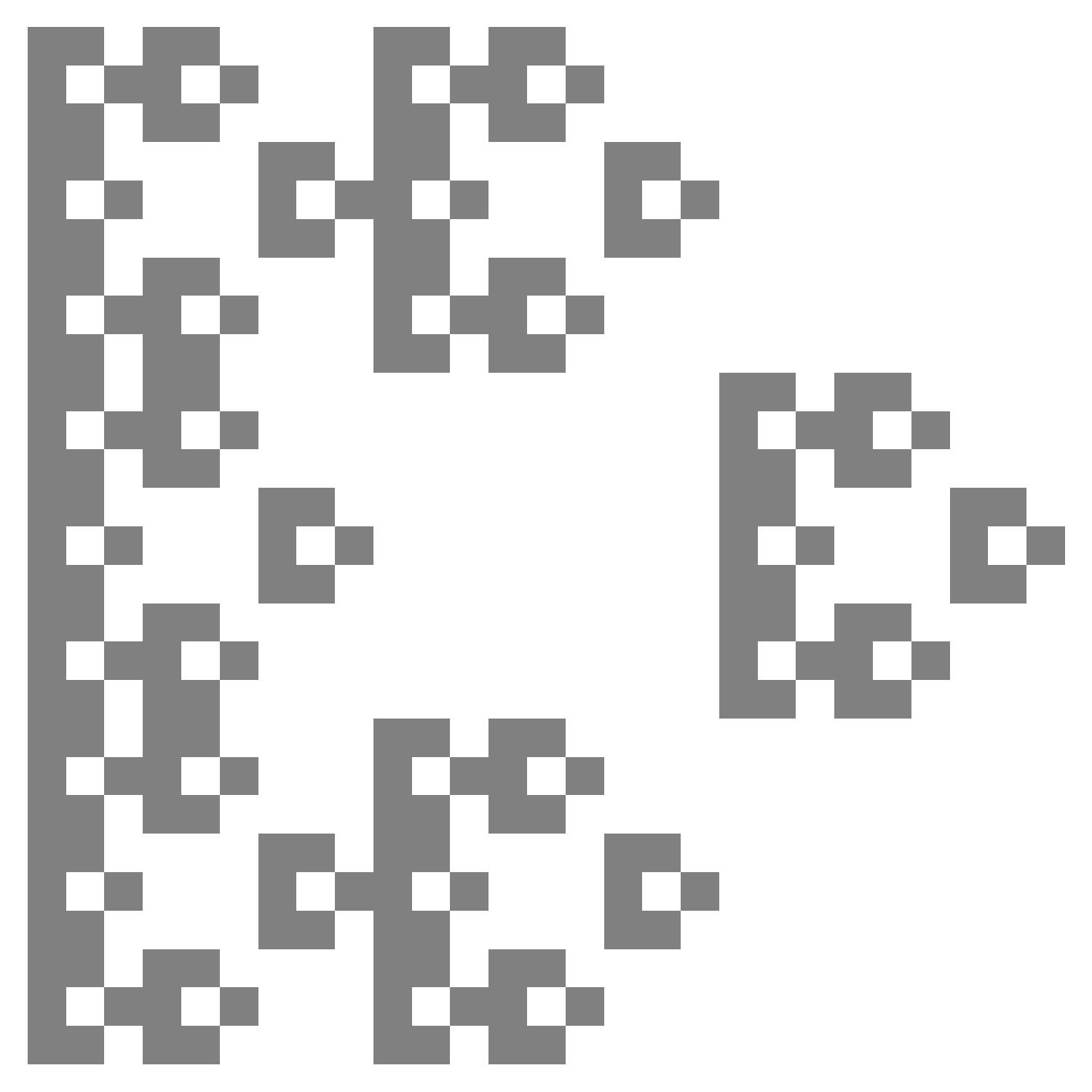}
\hbox to\hsize{\hss $\mathcal C_{k,\mathcal A}$ \hss\qquad\qquad $\mathcal C_{k,\mathcal B}$ \hss}
\caption{An example of Cantor sets with $M=3$, $k=3$ satisfying the conditions
of Proposition~\ref{l:hds-1}.}
\label{f:hds-1}
\end{figure}
%%%%%%%%%%%%%%%%%%%%%%%%%%%%%%%%%%%%%%%%%%%%%%%%%%%%%%%%%%%%%%%%%%%%%%%%%%%%%%%%
\begin{proof}
By Lemma~\ref{l:hd-submul}, it suffices to show that
the left-hand side of~\eqref{e:hd-fup-discrete} is $<1$ for some value of $k$.
We argue by contradiction, assuming that this left-hand side is equal to~1.
Similarly to Lemma~\ref{l:improve-0}, there then exists nonzero $u\in\mathbb C^{N\times N}$
such that 
\begin{equation}
  \label{e:hds-1-0}
\supp u\subset \mathcal C_{k,\mathcal B},\quad
\supp (\mathcal F_{N\times N}u)\subset \mathcal C_{k,\mathcal A}.
\end{equation}
For every $\ell_2\in\mathbb Z_N$, define $u_{\ell_2}\in\mathbb C^N$
by $u_{\ell_2}(\ell_1)=u(\ell_1,\ell_2)$. Denote by
$\mathcal D\subset\mathbb Z_N$ the set of all numbers of the form
$d_0+d_1M+\dots+d_{k-1}M^{k-1}$ where $d_0,d_1,\dots,d_{k-1}\neq s$.
Then by condition~(1) the projection of $\supp \mathcal F_{N\times N}u$ onto the first coordinate lies inside $\mathcal D$. Writing
$$
\mathcal F_Nu_{\ell_2}(j_1)={1\over\sqrt N}\sum_{j_2=0}^{N-1}\exp\Big({2\pi ij_2\ell_2\over N}\Big)\mathcal F_{N\times N}u(j_1,j_2)
$$
we see that
$$
\supp(\mathcal F_Nu_{\ell_2})\subset \mathcal D\quad\text{for all}\quad \ell_2\in\mathbb Z_N.
$$
In particular, $\supp(\mathcal F_Nu_{\ell_2})$ does not intersect
some interval of length $M^{k-1}$. Arguing as in the proof of Lemma~\ref{l:improve-0} we see
that
\begin{equation}
  \label{e:hds-1-1}
\text{for each }\ell_2\in\mathbb Z_N,\quad
\text{either }\supp u_{\ell_2}=\emptyset\quad
\text{or }\#(\supp u_{\ell_2})>M^{k-1}.
\end{equation}
On the other hand, by condition~(2) the support of each $u_{\ell_2}$
has at most $(M-1)^k$ elements. Choosing $k$ large enough so that
$(M-1)^k\leq M^{k-1}$ we get that $u\equiv 0$, giving a contradiction.
\end{proof}
%%%%%%%%%%%%%%%%%%%%%%%%%%%%%%%%%%%%%%%%%%%%%%%%%%%%%%%%%%%%%%%%%%%%%%%%%%%%%%%%
The second condition can be viewed as a special case of the work of Han--Schlag~\cite{HanSchlag}
presented in~\S\ref{s:hd-cont}, letting $\mathcal B$ be any alphabet not equal to the entire
$\mathbb Z_M^2$ and requiring that the complement of $\mathcal A$ contain both a horizontal and a vertical line:
%%%%%%%%%%%%%%%%%%%%%%%%%%%%%%%%%%%%%%%%%%%%%%%%%%%%%%%%%%%%%%%%%%%%%%%%%%%%%%%%
\begin{prop}
  \label{l:hds-2}
Assume that (see Figure~\ref{f:hds-2}):
\begin{enumerate}
\item there exist $s,t\in\mathbb Z_M$ such that $(s,j)\notin\mathcal A$
and $(j,t)\notin\mathcal A$ for all~$j$, and
\item $\mathcal B$ is not equal to the entire $\mathbb Z_M^2$.
\end{enumerate}
Then~\eqref{e:hd-fup-discrete} holds with some $\beta>0$.
\end{prop}
%%%%%%%%%%%%%%%%%%%%%%%%%%%%%%%%%%%%%%%%%%%%%%%%%%%%%%%%%%%%%%%%%%%%%%%%%%%%%%%%
\begin{figure}
\includegraphics[width=5cm]{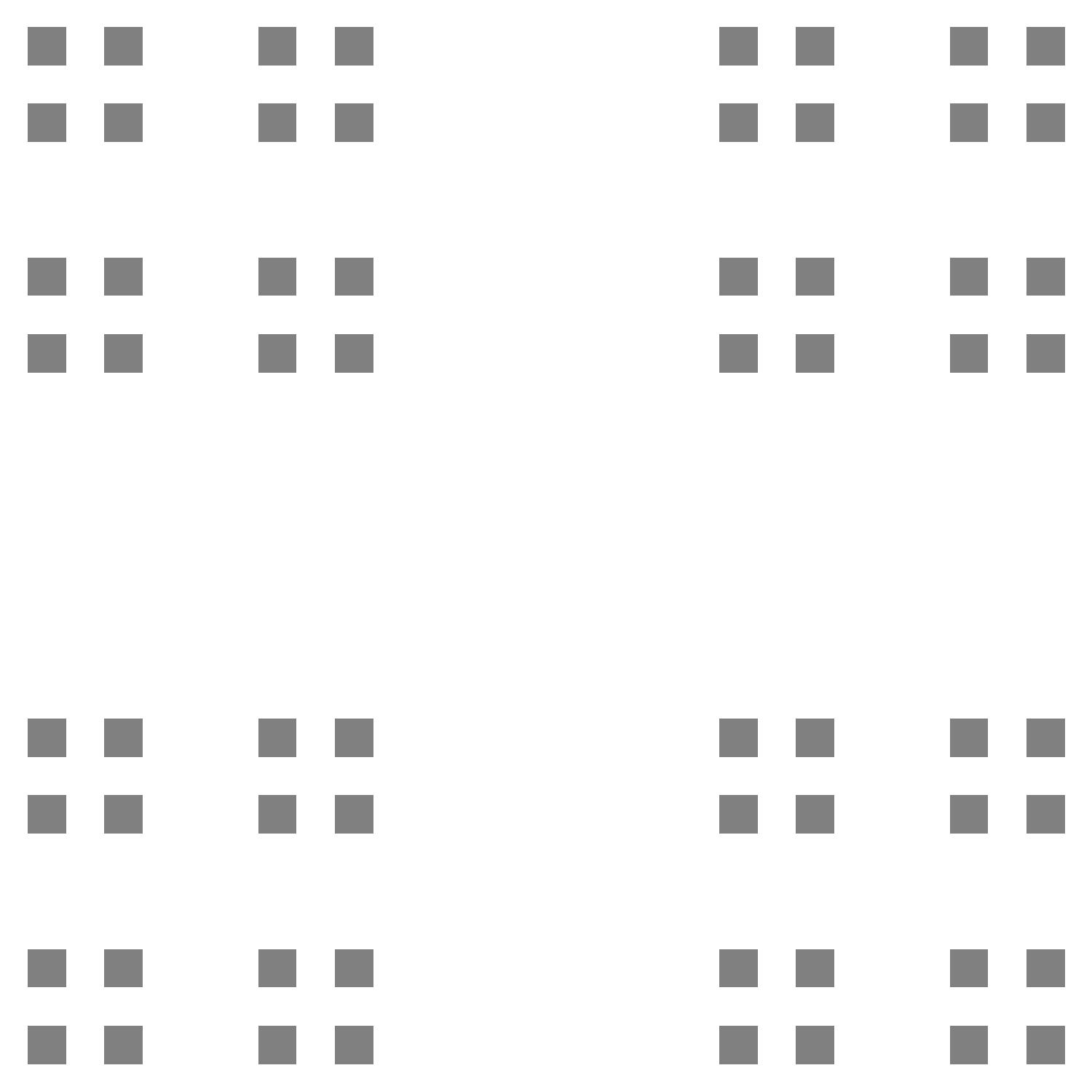}
\qquad\qquad
\includegraphics[width=5cm]{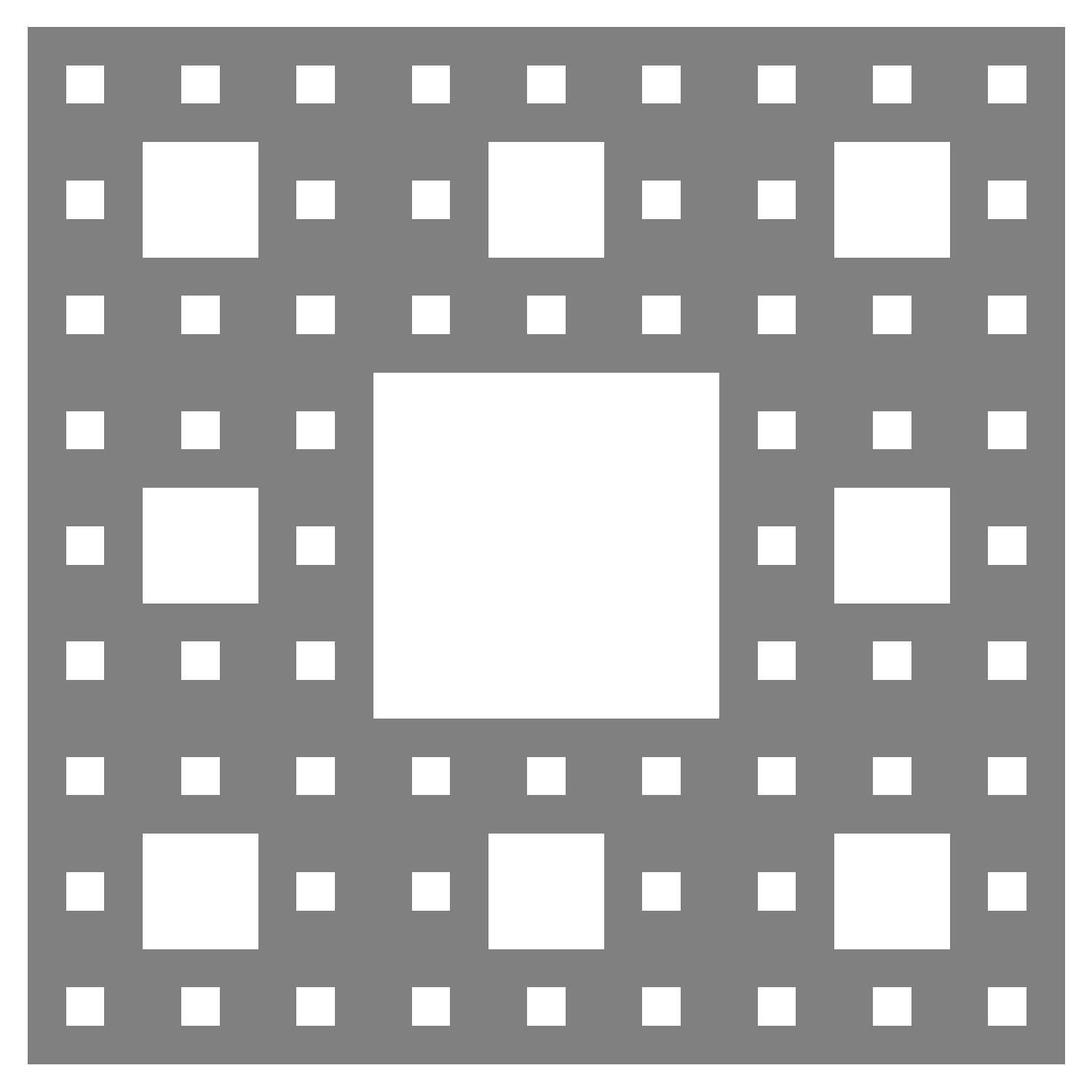}
\hbox to\hsize{\hss $\mathcal C_{k,\mathcal A}$ \hss\qquad\qquad $\mathcal C_{k,\mathcal B}$ \hss}
\caption{An example of Cantor sets with $M=3$, $k=3$ satisfying the conditions
of Proposition~\ref{l:hds-2}.}
\label{f:hds-2}
\end{figure}
%%%%%%%%%%%%%%%%%%%%%%%%%%%%%%%%%%%%%%%%%%%%%%%%%%%%%%%%%%%%%%%%%%%%%%%%%%%%%%%%
\begin{proof}
We argue similarly to the proof of Proposition~\ref{l:hds-1}, assuming the
existence of nonzero $u$ satisfying~\eqref{e:hds-1-0}. By the first part of condition~(1) we see that
$\supp u$ satisfies~\eqref{e:hds-1-1};
that is, the intersection of $\supp u$ with each horizontal line is either empty
or contains $>M^{k-1}$ points.
We fix $k_0\geq 0$ such that
\begin{equation}
  \label{e:hds-2-0}
\Big(1-{1\over M}\Big)^{k_0}\leq {1\over M}
\end{equation}
and take $k\geq k_0$.
By condition~(2) there exists $(s',t')\in\mathbb Z_M^2\setminus \mathcal B$.
For $b=b_0+b_1M+\dots+b_{k-1}M^{k-1}$, where $b_0,\dots,b_{k-1}\in \mathbb Z_M$,
we say that $b$ is \emph{thin} if at most $k_0$ digits $b_0,\dots,b_{k-1}$
are equal to~$t'$, and \emph{thick} otherwise.

If $b$ is thick, then the set $\{a\in\mathbb Z_N\mid (a,b)\in \supp u\}$
contains at most $(M-1)^{k_0}M^{k-k_0}$ points. Indeed, writing
$b$ as above and $a=a_0+a_1M+\dots +a_{k-1}M^{k-1}$, if
$(a,b)\in\supp u\subset\mathcal C_{k,\mathcal B}$,
then for each $r=0,\dots,k-1$ we have
$(a_r,b_r)\neq (s',t')$. Since $b$ is thick,
there exist $r_1<r_2<\dots<r_{k_0}$ such that $b_{r_1}=\dots=b_{r_{k_0}}=t'$.
Then $a$ has to satisfy the conditions
$a_{r_1},\dots,a_{r_{k_0}}\neq s'$, and the number of such points~$a$
is equal to $(M-1)^{k_0}M^{k-k_0}$.

By~\eqref{e:hds-2-0} we have $(M-1)^{k_0}M^{k-k_0}\leq M^{k-1}$. Then by~\eqref{e:hds-1-1}
we see that for each thick $b$, the intersection of $\supp u$ with the
horizontal line $\{(a,b)\mid a\in\mathbb Z_N\}$ is empty.
Thus the projection $\pi_2(\supp u)$ of $\supp u$ onto the second coordinate is contained
inside the set of all thin elements of $\mathbb Z_N$, giving
\begin{equation}
  \label{e:hds-2-1}
\#(\pi_2(\supp u))\leq \binom{k}{k_0}M^{k_0}(M-1)^{k-k_0}.
\end{equation}
On the other hand, using the second part of condition~(1) similarly to~\eqref{e:hds-1-1}
we see that the intersection of $\supp u$ with each vertical line is either
empty or contains $>M^{k-1}$ points. Choose $k$ large enough
(depending on $k_0$) so
that the left-hand side of~\eqref{e:hds-2-1} is $\leq M^{k-1}$.
Then we see that $u\equiv 0$, giving a contradiction.
\end{proof}
%%%%%%%%%%%%%%%%%%%%%%%%%%%%%%%%%%%%%%%%%%%%%%%%%%%%%%%%%%%%%%%%%%%%%%%%%%%%%%%%

%%%%%%%%%%%%%%%%%%%%%%%%%%%%%%%%%%%%%%%%%%%%%%%%%%%%%%%%%%%%%%%%%%%%%%%%%%%%%%%%
%%%%%%%%%%%%%%%%%%%%%%%%%%%%%%%%%%%%%%%%%%%%%%%%%%%%%%%%%%%%%%%%%%%%%%%%%%%%%%%%
\medskip\noindent\textbf{Acknowledgements.}
The author was supported by the NSF CAREER grant DMS-1749858
and a Sloan Research Fellowship.
Part of this article originated as lecture notes for the minicourse
on fractal uncertainty principle at the Third Symposium on Scattering and Spectral Theory in Florianopolis, Brazil, July 2017, and another part as lecture notes for
the Emerging Topics Workshop on
quantum chaos and fractal uncertainty principle at the Institute for Advanced Study,
Princeton, October 2017.
The numerics used to plot Figures~\ref{f:sch1}, \ref{f:3f}, and~\ref{f:ae} were originally developed for an undergraduate project at MIT joint with Arjun Khandelwal during the 2015--2016 academic year.
The author thanks the anonymous referee for a careful reading of the manuscript
and many suggestions to improve the presentation.

%%%%%%%%%%%%%%%%%%%%%%%%%%%%%%%%%%%%%%%%%%%%%%%%%%%%%%%%%%%%%%%%%%%%%%%%%%%%%%%%

\end{document}